\documentclass[12pt,oneside]{amsart}

\usepackage{geometry}                
\usepackage{graphicx}
\usepackage{amssymb}
\usepackage{color}
\usepackage{bbm, dsfont}
\usepackage[hidelinks]{hyperref}

\hypersetup{
	colorlinks=false,
	pdfborder={0 0 0},
	pdfborderstyle={/S/U/W 0},
}

\numberwithin{equation}{section}

\usepackage{verbatim}

\usepackage[dvipsnames]{xcolor}

\usepackage{ulem}
\newtheorem{theorem}{Theorem}[section]
\newtheorem{lemma}[theorem]{Lemma}
\newtheorem{corollary}[theorem]{Corollary}
\newtheorem{proposition}[theorem]{Proposition}

\theoremstyle{definition}
\newtheorem{definition}[theorem]{Definition}

\newtheorem{conjecture}{Conjecture}

\theoremstyle{remark}
\newtheorem{remark}[theorem]{Remark}
\newtheorem*{remark*}{Note}

\numberwithin{equation}{section}

\usepackage{amsmath}
\usepackage{mathrsfs}

\usepackage{verbatim}
\usepackage{amsmath,amssymb,amsthm}           
\usepackage{amssymb}            
\usepackage{amsfonts}            
\usepackage{mathrsfs}          
\usepackage{amsthm}
\usepackage{algorithm}  
\usepackage{algorithmicx}  
\usepackage{algpseudocode}  
\usepackage{mathtools}
\usepackage{commath}
\usepackage{physics}
\usepackage{bm}
\usepackage{graphicx}

\usepackage{float}
\usepackage{listings}
\usepackage{subfigure}
\usepackage{multirow}
\usepackage{color}
\usepackage{bbm}
\usepackage[export]{adjustbox}
\usepackage{enumerate}
\usepackage{bbm}

\usepackage{amsmath}
\usepackage{mathrsfs}

\newcommand{\RNum}[1]{\uppercase\expandafter{\romannumeral #1\relax}}

\newcommand{\specificthanks}[1]{\@fnsymbol{#1}}

\DeclareFontFamily{OML}{rsfs}{\skewchar\font'177}
\DeclareFontShape{OML}{rsfs}{m}{n}{ <5> <6> rsfs5 <7> <8> <9>
	rsfs7 <10> <10.95> <12> <14.4> <17.28> <20.74> <24.88> rsfs10 }{}
\DeclareMathAlphabet{\mathfs}{OML}{rsfs}{m}{n}

\newcounter{cnstcnt}
\newcommand{\cl}{%
	\refstepcounter{cnstcnt}%
	\ensuremath{c_{\thecnstcnt}}}
\newcommand{\cref}[1]{\ensuremath{c_{\ref*{#1}}}}

\newcounter{newcnstcnt}
\newcommand{\Cl}{%
	\refstepcounter{newcnstcnt}%
	\ensuremath{C_{\thenewcnstcnt}}}
\newcommand{\Cref}[1]{\ensuremath{C_{\ref*{#1}}}}

\newcommand{\overbar}[1]{\mkern 1.5mu\overline{\mkern-1.5mu#1\mkern-1.5mu}\mkern 1.5mu}

\DeclareFontFamily{U}{mathx}{}
\DeclareFontShape{U}{mathx}{m}{n}{<-> mathx10}{}
\DeclareSymbolFont{mathx}{U}{mathx}{m}{n}
\DeclareMathAccent{\widehat}{0}{mathx}{"70}
\DeclareMathAccent{\widecheck}{0}{mathx}{"71}

\begin{document}

	\title{Incipient infinite clusters and volume growth for Gaussian free fields and loop soups on metric graphs}
	
		\author{Zhenhao Cai$^1$}
		\address[Zhenhao Cai]{Faculty of Mathematics and Computer Science, Weizmann Institute of Science}
		\email{zhenhao.cai@weizmann.ac.il}
		\thanks{$^1$Faculty of Mathematics and Computer Science, Weizmann Institute of Science}

		\author{Jian Ding$^2$}
		\address[Jian Ding]{School of Mathematical Sciences, Peking University}
		\email{dingjian@math.pku.edu.cn}
		\thanks{$^2$School of Mathematical Sciences, Peking University. Partially supported by NSFC Key Program Project No.12231002 and by New Cornerstone Science Foundation through the XPLORER PRIZE}
	\maketitle
	%
	%
	
	 	\begin{abstract}
In this paper, we establish the existence and equivalence of four types of incipient infinite clusters (IICs) for the critical Gaussian free field (GFF) level-set and the critical loop soup on the metric graph $\widetilde{\mathbb{Z}}^d$ for all $d\ge 3$ except the critical dimension $d=6$. These IICs are defined as four limiting conditional probabilities, involving different conditionings and various ways of taking limits:
\begin{enumerate}

	\item  conditioned on $\{\bm{0}\xleftrightarrow{} \partial B(N)\}$	at criticality (where $\bm{0}$ is the origin of $\mathbb{Z}^d$, and $\partial B(N)$ is the boundary of the box $B(N)$ centered at $\bm{0}$ with side length $2N$), and letting $N\to \infty$;

	\item conditioned on $\{\bm{0}\xleftrightarrow{} \infty\}$	at super-criticality, and letting the parameter tend to the critical threshold;

	\item conditioned on $\{\bm{0}\xleftrightarrow{} x\}$ at criticality (where $x\in  \mathbb{Z}^d$), and letting $x\to \infty$;

	\item conditioned on the event that the capacity of the critical cluster containing $\bm{0}$ exceeds $T$, and letting $T\to \infty$.

\end{enumerate}
Our proof employs a robust framework of Basu and Sapozhinikov (2017) for constructing IICs as in (1) and (2) for Bernoulli percolation on $\mathbb{Z}^d$ in low dimensions (i.e., $3\le  d\le  5$), where a key hypothesis on the quasi-multiplicativity is proved in our companion paper.

We further show that conditioned on $\{\bm{0}\xleftrightarrow{} \partial B(N)\}$, the volume of the critical cluster containing $\bm{0}$ within $B(M)$ is typically of order $M^{(\frac{d}{2}+1)\land 4}$, as long as $N\gg M$. This phenomenon indicates that the critical cluster of the GFF or the loop soup exhibits self-similarity, which supports Werner's conjecture (2016) that such cluster has a scaling limit. Moreover, the exponent of $M^{(\frac{d}{2}+1)\land 4}$ matches the conjectured fractal dimension of the scaling limit proposed by Werner (2016).

	 	\end{abstract}

\section{Introduction}\label{section_intro}

Our main purpose in this paper is to study the incipient infinite cluster (IIC) of two related models on the metric graph $\widetilde{\mathbb{Z}}^d$: the Gaussian free field (GFF) and the loop soup. For clarity, we first review their definitions. For each unordered pair of adjacent points $x\sim y$ (i.e., $|x-y|_1=1$, where $|\cdot|_1$ is the $L^{1}$-norm) in the $d$-dimensional integer lattice $\mathbb{Z}^d$, consider a compact interval $I_{\{x,y\}}$ of length $d$, whose endpoints are identical to $x$ and $y$ respectively. The metric graph of $\mathbb{Z}^d$ is defined as the union of all these intervals, i.e., $\widetilde{\mathbb{Z}}^d:=\cup_{x\sim y\in \mathbb{Z}^d}I_{\{x,y\}}$. We assume $d \ge 3$ throughout this paper. The GFF on $\widetilde{\mathbb{Z}}^d$, denoted by $\{\widetilde{\phi}_v\}_{v\in \widetilde{\mathbb{Z}}^d}$, can be constructed in the following two steps. 
\begin{enumerate}
	\item  Let $\{\phi_x\}_{x\in \mathbb{Z}^d}$ be a discrete GFF on $\mathbb{Z}^d$,  i.e., a family of mean-zero Gaussian random variables whose covariance is given by
	\begin{equation}
		\mathbb{E}[\phi_x\phi_y]=G(x,y),\ \ \forall x,y\in \mathbb{Z}^d,
	\end{equation}
	where the Green's function $G(x,y)$ represents the expected number of visits to $y$ by a simple random walk starting from $x$.

	\item For each interval $I_{\{x,y\}}\subset \widetilde{\mathbb{Z}}^d$, the values of $\widetilde{\phi}_v$ for $v\in I_{\{x,y\}}$ are distributed as an independent Brownian bridge on $I_{\{x,y\}}$ with boundary values $\phi_x$ at $x$ and $\phi_y$ at $y$, generated by a Brownian motion with variance $2$ at time $1$.

\end{enumerate}
The other model called loop soup is a random collection of rooted loops (i.e., continuous paths on $\widetilde{\mathbb{Z}}^d$ which start and end at the same point). Precisely, the loop soup of intensity $\alpha>0$ (denoted by $\widetilde{\mathcal{L}}_{\alpha}$) is the Poisson point process with intensity measure $\alpha  \widetilde{\mu}$, where the loop measure $\widetilde{\mu}$ is defined as 
\begin{equation}\label{def_mu}
	 \widetilde{\mu}(\cdot) := \int_{v\in \widetilde{\mathbb{Z}}^d} \mathrm{d}m(v) \int_{0< t< \infty} t^{-1} \widetilde{q}_t(v,v)\widetilde{\mathbb{P}}^t_{v,v}(\cdot) \mathrm{d}t.
\end{equation}
Here $m(\cdot)$ is the Lebesgue measure on $\widetilde{\mathbb{Z}}^d$, $\widetilde{q}_t(v_1, v_2)$ denotes the transition density of the Brownian motion on $\widetilde{\mathbb{Z}}^d$ (whose precise definition is provided in Section \ref{section_notation}) from $v_1$ to  $v_2$ at time $t$ with respect to $m(\cdot)$, and $\widetilde{\mathbb{P}}^t_{v_1,v_2}(\cdot)$ is the law of the Brownian bridge on $\widetilde{\mathbb{Z}}^d$ with duration $t$ and transition density $\frac{\widetilde{q}_s(v_1,\cdot )\widetilde{q}_{t-s}(\cdot,v_2)}{\widetilde{q}_t(v_1, v_2)}$ for $0\le s\le t$. The isomorphism theorem \cite[Proposition 2.1]{lupu2016loop} establishes a powerful coupling between the GFF $\{\widetilde{\phi}_v\}_{v\in \widetilde{\mathbb{Z}}^d}$ and the critical loop soup $\widetilde{\mathcal{L}}_{1/2}$ (the criticality of the intensity $1/2$ for percolation was proved in \cite{chang2024percolation}) such that 
\begin{equation}
	\widehat{\mathcal{L}}_{1/2}^{v}= \tfrac{1}{2}\widetilde{\phi}_{v}^2, \ \ \forall v\in  \widetilde{\mathbb{Z}}^d,
\end{equation}
where $\widehat{\mathcal{L}}_{1/2}^{v}$ denotes the total local time at $v$ of all loops in $\widetilde{\mathcal{L}}_{1/2}$.  An immediate corollary is that the (critical) loop clusters (i.e., maximal connected subgraphs consisting of loops in $\widetilde{\mathcal{L}}_{1/2}$) have the same distribution as the GFF sign clusters (i.e., maximal connected subgraphs on which $\widetilde{\phi}_{\cdot}$ has the same sign). Furthermore, the sign of the GFF values on each sigh cluster is independent of others, and takes ``$+$'' (or ``$-$'') with probability $\frac{1}{2}$.

There has been a series of studies on the percolation of GFF level-sets $\widetilde{E}^{\ge h}:=\big\{v\in \widetilde{\mathbb{Z}}^d:\widetilde{\phi}_v\ge h\big\}$ for $h\in \mathbb{R}$. In \cite{lupu2016loop}, it was proved that for any $h<0$, $\widetilde{E}^{\ge h}$ a.s. percolates (i.e., includes an infinite connected component), while for any $h\ge 0$, $\widetilde{E}^{\ge h}$ a.s. does not percolate. At the critical level $\widetilde{h}_*=0$, the following formula for the two-point function (i.e., the probability that two points are connected by $\widetilde{E}^{\ge 0}$) was established in \cite[Proposition 5.2]{lupu2016loop}: for any $x\neq y\in \mathbb{Z}^d$, 
\begin{equation}\label{two_point}
	\mathbb{P}\big(x\xleftrightarrow{\widetilde{E}^{\ge 0}} y\big)= \pi^{-1}\arcsin\Big(\tfrac{G(x,y)}{\sqrt{G(x,x)G(y,y)}}\Big) \asymp |x-y|^{2-d},\end{equation}
where $A_1\xleftrightarrow{\mathcal{D}} A_2$ represents the event that there exists a path in $\mathcal{D}$ connecting $A_1$ and $A_2$, $f\asymp g$ means $cg\le f\le Cg$ for some constants $c$ and $C$ depending only on $d$, and $|\cdot |$ denotes the Euclidean norm. For brevity, we denote ``$\xleftrightarrow{\widetilde{E}^{\ge h}}$'' by ``$\xleftrightarrow{\ge h}$''. As a corollary of the absence of percolation at $\widetilde{h}_*=0$, the one-arm probability $\theta_d(N):=\mathbb{P}\big(\bm{0}\xleftrightarrow{\ge 0} \partial B(N) \big)$ converges to $0$ as $N\to \infty$, where $\bm{0}$ is the origin of $\mathbb{Z}^d$, $B(N):=[-N,N]^d\cap \mathbb{Z}^d$, and $\partial A:= \{x\in A: \exists y\in \mathbb{Z}^d\setminus A\ \text{such that}\ y\sim x\}$. Since then, the decay rate of $\theta_d(N)$ became a subject of extensive study due to its strong relation to the geometry of the critical GFF level-set $\widetilde{E}^{\ge 0}$. Since our paper focuses on results for $d\neq 6$, we will only provide precise mathematical statements for $d\neq 6$ when reviewing previous work (for $d=6$, it is usually the case that exponents were computed). Through a series of works \cite{ding2020percolation, cai2024high, drewitz2023arm,drewitz2024critical,cai2024one} (see also \cite{drewitz2023critical} for extensions to more general transient graphs), the exact order of $\theta_d(N)$ for all $d\ge 3$ except the critical dimension $d=6$, as well as the exponent of $\theta_6(N)$, has been determined. Specifically, it has been established that (see \cite{cai2024high} for $d>6$; see \cite{cai2024one} for $3\le d< 6$ and also see a concurrent work \cite{drewitz2024critical} for $d = 3$)
\begin{align}
	&\text{when}\ 3\le d<6,\ \  \theta_d(N) \asymp N^{-\frac{d}{2}+1},\label{one_arm_low} \\
	&\text{when}\ d>6,\ \ \ \ \ \ \ \ \theta_d(N) \asymp N^{-2}, \label{one_arm_high}
\end{align}
where $f\lesssim g$ means $f\le Cg$ for some constant $C$ depending only on $d$. Moreover, the crossing probability for an annulus, i.e., $\rho_d(n,N):=\mathbb{P}\big(B(n)\xleftrightarrow{\ge 0} \partial B(N) \big)$, was also studied in \cite{cai2024one} as a natural extension of the one-arm probability $\theta_d(N)$. Precisely, it was proved in \cite[Theorem 1.2]{cai2024one} that 
\begin{align}
	&\text{when}\ 3\le d<6,\  \  \rho_d(n,N) \asymp \big(\frac{n}{N}\big)^{\frac{d}{2}+1};\label{crossing_low} \\
	&\text{when}\ d>6,\ \  \ \ \ \ \ \  \rho_d(n,N) \lesssim  (n^{d-4}N^{-2})\land 1. \label{crossing_high}
\end{align}
Besides the diameter (as characterized by $\theta_d$), the volume of the cluster of $\widetilde{E}^{\ge 0}$ has also been studied. For any $h\in \mathbb{R}$ and $v\in \widetilde{\mathbb{Z}}^d$, let $\mathcal{C}^{\ge h}(v):=\{w\in \widetilde{\mathbb{Z}}^d:w \xleftrightarrow{\ge h} v \}$ denote the cluster of $\widetilde{E}^{\ge h}$ containing $v$. For any $A\subset \mathbb{Z}^d$, the volume of $A$ (denoted by $\mathrm{vol}(A)$) is defined to be the cardinality of $A\cap \mathbb{Z}^d$. For any $d\ge 3$ and $M\ge 1$, we denote $\nu_d(M):=\mathbb{P}\big(
\mathrm{vol}(\mathcal{C}^{\ge 0}(\bm{0}))\ge M \big)$. Referring to \cite{cai2024high,inpreparation,drewitz2024cluster}, it has been established that (see \cite{cai2024high} for $d>6$, and see \cite{inpreparation,drewitz2024cluster} for $3\le d< 6$)
\begin{align}
	&\text{when}\ 3\le d<6,\ \    \ \nu_d(M) \asymp M^{-\frac{d-2}{d+2}};\label{volume_low} \\
	&\text{when}\ d>6,\ \ \ \ \ \  \ \ \ \nu_d(M) \asymp M^{-\frac{1}{2}}.\label{volume_high}
\end{align}

For any $v\in \widetilde{\mathbb{Z}}^d$, let $\mathcal{C}(v)$ denote the loop cluster of $\widetilde{\mathcal{L}}_{1/2}$ containing $v$. By the isomorphism theorem and the symmetry of $\widetilde{\phi}_{\cdot}$, we know that the positive cluster $\mathcal{C}^{\ge 0}(v)$, given it is non-empty, has the same distribution as $\mathcal{C}(v)$. In light of this, we simply refer to both types of clusters as ``critical clusters''. Moreover, since $\theta_d$, $\rho_d$ and $\nu_d$ are exactly half of their analogues for $\mathcal{C}(\bm{0})$, we know that all the aforementioned bounds on $\theta_d$, $\rho_d$ and $\nu_d$ also hold for $\mathcal{C}(\bm{0})$, up to a factor of $2$.

\subsection{Incipient infinite cluster}
This paper is primarily motivated by the following insightful and important conjecture in \cite{werner2021clusters}. For any two functions $f$ and $g$ depending on the dimension $d$, we denote  
\begin{equation}
	(f \boxdot g)(d):= f(d)\cdot \mathbbm{1}_{d\le 6} + g(d)\cdot \mathbbm{1}_{d> 6}. 
\end{equation}
\begin{conjecture}\label{conj1}
	For any $3\le d\le 5$ and $d\ge 7$, the scaling limit of critical clusters exists and has the fractal dimension $(\frac{d}{2}+1)\boxdot 4$.
\end{conjecture}
A natural starting point for investigating the scaling limit is to understand the spatial properties of a critical cluster with a large diameter (note that those with small diameters will vanish in the scaling limit). In particular, to understand the microscopic structure of a large critical cluster, a classical perspective is to analyze the incipient infinite cluster, originally introduced by \cite{kesten1986incipient} in the context of Bernoulli percolation on two-dimensional graphs (e.g., $\mathbb{Z}^2$). Intuitively, the IIC can be thought of as the critical cluster that contains $\bm{0}$ and is conditioned to percolate. However, for a large family of percolation models (including Bernoulli percolation on $\mathbb{Z}^2$, GFF level-sets, loop soups, etc), the critical cluster containing $\bm{0}$ almost surely does not percolate. Consequently, directly imposing a conditioning of ``percolating'' at criticality is not feasible. Referring to \cite{kesten1986incipient}, for Bernoulli percolation there are two types of schemes to achieve this: 
\begin{enumerate}[(1)] 
	\item  Consider the probability measure conditioned on $\bm{0}$ being connected to $\partial B(N)$ at criticality. Then define the IIC as the limiting measure of this probability measure as $N\to \infty$.

	\item  Consider the probability measure conditioned on $\bm{0}$ being contained in an infinite cluster at super-criticality (which happens with positive probability). Then define the IIC as the limiting measure of this probability measure as the parameter decreases to the critical threshold.

\end{enumerate} 
In \cite{kesten1986incipient}, it was proved that these two limiting measures both exist and are equivalent for a large class of two-dimensional graphs (including $\mathbb{Z}^2$, the triangular lattice and the hexagonal lattice) where the critical Bernoulli percolation crosses any rectangular region with at least a positive probability depending only on the aspect ratio of the rectangle. In the context of GFF level-sets, these two types of IICs can be ``defined'' as follows (here the quotation mark indicates that the existence of the limit needs to be justified, and the same applies to (\ref{iic_type3}) and (\ref{iic_type4})):
\begin{equation}\label{iic_type1}
\mathbb{P}_{d,\mathrm{IIC}}^{(\mathrm{1})}(\cdot ):=	 \lim\limits_{N\to \infty} 	\mathbb{P}\big(\cdot\mid \bm{0}\xleftrightarrow{\ge 0}\partial B(N)\big),
\end{equation}
\begin{equation}\label{iic_type2}
\mathbb{P}_{d,\mathrm{IIC}}^{(\mathrm{2})}(\cdot ):=	\lim\limits_{h\uparrow 0} \mathbb{P}\big(\cdot\mid \bm{0}\xleftrightarrow{\ge h}\infty \big),
\end{equation}
where $\bm{0}\xleftrightarrow{\ge h}\infty$ denotes the event that $\bm{0}$ is contained in an infinite cluster of $\widetilde{E}^{\ge h}$.

In \cite{van2004incipient}, the following type of IIC was studied:
\begin{enumerate}[(3)]
	\item  Consider the probability measure conditioned on $\bm{0}$ being connected to a lattice point $x\in \mathbb{Z}^d$ at criticality. Then define the IIC as the limiting measure of this probability measure as $x\to \infty$.
	
\end{enumerate}
The existence of this limiting measure was established in \cite{van2004incipient} for Bernoulli percolation on $\mathbb{Z}^d$ with all sufficiently large $d$ (i.e., $d\ge 19$). This result was subsequently extended to $d\ge 11$ by \cite{fitzner2017mean}. See also \cite{hara1998incipient,hara2000scaling2,hara2000scaling1,heydenreich2014high,heydenreich2014random,kozma2009alexander,10.1214/ECP.v20-3570} for related results. In the setting of GFF level-sets, this type of IIC can be formulated as follows:
\begin{equation}\label{iic_type3}
	\mathbb{P}_{d,\mathrm{IIC}}^{(\mathrm{3})}(\cdot ):=	 \lim\limits_{x\to \infty} 	\mathbb{P}\big(\cdot\mid \bm{0}\xleftrightarrow{\ge 0}x\big).
\end{equation}
In \cite{van2004incipient}, another type of IIC concerning the susceptibility (i.e., the expected volume of the cluster containing $\bm{0}$ at sub-criticality) was also established (we note that the approach in this paper seems insufficient for handling this type).

As a useful tool in the analysis of $\theta_d$ (see \cite{ding2020percolation,drewitz2023arm,drewitz2023critical}), the capacity of the critical cluster has also been studied. Precisely, for any $A\subset \widetilde{\mathbb{Z}}^d$, the capacity of $A$ (denoted by $\mathrm{cap}(A)$) reflects the hitting probability of $A$ for a Brownian motion starting from a distant point (see the precise definition of $\mathrm{cap}(\cdot)$ in Section \ref{section_notation}). Referring to \cite[(3.6) and (3.7)]{drewitz2023critical}, it is known that for all $d\ge 3$, $\mathrm{cap}\big( \mathcal{C}^{\ge 0}(\bm{0}) \big)$ has density  
\begin{equation}
	 \big(2\pi t \sqrt{G(\bm{0},\bm{0}) [t-G(\bm{0},\bm{0})]}  \big)^{-1}\cdot \mathbbm{1}_{t>[G(\bm{0},\bm{0})]^{-1}}
\end{equation}
at $t$. Consequently, for all sufficiently large $T>0$, one has 
 \begin{equation}\label{iic_type4new}
 \mathbb{P}\big( \mathrm{cap}\big( \mathcal{C}^{\ge 0}(\bm{0}) \big)\ge T \big)=  \big(\pi \sqrt{G(\bm{0},\bm{0})} \big)^{-1} T^{-\frac{1}{2}} + O(T^{-\frac{3}{2}}).
\end{equation}
In this paper, we also consider the following formulation of the IIC with respect to the capacity:
\begin{enumerate}[(4)]
	\item Consider the probability measure conditioned on the event that the capacity of the critical cluster containing $\bm{0}$ exceeds $T$. Then define the IIC as the limiting measure of this probability measure as $T\to \infty$.
	\end{enumerate}
 Formally, we define this type of IIC for GFF level-sets as follows:
  \begin{equation}\label{iic_type4}
	\mathbb{P}_{d,\mathrm{IIC}}^{(\mathrm{4})}(\cdot ):=	 \lim\limits_{T\to \infty} 	\mathbb{P}\big(\cdot\mid  \mathrm{cap}\big(\mathcal{C}^{\ge 0}(\bm{0})\big)\ge T\big).
\end{equation}
Note that Type (4) IIC can be regarded as an analogue of Type (1) IIC obtained by replacing the diameter with the capacity. Certainly, it is natural to consider its counterpart for the volume. As shown in later sections, Type (1) IIC is established using powerful tools for analyzing connecting events, while Type (4) IIC relies on the known asymptotic of the decay rate of the capacity (see (\ref{iic_type4new}); note that up-to-constant bounds are insufficient). However, neither of these approaches is currently available for the volume, making it challenging to study the volume-based IIC.



We hereby clarify the precise sense in which the limiting measures presented in (\ref{iic_type1}), (\ref{iic_type2}), (\ref{iic_type3}) and (\ref{iic_type4}) converge. Let $\{\mathsf{A}_i\}_{i \in \mathcal{I}}$ be a family of events measurable with respect to $\{\widetilde{\phi}_v\}_{v\in \widetilde{\mathbb{Z}}^d}$, where the index set $\mathcal{I}$ is a subset of either $\mathbb{R}$ or $\widetilde{\mathbb{Z}}^d$. When referring to ``$\mathbb{P}_{\infty}(\cdot):=\lim\limits_{i\to i_\diamond}\mathbb{P}\big(\cdot \mid \mathsf{A}_i\big)$ exists'' (where $i_{\diamond}$ may be infinity), we mean that for any increasing cylinder event $\mathsf{F}$ (i.e., $\mathsf{F}:=\cap_{j=1}^{k}\{\widetilde{\phi}_{x_i}\ge h_i \}$ for some $k\in \mathbb{N}^+$, $\{x_i\}_{i=1}^{k}\subset \widetilde{\mathbb{Z}}^d$ and $\{h_i\}_{i=1}^{k}\subset \mathbb{R}$), the conditional probability $\mathbb{P}\big(\mathsf{F} \mid \mathsf{A}_i\big)$ converges as $i\to i_{\diamond}$. Given the value of $\mathbb{P}_{\infty}(\mathsf{F})$ for every increasing cylinder event $\mathsf{F}$, the measure $\mathbb{P}_{\infty}(\cdot)$ can be uniquely extended to a probability measure, which defines a random field on $\widetilde{\mathbb{Z}}^d$.

One of the main results in this paper demonstrates that the aforementioned four types of IICs exist and are equivalent for GFF level-sets.

\begin{theorem}\label{thm_iic}
	For any $d\ge 3$ with $d\neq 6$, the limiting measures in (\ref{iic_type1}), (\ref{iic_type2}), (\ref{iic_type3}) and (\ref{iic_type4}) exist and are equivalent. Moreover, under this limiting measure, denoted by $\mathbb{P}_{d,\mathrm{IIC}}(\cdot )$, the incipient infinite cluster $\mathscr{C}^{\ge 0}$ (i.e., the cluster containing $\bm{0}$ with non-negative values) is almost surely infinite and one-ended (i.e., for all $N>0$, $\mathscr{C}^{\ge 0}\setminus B(N)$ contains a unique infinite cluster). 
\end{theorem}

Our proof of Theorem \ref{thm_iic} is mainly inspired by \cite{basu2017kesten}. Precisely, in \cite{basu2017kesten}, a fundamental property called \textit{quasi-multiplicativity} was conjectured to hold for Bernoulli percolation on $\mathbb{Z}^d$ with $3\le d<6$, stating that the connecting probability between two general sets is of the same order as the product of the connecting probabilities between each set and the boundary of an intermediate box. Assuming quasi-multiplicativity holds for all $p\in [p_c,p_c+\delta]$ with some $\delta>0$ (where $p_c$ is the critical threshold), the existence and equivalence of Type (1) and Type (2) IICs were established in \cite{basu2017kesten}. In a companion paper \cite{inpreparation}, we proved quasi-multiplicativity for the critical levet-set $\widetilde{E}^{\ge 0}$ for all $d\ge 3$ except the critical dimension $d=6$, with correction factor $N^{6-d}$ when $d>6$; for $d=6$, similar bounds were also proved, although the upper and lower bounds differ
by a divergent factor of $N^{o(1)}$. Specifically, it was proved in \cite[Theorem 1.1]{inpreparation} that for any $d\ge 3$ with $d\neq 6$, there exist constants $C=C(d),c=c(d)>0$ such that for any $N\ge 1$, $A_1,D_1\subset  \widetilde{B}(cN^{1\boxdot \frac{2}{d-4}})$ and $A_2,D_2\subset [\widetilde{B}(CN^{1\boxdot \frac{d-4}{2}})]^c$,
	\begin{equation}\label{QM_ineq_1}
	\begin{split}
			\mathbb{P}^{D_1\cup D_2}\big( A_1\xleftrightarrow{\ge 0} A_2\big) 
			 \asymp N^{0 \boxdot (6-d)} \mathbb{P}^{D_1}\big( A_1 \xleftrightarrow{\ge 0} \partial B(N) \big) \mathbb{P}^{D_2}\big( A_2 \xleftrightarrow{\ge 0} \partial B(N) \big).	
			 \end{split}
	\end{equation}
	Here $\mathbb{P}^D(\cdot)$ denotes the law of $\{\widetilde{\phi}_v\}_{v\in \widetilde{\mathbb{Z}}^d}$ conditioned on $\cap_{v\in D} \{\widetilde{\phi}_v=0\}$, and $\widetilde{B}(\cdot )$ denotes the box of the metric graph defined by
	\begin{equation}\label{def_continuous_box}
			\widetilde{B}(M):=\bigcup_{y_1\sim y_2\in B(M): \{y_1,y_2\}\cap B(M-1)\neq \emptyset} I_{\{y_1,y_2\}}, \ \ \forall M\ge 1. 
		\end{equation}

Our proof of Theorem \ref{thm_iic} is based on (\ref{QM_ineq_1}) and on the proof framework in \cite{basu2017kesten}. Notably, in this proof we first establish the existence and equivalence of the analogues of (\ref{iic_type1}) and (\ref{iic_type2}) for the loop soup $\widetilde{\mathcal{L}}_{1/2}$, and then extend them to the case in $\widetilde{E}^{\ge 0}$ by using the isomorphism theorem. Moreover, since the distribution of the loop cluster $\mathcal{C}(\bm{0})$ conditioned on $\bm{0}$ being connected to $\partial B(N)$ is identical to that of $\mathcal{C}^{\ge 0}(\bm{0})$ under the same conditioning (by the isomorphism theorem), we know that the IIC of $\widetilde{\mathcal{L}}_{1/2}$ (denoted by $\mathscr{C}$) has the same distribution as $\mathscr{C}^{\ge 0}$.

\textbf{P.S.} The exclusion of $d=6$ in Theorem \ref{thm_iic} arises from the absence of a proof of quasi-multiplicativity for $d=6$ (probably involving a poly-logarithmic correction factor). As noted in \cite[Remark 1.3]{inpreparation}, this issue stems from the lack of an estimate for the order of $\theta_6(N)$. For clarity, in the remainder of this paper, we assume $d\neq 6$ unless otherwise specified. As a supplement, in Remark \ref{remark_extension} we will point out which results can be extended to $d=6$, albeit with unfavorable error terms.

\subsection{Volume growth}

The typical volume of a critical cluster with a large diameter, as a fundamental quantity, is closely related to $\theta_d$ and $\nu_d$. Specifically, for any $d\ge 3$ and $N\ge 1$, we denote $\widehat{\mathbb{P}}_{d,N}(\cdot):= \mathbb{P}(\cdot\mid \bm{0}\xleftrightarrow{\ge 0}\partial B(N))$. By putting (\ref{one_arm_low}), (\ref{one_arm_high}), (\ref{volume_low}) and (\ref{volume_high}) together, we obtain that for any $\lambda>1$,
\begin{equation}
	\begin{split}
		&\widehat{\mathbb{P}}_{d,N}\big( \mathrm{vol}(\mathcal{C}^{\ge 0})\ge \lambda N^{(\frac{d}{2}+1)\boxdot 4} \big) \\
		\le & [\theta_d(N)]^{-1} \nu_d(\lambda N^{(\frac{d}{2}+1)\boxdot 4})\lesssim \lambda^{-(\frac{d-2}{d+2} \boxdot \frac{1}{2})}\overset{\lambda \to \infty}{\to} 0,
	\end{split}
\end{equation}
which implies that the typical volume is at most of order $N^{(\frac{d}{2}+1)\boxdot 4}$. The following result shows that this is indeed the exact order of the typical volume. For brevity, we denote $\mathcal{V}_{M}^{\ge 0}:=\mathrm{vol}(\mathcal{C}^{\ge 0}(\bm{0})\cap B(M))$ for $M\ge 1$.

\begin{theorem}\label{thm_1.2}
For any $d\ge 3$ with $d\neq 6$ and for any $\epsilon>0$, there exist constants $\Cl\label{const_typical_volume_1}(d), \Cl\label{const_typical_volume_2}(d,\epsilon), \cl\label{const_typical_volume_3}(d,\epsilon)>0$ such that for any $M\ge 1$ and $N\ge \Cref{const_typical_volume_1}M$, 
\begin{equation}\label{intro_1.20}
	\widehat{\mathbb{P}}_{d,N}\big(\cref{const_typical_volume_3}M^{(\frac{d}{2}+1)\boxdot 4}  \le  \mathcal{V}_{M}^{\ge 0} \le \Cref{const_typical_volume_2}M^{(\frac{d}{2}+1)\boxdot 4} \big)\ge 1-\epsilon.
\end{equation}
\end{theorem}

Theorem \ref{thm_1.2} establishes a tightness property for the volume of the critical cluster under normalizations at different scales. This phenomenon suggests that the critical cluster should exhibit self-similarity (i.e., a similarity between the macroscopic and microscopic structures), which further supports Conjecture 1 on the scaling limit of the critical cluster.

 Provided with Theorem \ref{thm_1.2}, it remains unclear whether the scaling limit of critical clusters (assuming it exists) is deterministic or stochastic. The next result demonstrates an anti-concentration property for the volume $\mathcal{V}_{M}^{\ge 0}$ conditioned on $\{\bm{0}\xleftrightarrow{\ge 0} \partial B(N)\}$, thereby implying that the scaling limit (if exists) is stochastic.

 \begin{theorem}\label{thm_1.3}
	For any $d\ge 3$ with $d\neq 6$ and for any $\lambda\ge 1$, there exist constants $\Cl\label{const_volume_sharp1}(d),\cl\label{const_volume_sharp2}(d,\lambda)>0$ such that for any $M\ge 1$ and  $N\ge \Cref{const_volume_sharp1}M$, 
		\begin{equation}
		\widehat{\mathbb{P}}_{d,N}\big(\mathcal{V}_{M}^{\ge 0}\ge \lambda M^{(\frac{d}{2}+1)\boxdot 4}  \big)  \ge \cref{const_volume_sharp2}.
	\end{equation}
\end{theorem}

 During the proof of Theorems \ref{thm_1.2} and \ref{thm_1.3}, we obtain the following estimate for the normalized two-point function, which is interesting on its own right. It gives the probability of connecting two points conditioned on one of them being contained in a large critical cluster.

\begin{theorem}\label{thm_1.4}
	For any $d\ge 3$ with $d\neq 6$, there exist constants $\Cl\label{const_iic_two_point1}(d),\Cl\label{const_iic_two_point2}(d), \cl\label{const_iic_two_point3}(d)>0$ such that for any $M\ge 1$, $y\in \partial B(M)$ and $N\ge \Cref{const_iic_two_point1}M$, 
\begin{equation}
\cref{const_iic_two_point3} M^{-[(\frac{d}{2}-1)\boxdot (d-4)]} \le 	\widehat{\mathbb{P}}_{d,N}\big( \bm{0}\xleftrightarrow{\ge 0} y \big) \le \Cref{const_iic_two_point2} M^{-[(\frac{d}{2}-1)\boxdot  (d-4)]}. 
\end{equation}
\end{theorem}

 By taking the limits in Theorems \ref{thm_1.2}, \ref{thm_1.3} and \ref{thm_1.4} as $N\to \infty$, we immediately derive the following properties of the IIC $\mathscr{C}^{\ge 0}$ established in Theorem \ref{thm_iic}. For simplicity, we denote $\mathscr{V}_{M}^{\ge 0}:=\mathrm{vol}(\mathscr{C}^{\ge 0}\cap B(M))$ for $M\ge 1$.

 \begin{corollary}
 	For any $d\ge 3$ with $d\neq 6$, the following properties hold. 
 	\begin{enumerate}
 		\item Recall $\cref{const_typical_volume_3}(d,\epsilon),\Cref{const_typical_volume_2}(d,\epsilon)$ in Theorem \ref{thm_1.2}. For any $\epsilon>0$ and $M\ge 1$, 
 		\begin{equation}\label{use1.25}
	\mathbb{P}_{d,\mathrm{IIC}}\big(\cref{const_typical_volume_3}M^{(\frac{d}{2}+1)\boxdot  4}  \le  \mathscr{V}_{M}^{\ge 0} \le \Cref{const_typical_volume_2}M^{(\frac{d}{2}+1)\boxdot   4} \big)\ge 1-\epsilon.
\end{equation}

 		\item Recall $\cref{const_volume_sharp2}(d,\lambda)$ in Theorem \ref{thm_1.3}. For any $\lambda\ge 1$ and $M\ge 1$,  
 		\begin{equation}
			\mathbb{P}_{d,\mathrm{IIC}}\big(\mathscr{V}_{M}^{\ge 0}\ge \lambda M^{(\frac{d}{2}+1)\boxdot  4}  \big) \ge \cref{const_volume_sharp2}.
	\end{equation}

 		\item  Recall $\cref{const_iic_two_point3}(d), \Cref{const_iic_two_point2}(d)$ in Theorem \ref{thm_1.4}. For any $M\ge 1$ and $y\in \partial B(M)$, 
 		\begin{equation}\label{use1.27}
\cref{const_iic_two_point3} M^{-[(\frac{d}{2}-1)\boxdot  (d-4)]} \le 	\mathbb{P}_{d,\mathrm{IIC}}\big( \bm{0}\xleftrightarrow{\ge 0} y \big) \le \Cref{const_iic_two_point2} M^{-[(\frac{d}{2}-1)\boxdot  (d-4)]}. 
\end{equation}
 		
 	\end{enumerate}
 \end{corollary}

\begin{remark}[extension to the critical dimension $d=6$]\label{remark_extension}
	Referring to Sections \ref{section_iic_twopoint} and \ref{section_order_typical_volume}, the proofs of Theorems \ref{thm_1.2} and \ref{thm_1.4} for low dimensions (i.e., $3\le d\le 5$) can be applied to the critical dimension $d=6$, yielding the following results: for $d=6$, there exist constants $C,C',C''>0$ such that for all sufficiently large $N\ge CM$, 
	\begin{equation}
		\widehat{\mathbb{P}}_{d,N}\big( M^{4-C'\varsigma(M)}  \le  \mathcal{V}_{M}^{\ge 0} \le  M^{4+C''\varsigma(M)} \big)\ge 1- o_M(1),
	\end{equation}
	\begin{equation}
	  M^{-4-C'  \varsigma(M)} \le 	\widehat{\mathbb{P}}_{d,N}\big( \bm{0}\xleftrightarrow{\ge 0} y \big) \le M^{-4+C'' \varsigma(M)},	\end{equation}	
	  where $\varsigma(M):=  \tfrac{\ln\ln(M)}{\ln^{1/2}(M)}\ll 1$. 
\end{remark}

\textbf{Related results on IICs and scaling limits.} For Type (3) IIC of Bernoulli percolation in high dimensions, \cite{kozma2009alexander} studied the random walk thereon, and computed its spectral dimension as well as its expected diameter and range, thereby confirming the mean-field behavior conjectured in \cite{alexander1982density}. Building on \cite{kozma2009alexander} and extending the techniques in \cite{cai2024high}, it was proved in \cite{ganguly2024ant} that for $d>20$, any sub-sequential limiting measure of (\ref{iic_type1}) or (\ref{iic_type3}) exhibits the same properties as shown in \cite{kozma2009alexander} (these properties are now confirmed to hold for the IIC, in light of Theorem \ref{thm_iic}). These results were recently extended to the entire mean-field region $d>6$ in \cite{ganguly2024critical}. Through an insightful analysis of a gauge-twisted Gaussian free field, \cite{lupu2022equivalence} proposed a rigorous conjecture called ``intensity doubling conjecture'' on the structure of the scaling limit of the critical cluster in high dimensions (i.e., $d>6$). This conjecture mainly suggests that in high dimensions the backbone (i.e., the homologically non-trivial part) of the scaling limit is distributed as a loop soup with intensity $1$, doubling the intensity of the loop soup $\widetilde{\mathcal{L}}_{1/2}$ before scaling. Notably, in the context of the FK-Ising model on $\mathbb{Z}^d$ with $d\ge 3$, the existence of Type (3) IIC (analogous to (\ref{iic_type3})) was established in \cite{panis2024incipient}.

\textbf{Statements about constants.} We use the letters $C$ and $c$ to represent constants that may change depending on the context. Moreover, constants with numerical labels like $C_1,C_2,c_1,c_2,...$ are fixed throughout the paper. The uppercase letter $C$ (with or without superscripts or subscripts) denotes large constants, while the lowercase $c$ represents small constants. Unless otherwise stated, a constant can only depend on the dimension $d$. If a constant depends on parameters other than $d$, all these parameters will be specified in parentheses.

\subsection{Organization of the paper}

In Section \ref{section_notation}, we fix some necessary notations and review some relevant results. The proof of Theorem \ref{thm_iic} is presented in Section \ref{section_exi_IIC}. Subsequently, in Section \ref{section_iic_twopoint} we establish Theorem \ref{thm_1.4}. Finally, Theorems \ref{thm_1.2} and \ref{thm_1.3} are confirmed in Section \ref{section_order_typical_volume}.

\section{Preliminaries}\label{section_notation}

In this section, we gather some necessary notations and useful results to simplify the exposition in the later proofs.

\subsection{Brownian motion on $\widetilde{\mathbb{Z}}^d$}

We now define the Brownian motion on $\widetilde{\mathbb{Z}}^d$. Within each interval $I_e$, the Brownian motion $\widetilde{S}_{\cdot}$ behaves as a standard one-dimensional Brownian motion. When reaching any lattice point $x \in \widetilde{\mathbb{Z}}^d$, the process $\widetilde{S}_{\cdot}$ uniformly chooses one of the neighboring intervals $\{I_{\{x, y\}}\}_{y \sim x}$ and initiates a Brownian excursion from $x$  along that interval. When this excursion arrives at an adjacent point $y\sim x$, the process resumes at $y$ and continues in the same manner. Clearly, the projection of $\big\{\widetilde{S}_t\big\}_{t \ge 0}$ onto $\mathbb{Z}^d$ is distributed as a continuous-time simple random walk on $\mathbb{Z}^d$. For any $v\in \widetilde{\mathbb{Z}}^d$, we denote by $\widetilde{\mathbb{P}}_v$ the law of $\big\{\widetilde{S}_t\big\}_{t \ge 0}$ with the starting point $v$, and we denote by $ \widetilde{\mathbb{E}}_v$ the expectation under $\widetilde{\mathbb{P}}_v$.

\textbf{Hitting times.} For any $D \subset \widetilde{\mathbb{Z}}^d$, the hitting time $\tau_D$ is the first time when $\{\widetilde{S}_t\}_{t \ge 0}$ enters $D$, i.e., $\tau_D := \inf\{ t \ge 0 :\widetilde{S}_t \in D \}$ (where we set $\inf\emptyset = \infty$ for completeness). For $v \in \widetilde{\mathbb{Z}}^d$, we abbreviate $\tau_{\{v\}}$ as $\tau_v$ for notational convenience.

\textbf{P.S.:} Unless otherwise specified, we use the symbol $D$ (possibly with superscripts or subscripts) to represent a subset of $\widetilde{\mathbb{Z}}^d$ consisting of finitely many compact connected components.

\textbf{Green's function.} For any $D\subset \widetilde{\mathbb{Z}}^d$, the Green's function for $D$ is defined as
\begin{equation}
	\widetilde{G}_D(v,w):=\int_{0}^{\infty} \Big\{\widetilde{q}_t(v,w) - \widetilde{\mathbb{E}}_v \big[ \widetilde{q}_{t-\tau_{D}}(\widetilde{S}_{\tau_D},w)\cdot \mathbbm{1}_{\tau_D<t}  \big]\Big\}dt,\ \ \forall v, w\in \widetilde{\mathbb{Z}}^d.
	\end{equation}
	It is well-known that (see \cite[Section 3]{lupu2016loop}) $\widetilde{G}_D(\cdot, \cdot)$ is finite, symmetric and continuous. Note that $\widetilde{G}_D(v,w):=0$ if $\{w,v\}\cap D\neq \emptyset$, and that $\widetilde{G}_D(v, w)$ is decreasing in $D$. When $D=\emptyset$, we omit the subscript and denote $\widetilde{G}(\cdot ,\cdot ):=\widetilde{G}_{\emptyset}(\cdot ,\cdot )$. Moreover, 
	\begin{equation}\label{bound_green}
	\widetilde{G}(v ,w) \asymp (\|v-w\|+1)^{2-d}, \ \ \forall v,w\in \widetilde{\mathbb{Z}}^d,
\end{equation}
	where $\|v -w\|$ denotes the graph distance between $v$ and $w$ on $\widetilde{\mathbb{Z}}^d$.

\textbf{Boundary excursion kernel.} For any $D\subset \widetilde{\mathbb{Z}}^d$, we denote by $\mathbb{K}_D(\cdot,\cdot )$ the boundary excursion kernel for $D$. I.e., 
\begin{equation}\label{def_KD}
	\mathbb{K}_D(v,w) := \lim\limits_{\epsilon \downarrow  0} (2\epsilon)^{-1}\sum\nolimits_{v'\in \widetilde{\mathbb{Z}}^d: \|v'-v\|=\epsilon} \widetilde{\mathbb{P}}_{v'} (\tau_{D}=\tau_{w}<\infty), \ \ \forall v,w\in \widetilde{\partial} D.
\end{equation}
Here the boundary of $D$ is defined as 
$$\widetilde{\partial} D:=\{v\in \widetilde{\mathbb{Z}}^d: \inf\nolimits_{w\in D}\|w-v\|=\inf\nolimits_{w\in \widetilde{\mathbb{Z}}^d\setminus D}\|w-v\| =0\}. $$
By the time-reversal invariance property of Brownian motion, one has the symmetric property that $\mathbb{K}_D(v,w)=\mathbb{K}_D(w,v)$ for all $v,w\in \widetilde{\partial} D$.

\textbf{Equilibrium measure and capacity.} For any $D\subset \widetilde{\mathbb{Z}}^d$ and $v\in \widetilde{\partial} D$, the equilibrium measure for $D$ at $v$ is defined by 
\begin{equation}\label{defQDv}
	\mathbb{Q}_D(v):= \lim\limits_{N \to \infty} \sum\nolimits_{w\in \partial B(N)}\mathbb{K}_{D\cup \partial B(N)}(v,w).
\end{equation}
Provided with (\ref{defQDv}), the capacity of $D$ is defined by 
\begin{equation}
	\mathrm{cap}(D):= \sum\nolimits_{v\in \widetilde{\partial} D}\mathbb{Q}_D(v).
\end{equation}
Here the sum is well-defined because $\widetilde{\partial} D$ is countable. In addition, the capacity $\mathrm{cap}(\cdot)$ is subadditive, meaning that for any $D_1,D_2\subset \widetilde{\mathbb{Z}}^d$, one has 
\begin{equation}\label{2.8}
	\mathrm{cap}(D_1\cup D_2)\le \mathrm{cap}(D_1)+\mathrm{cap}(D_2).
\end{equation} 
Moreover, the capacity of $\widetilde{B}(N)$ is known to be of order $N^{d-2}$. For any $D\subset \widetilde{\mathbb{Z}}^d$ and $w\in \widetilde{\mathbb{Z}}^d\setminus D$, by the last-exit decomposition (see e.g., \cite[Section 8.2]{morters2010brownian}), one has 
\begin{equation}\label{ineq_2.15}
	\widetilde{\mathbb{P}}_w\big(\tau_{D}<\infty\big) = \sum\nolimits_{v\in  \widetilde{\partial} D} \widetilde{G}(w,v) \mathbb{Q}_D(v). 
\end{equation}
Combined with (\ref{bound_green}), it implies that for $D\subset \widetilde{B}(N)$ and $w\in [\widetilde{B}(2N)]^c$ with $N\ge 1$, 
\begin{equation}\label{new2.20}
		\widetilde{\mathbb{P}}_w(\tau_D<\infty)\asymp |w|^{2-d}\mathrm{cap}(D).
	\end{equation}

\subsection{Properties of the GFF}\label{submit_section2.2}
Recall from the sentence below (\ref{QM_ineq_1}) that $\mathbb{P}^D$ denotes the law of the GFF with zero boundary conditions on $D$. Let $\mathbb{E}^D$ denote the expectation under $\mathbb{P}^D$. The covariance of $\{\widetilde{\phi}_v\}_{v\in \widetilde{\mathbb{Z}}^d\setminus D}\sim \mathbb{P}^D$ is given by 
 \begin{equation}
	\mathbb{E}^D\big[\widetilde{\phi}_{v_1}\widetilde{\phi}_{v_2}\big]= \widetilde{G}_{D}(v_1,v_2),\ \ \forall v_1,v_2\in \widetilde{\mathbb{Z}}^d\setminus D.
\end{equation}

\textbf{Harmonic average.} 
For any $D \subset \widetilde{\mathbb{Z}}^d $, suppose that all GFF values on $D$ are given. For any $v\in \widetilde{\mathbb{Z}}^d$, by treating values on $D$ as boundary conditions, we define the harmonic average at $v$ by 
	\begin{equation}\label{def_Hv}
		\mathcal{H}_v(D):=\left\{\begin{array}{ll}
			0   &\   \text{if}\ v\in D^{\circ}; \\
			\sum\nolimits_{w\in \widetilde{\partial} D} \widetilde{\mathbb{P}}_v\big(\tau_{D}=\tau_w<\infty \big) \widetilde{\phi}_{w}  &\  \text{otherwise}.
		\end{array}
		\right.
	\end{equation} 
Here $D^{\circ}:=\{v\in \widetilde{\mathbb{Z}}^d:\inf\nolimits_{w\in \widetilde{\mathbb{Z}}^d\setminus D} \|w-v\|>0 \}$ represents the interior of $D$. For any $v\in (D^{\circ})^c$, by the strong Markov property of $\widetilde{\phi}_\cdot$ (see e.g., \cite[Theorem 8]{ding2020percolation}), $\mathcal{H}_v(D)$ exactly equals the conditional expectation of $\widetilde{\phi}_v$ given all GFF values on $D$.

The next lemma shows that the harmonic average captures the conditional probability for a faraway point to be connected to $\bm{0}$ by $\widetilde{E}^{\ge 0}$. For any $m>0$, we denote by $\mathcal{B}(m):=\{y\in \mathbb{Z}^d:|y|\le m\}$ the Euclidean ball centered at $\bm{0}$ with radius $m$.



\begin{lemma}\label{newlemma3.1}
	For any $d\ge 3$, $n\ge 1$ and $D\subset \widetilde{B}(n)$, suppose that all values of $\widetilde{\phi}_{\cdot}$ on $D$ are given and are all non-negative. Then for any $x\in \mathbb{Z}^d\setminus B(2dn)$, the conditional probability of $\big\{D \xleftrightarrow{\ge 0} x\big\}$ is of the same order as $\big(  |x|^{2-d}n^{d-2}\overline{\mathcal{H}}_n(D) \big)\land 1$, where $\overline{\mathcal{H}}_n(D):=|\partial \mathcal{B}(dn)|^{-1}\sum\nolimits_{y\in \partial \mathcal{B}(dn)}\mathcal{H}_{y}(D)$.
\end{lemma}
\begin{proof}
	According to \cite[Equation (18)]{lupu2018random} (where $C_A^{\mathrm{eff}}$ is equivalent to $\mathbb{K}_{A}$ defined in (\ref{def_KD})), the conditional probability of $D \xleftrightarrow{\ge 0} x$ equals to 
	\begin{equation}\label{newest_3.4}
	\begin{split}
\mathfrak{X} := 	\mathbb{E}\big[  \big(1- e^{-2\widetilde{\phi}_x\sum_{v\in \widetilde{\partial} D } \mathbb{K}_{D \cup \{x\}}(x,v) \cdot  \widetilde{\phi}_v  } \big)\cdot   \mathbbm{1}_{\widetilde{\phi}_x\ge 0} \big],
 \end{split}
\end{equation}
where $\{\widetilde{\phi}_v\}_{v\in \widetilde{\partial} D}$ are deterministic, and the law of $\widetilde{\phi}_x$ is given by $N(a_x, \sigma^2)$ with $a_x:= \mathcal{H}_x(D)$ and $\sigma^2:=\widetilde{G}_D(x,x)\asymp 1$. Moreover, it was established in the proof of \cite[Lemma 2.6]{cai2024one} (more precisely, see the equality below \cite[Equation (2.19)]{cai2024one}) that 
\begin{equation}\label{newest_3.5}
	\mathbb{K}_{D \cup \{x\}}(x,v)  = \tfrac{1}{2d}\sum\nolimits_{x'\in \mathbb{Z}^d:x'\sim x} \widetilde{\mathbb{P}}_{x'}\big(\tau_{D  \cup \{x\}} =\tau_{v}<\infty\big). 
\end{equation}
In addition, for each $x'\sim x$, by \cite[Lemma 2.1]{inpreparation} and the Harnack's inequality (see e.g., \cite[Theorem 6.3.9]{lawler2010random}), we have 
\begin{equation}\label{newest_3.6}
	\widetilde{\mathbb{P}}_{x'}\big(\tau_{D \cup \{x\}} =\tau_{v}<\infty\big) \asymp  \widetilde{\mathbb{P}}_{x'}\big(\tau_{D} =\tau_{v}<\infty\big) \asymp  \widetilde{\mathbb{P}}_{x}\big(\tau_{D } =\tau_{v}<\infty\big).
\end{equation}
Plugging (\ref{newest_3.5}) and (\ref{newest_3.6}) into (\ref{newest_3.4}), we obtain that $\mathbb{I}(c_\diamond)	\le \mathfrak{X} \le  \mathbb{I}(C_\diamond)$ for some $C_\diamond,c_\diamond>0$, where the function $\mathbb{I}(\cdot)$ is defined as  
	\begin{equation}
\mathbb{I}(b):= \int_{0}^{\infty}  \frac{1}{\sqrt{2\pi \sigma^2}}e^{-\frac{(t-a_x)^2}{2\sigma^2}}\big(1-e^{-ba_xt} \big)  dt. 
\end{equation}


Next, we estimate $\mathbb{I}(b)$ for $b\in \{C_\diamond,c_\diamond\}$ separately in two cases. 

	\noindent \textbf{Case 1:} when $0\le a_x\le  \frac{1}{10}$. We decompose $\mathbb{I}(b)$ as follows:
\begin{equation}\label{newineq.3.11}
	\begin{split}
		\mathbb{I}(b)= &  \int_{0}^{\frac{1}{a_x}}   \frac{1}{\sqrt{2\pi \sigma^2}}e^{-\frac{(t-a_x)^2}{2\sigma^2}}  \big(1-e^{-ba_xt} \big)  dt\\
		&+\int_{\frac{1}{a_x}}^{\infty}   \frac{1}{\sqrt{2\pi \sigma^2}}e^{-\frac{(t-a_x)^2}{2\sigma^2}}  \big(1-e^{-ba_xt} \big)  dt \\
		\asymp  &  a_x \int_{0}^{\frac{1}{a_x}}   \frac{t}{\sqrt{2\pi \sigma^2}}e^{-\frac{(t-a_x)^2}{2\sigma^2}}  dt  + \int_{\frac{1}{a_x}}^{\infty}   \frac{1}{\sqrt{2\pi \sigma^2}}e^{-\frac{(t-a_x)^2}{2\sigma^2}}dt\\
		= & a_x \Big(  \int_{-a_x}^{\frac{1}{a_x}-a_x}   \frac{t}{\sqrt{2\pi \sigma^2}}e^{-\frac{t^2}{2\sigma^2}}  dt + a_x  \int_{-a_x}^{\frac{1}{a_x}-a_x}   \frac{1}{\sqrt{2\pi \sigma^2}}e^{-\frac{t^2}{2\sigma^2}}   dt \Big)\\
		& +\int_{\frac{1}{a_x}-a_x}^{\infty}   \frac{1}{\sqrt{2\pi \sigma^2}}e^{-\frac{t^2}{2\sigma^2}}dt,
	\end{split}
\end{equation}
where the last equality follows from a simple change of variable in calculus. In this case, since $\frac{1}{a_x}-a_x>1$, one has 
\begin{equation}\label{newineq_3.12}
	\int_{-a_x}^{\frac{1}{a_x}-a_x}   \frac{t}{\sqrt{2\pi \sigma^2}}e^{-\frac{t^2}{2\sigma^2}}  dt  \asymp 1
\end{equation}
and hence, $ \mathbb{I}(b)\gtrsim a_x$ (by (\ref{newineq.3.11})). Meanwhile, by the tail estimate of the normal distribution (see e.g., \cite[Lemma 2.11]{cai2024one}), we have 
\begin{equation}
\int_{\frac{1}{a_x}-a_x}^{\infty}   \frac{1}{\sqrt{2\pi \sigma^2}}e^{-\frac{t^2}{2\sigma^2}}dt \lesssim  e^{-c(\frac{1}{a_x}-a_x )^2} \overset{0\le a_x\le  \frac{1}{10}}{\lesssim} a_x.
\end{equation}
In addition, since $0\le a_x\le  \frac{1}{10}$, we have
\begin{equation}\label{newineq_3.14}
	a_x  \int_{-a_x}^{\frac{1}{a_x}-a_x}   \frac{1}{\sqrt{2\pi \sigma^2}}e^{-\frac{t^2}{2\sigma^2}}   dt \le \frac{1}{10}. 
\end{equation}
 By plugging (\ref{newineq_3.12})--(\ref{newineq_3.14}) into (\ref{newineq.3.11}), we get $ \mathbb{I}(b)\lesssim  a_x$, and thus obtain $ \mathbb{I}(b)\asymp a_x$.

\noindent\textbf{Case 2:} when $a_x >  \frac{1}{10}$. Since $1-e^{-ba_xt}\gtrsim 1$ for all $t\ge 1$ and $a_x >  \frac{1}{10}$, one has 
\begin{equation}
	\begin{split}
			\mathbb{I}(b) \gtrsim   \int_{1}^{\infty} \frac{1}{\sqrt{2\pi \sigma^2}} e^{-\frac{(t-a_x)^2}{2\sigma^2}} dt = \int_{1-a_x}^{\infty} \frac{1}{\sqrt{2\pi \sigma^2}} e^{-\frac{t^2}{2\sigma^2}} dt  \asymp 1. 
	\end{split}
\end{equation}
Combined with the trivial bound $\mathbb{I}(b) \le 1$, it yields that $\mathbb{I}(b)\asymp 1$.

To sum up, we obtain that 
\begin{equation}\label{newest_3.11}
\mathfrak{X} \asymp  a_x\land 1. 
\end{equation}
Furthermore, by the strong Markov property one has 
\begin{equation}
\min_{y\in \partial \mathcal{B}(dn)} \mathcal{H}_{y}(D) \le  \frac{a_x}{\widetilde{\mathbb{P}}_x\big( \tau_{\mathcal{B}(dn)}<\infty \big) } \le  \max_{y\in \partial \mathcal{B}(dn)} \mathcal{H}_{y}(D),
\end{equation}
where both sides are of the same order as $\overline{\mathcal{H}}_n(D)$ by the Harnack's inequality. Thus, by $\widetilde{\mathbb{P}}_x\big( \tau_{\mathcal{B}(dn)}<\infty \big) \asymp  |x|^{2-d} n^{d-2}$ (which follows from (\ref{new2.20})), we get 
\begin{equation}
	a_x   \overset{}{\asymp} |x|^{2-d} n^{d-2} \overline{\mathcal{H}}_n(D). 
\end{equation}
Combined with (\ref{newest_3.4}) and (\ref{newest_3.11}), it concludes the proof of this lemma.
\end{proof}

\textbf{Decoupling inequality.} A useful tool called the \textit{decoupling inequality} is frequently used in the study of GFF level-sets (see \cite{drewitz2018sign, drewitz2014chemical,popov2015decoupling, rodriguez2013phase}), particularly in analyzing the correlation between increasing events. Specifically, by applying the same proofs as in \cite[Corollary 1.3 and Proposition 1.4]{popov2015decoupling} (which focus on the discrete GFF), we have the following lemma. Recall the notation $\widetilde{B}(\cdot)$ in (\ref{def_continuous_box}). We denote $\widetilde{B}_x(N):= x+ \widetilde{B}(N)$ for $x\in \mathbb{Z}^d$ and $N>0$. 
\begin{lemma}\label{lemma_decoupling_ineq}
	For any $d\ge 3$, there exist $C,c>0$ such that for any $\delta>0$, $N\ge 1$, $x\in \partial B(4N)$, and any increasing functions $f_1:\{\widetilde{\phi}_v\}_{v\in \widetilde{B}(N)}\to [0,1]$ and $f_2:\{\widetilde{\phi}_v\}_{v\in \widetilde{B}_x(N)}\to [0,1]$, 
		\begin{equation}
		\mathbb{E}\big[f_1(\widetilde{\phi}_{\cdot})f_2(\widetilde{\phi}_{\cdot}) \big]\le \mathbb{E}\big[f_1(\widetilde{\phi}_{\cdot})\big]\mathbb{E}\big[f_2(\widetilde{\phi}_{\cdot}+\delta) \big]+ Ce^{-c\delta^2N^{d-2}}.
	\end{equation}

\end{lemma}

\textbf{Percolation probability.} For the percolation probability in the supercritical regime, it has been proved in \cite[Theorem 3]{ding2020percolation} that for any $h\ge 0$, 
\begin{equation}\label{usenew2.23}
	\mathbb{P}\big(\bm{0}\xleftrightarrow{\ge -h} \infty \big) = \mathbb{E}\big[\mathrm{sign}(\widetilde{\phi}_{\bm{0}}+ h )\big]\asymp h\land 1,
\end{equation}
where $\mathrm{sign}(a):=\mathbbm{1}_{a\ge 0}-\mathbbm{1}_{a< 0}$. Moreover, for any $v\in\widetilde{\mathbb{Z}}^d$ and any increasing event $\mathsf{F}$, by the uniqueness of the infinite cluster (which can be obtained using \cite[Theorem 1]{gandolfi1992uniqueness}; see e.g., \cite[Lemma 5.3]{lupu2016loop}), we have 
 \begin{equation}\label{use2.37}
 \begin{split}
 	\mathbb{P}\big( \mathsf{F}, \bm{0}\xleftrightarrow{\ge -h} v \big) \ge & \mathbb{P}\big(\mathsf{F}, \bm{0}\xleftrightarrow{\ge -h}\infty  , v \xleftrightarrow{\ge -h}  \infty \big)\\
 \ge 	& \mathbb{P}\big(\mathsf{F}, \bm{0}\xleftrightarrow{\ge -h} \infty \big)	\mathbb{P}\big( \bm{0}\xleftrightarrow{\ge -h} \infty \big),
 \end{split}
 \end{equation}
 where the last inequality follows from the FKG inequality and the translation invariance of $\widetilde{\mathbb{Z}}^d$. Taking $\mathsf{F}$ as the sure event $\Omega$ in (\ref{use2.37}) and using (\ref{usenew2.23}), we have 
\begin{equation}\label{finish2.15}
		\mathbb{P}\big( \bm{0}\xleftrightarrow{\ge -h} v \big) \gtrsim (h\land 1)^2, \ \ v\in \widetilde{\mathbb{Z}}^d. 
\end{equation}

\subsection{Loop soup and isomorphism theorem}\label{subsection_iso_thm}

Recall the definition of the loop soup at the beginning of Section \ref{section_intro}. We now cite the following two estimates on the loop measure for later use. For a loop $\widetilde{\ell}$, we denote by $\mathrm{ran}(\widetilde{\ell})$ the collection of points in $\widetilde{\mathbb{Z}}^d$ visited by $\widetilde{\ell}$ (where ran is the abbreviation of the word ``range'').

For any $D\subset \widetilde{\mathbb{Z}}^d$, we denote $\widetilde{\mu}^{ D}:=\widetilde{\mu} \circ \mathbbm{1}_{\mathrm{ran}(\widetilde{\ell}) \subset \widetilde{\mathbb{Z}}^d\setminus D}$. For any $\alpha>0$, let $\widetilde{\mathcal{L}}_\alpha^{D}$ be the Poisson point process with intensity measures $\widetilde{\mu}^{D}$. For any $v\in \widetilde{\mathbb{Z}}^d\setminus D$, let $\widehat{\mathcal{L}}^{D,v}_{1/2}$ denote the total local time at $v$ of loops in $\widetilde{\mathcal{L}}_{1/2}^{D}$. The powerful coupling presented in \cite[Proposition 2.1]{lupu2016loop} implies that for $\{\widetilde{\phi}_v\}_{v\in \widetilde{\mathbb{Z}}^d\setminus D} \sim  \mathbb{P}^{D}$, one has 
\begin{equation}\label{iso}
	\big\{ \widehat{\mathcal{L}}^{D,v}_{1/2} \big\}_{v\in \widetilde{\mathbb{Z}}^d\setminus D} \overset{\mathrm{d}}{=}  \big\{\tfrac{1}{2}\widetilde{\phi}_v^2 \big\}_{v\in \widetilde{\mathbb{Z}}^d\setminus D}. 
\end{equation}
We denote the union of ranges of loops in a point measure $\widetilde{\mathcal{L}}$ by $\cup \widetilde{\mathcal{L}}$. For simplicity, we abbreviate ``$\xleftrightarrow{\cup \widetilde{\mathcal{L}}_{1/2}^{D}}$'' as ``$\xleftrightarrow{(D)}$'' (note that adding parentheses to a set above the left-right arrow changes its meaning compared to having no parentheses). Moreover, when $D=\emptyset$, we abbreviate ``$\xleftrightarrow{\cup \widetilde{\mathcal{L}}_{1/2}}$'' as ``$\xleftrightarrow{}$''.


\textbf{Isomorphism theorem involving random interlacements.} The model of random interlacements (denoted by $\mathcal{I}^u$, where $u>0$ represents the intensity of $\mathcal{I}^u$) was first introduced in \cite{sznitman2010vacant} as a Poisson point process composed of double-infinite trajectories of simple random walks on $\mathbb{Z}^d$ for $d\ge 3$. A detailed introduction on this model can be found in \cite{drewitz2014introduction}. Referring to \cite{lupu2016loop}, random interlacements can be generalized to the metric graph $\widetilde{\mathbb{Z}}^d$ by replacing simple random walks on $\mathbb{Z}^d$ with Brownian motions on $\widetilde{\mathbb{Z}}^d$. We denote this generalized model by $\widetilde{\mathcal{I}}^u$. Next, we list some basic properties of $\widetilde{\mathcal{I}}^u$ as follows. 
\begin{enumerate}
	\item  For any $u>0$, since every trajectory in $\widetilde{\mathcal{I}}^u$ has infinite range, we know that $\widetilde{\mathcal{I}}^u$ a.s. percolates. Moreover, the union of ranges of all trajectories in $\widetilde{\mathcal{I}}^u$ (denoted by $\cup \widetilde{\mathcal{I}}^u$) is connected.

	\item  For any $u>0$ and $D\subset \widetilde{\mathbb{Z}}^d$, one has 
	\begin{equation}\label{2.30}
		\mathbb{P}\big( D \cap (\cup \widetilde{\mathcal{I}}^u)= \emptyset \big)= e^{-u\cdot \mathrm{cap}(D)}. 
	\end{equation}

	\item  For any $N\ge 1$, $x\in \mathbb{Z}^d$ with $|x|\ge CN$, since the number of trajectories in $\widetilde{\mathcal{I}}^u$ intersecting $\widetilde{B}(N)$ is a Poisson random variable whose mean is of order $uN^{d-2}$ and since each trajectory hits $\widetilde{B}_x(N)$ with probability at most of order $|x|^{2-d}N^{d-2}$ (by (\ref{new2.20})), we have
	\begin{equation}\label{newadd_2.26}
		\mathbb{P}\big(\exists \eta\in \widetilde{\mathcal{I}}^u\ \text{intersecting both}\ \widetilde{B}(N)\ \text{and}\ \widetilde{B}_x(N) \big)\lesssim u|x|^{2-d}N^{2d-4}. 
	\end{equation}

\end{enumerate}
For any $v\in \widetilde{\mathbb{Z}}^d$, we denote by $\widehat{\mathcal{I}}^{u}_v$ the total local time at $v$ of trajectories in $\widetilde{\mathcal{I}}^u$. By \cite[Propositions 2.1 and 6.3]{lupu2016loop}, there exists a coupling between $\widetilde{\mathcal{I}}^u$, $\widetilde{\mathcal{L}}_{1/2}$ and $\{\widetilde{\phi}_v\}_{v\in \widetilde{\mathbb{Z}}^d}$ (where $\widetilde{\mathcal{I}}^u$ and $\widetilde{\mathcal{L}}_{1/2}$ are independent) such that for any $h\ge 0$, 
\begin{equation}
	\widehat{\mathcal{L}}_{1/2}^{v}+ \widehat{\mathcal{I}}^{\frac{1}{2}h^2}_v = \tfrac{1}{2}(\widetilde{\phi}_v+h)^2, \ \ \forall v\in \widetilde{\mathbb{Z}}^d.	\end{equation}
	Referring to \cite[Section 6]{lupu2016loop}, under this coupling, the infinite cluster consisting of loops in $\widetilde{\mathcal{L}}_{1/2}$ and trajectories in $\widetilde{\mathcal{I}}^{\frac{1}{2}h^2}$ (denoted by $\mathcal{C}^{[h]}_{\infty}$) is exactly the unique infinite cluster of $\widetilde{E}^{\ge -h}$ (to preserve the convention $h \geq 0$ in the remainder of this paper, we write $\widetilde{E}^{\ge -h}$ with $h \geq 0$ rather than $\widetilde{E}^{\ge h}$ with $h \leq 0$). In addition, by \cite[Theorem 2.4]{sznitman2016coupling} (see also \cite[Theorem 3.9]{drewitz2022cluster} for an extension to transient graphs), the sign of $\widetilde{\phi}_\cdot+h$ on each loop cluster disjoint from $\cup \widetilde{\mathcal{I}}^{\frac{1}{2}h^2}$ is constant and independent of others, and takes ``$+$'' (or ``$-$'') with probability $\frac{1}{2}$.

For any $h>0$ and $v\in \widetilde{\mathbb{Z}}^d$, we denote $\xleftrightarrow{(\cup \widetilde{\mathcal{L}}_{1/2})\cup (\cup \widetilde{\mathcal{I}}^{\frac{1}{2}h^2})}$ by $\xleftrightarrow{[h]} $, and define the cluster $\mathcal{C}^{[h]}(v):=\{w\in \widetilde{\mathbb{Z}}^d: w\xleftrightarrow{[h]} v\}$. For any $A_1,A_2\subset \widetilde{\mathbb{Z}}^d$, we denote by $A_1 \xleftrightarrow{[h]_\infty} A_2$ the event that the infinite cluster $\mathcal{C}^{[h]}_{\infty}$ intersects both $A_1$ and $A_2$, and denote by $A_1 \xleftrightarrow{[h]^c_{\infty}} A_2$ the event that $A_1$ and $A_2$ are connected by some loop cluster disjoint from $\mathcal{C}^{[h]}_{\infty}$. By this isomorphism theorem, we have   
\begin{equation}\label{232plus}
	\mathbb{P}(\bm{0}\xleftrightarrow{[h]}\infty)= \mathbb{P}(\bm{0}\xleftrightarrow{\ge -h}\infty)  \overset{(\ref*{usenew2.23})}{\asymp} h\land 1 . 
\end{equation}
For any $v\in \widetilde{\mathbb{Z}}^d$ and increasing event $\mathsf{F}$, by the same reasoning as in (\ref{use2.37}), one has 
\begin{equation}\label{newsubmit_223}
\begin{split}
	\mathbb{P}\big(\mathsf{F}, \bm{0}\xleftrightarrow{[h]}v, \bm{0}\xleftrightarrow{[h]}\infty\big)= &	\mathbb{P}\big(\mathsf{F}, \bm{0}\xleftrightarrow{[h]}\infty , v\xleftrightarrow{[h]}\infty\big) \\
	\overset{}{\ge } &	\mathbb{P}\big(\mathsf{F}, \bm{0}\xleftrightarrow{[h]}\infty  \big)	\mathbb{P}\big(\bm{0} \xleftrightarrow{[h]}\infty\big). 
\end{split}
\end{equation}
By taking $\mathsf{F} = \Omega$ in (\ref{newsubmit_223}), we have 
\begin{equation}\label{newsubmit_224}
	\mathbb{P}\big(  \bm{0}\xleftrightarrow{[h]}v, \bm{0}\xleftrightarrow{[h]}\infty\big) \ge \big[\mathbb{P}\big(\bm{0} \xleftrightarrow{[h]}\infty\big) \big]^2 \overset{(\ref*{232plus})}{\asymp }  (h\land 1)^2. 
\end{equation}


\subsection{Cluster decomposition}\label{subsection_cluster_decompose}
For any $m\ge n\ge 0$, we denote by 
\begin{equation}\label{def_L[n,m]}
	\widetilde{\mathcal{L}}_{1/2}[n,m]:=\widetilde{\mathcal{L}}_{1/2}\cdot \mathbbm{1}_{\mathrm{ran}(\widetilde{\ell})\cap \partial B(n)\neq \emptyset,\mathrm{ran}(\widetilde{\ell})\cap \partial B(m)\neq \emptyset}
\end{equation}
the collection of loops in $\widetilde{\mathcal{L}}_{1/2}$ crossing the annulus $B(m)\setminus B(n)$. Through a decomposition argument on loop clusters, it was proved in \cite[Lemma 5.1]{inpreparation} that the scenario with the occurrence of a large crossing loop only contributes a small fraction to the connecting probability. Subsequently, we review a special case of \cite[Lemma 5.1]{inpreparation}, which suffices in this paper.

\begin{lemma}\label{lemma_new_decomposition}
For any $d\ge 3$, there exist $C,C'>0$ such that for any $N\ge 1$, $M\ge CN$, $A  \subset \widetilde{B}(\frac{N}{10C'})$ and $A',D\subset [\widetilde{B}(10C'M)]^c$, 
\begin{equation} \label{ineq_new_decompose}
	\begin{split}
		&\mathbb{P}\big(  A \xleftrightarrow{(D)}A', \widetilde{\mathcal{L}}_{1/2}^{D}[N,M]\neq 0 \big)\\
		\lesssim  & \big(\frac{N}{M}\big)^{d-2}\mathbb{P}\big( A \xleftrightarrow{(D)}\partial B(\tfrac{N}{C'}) \big)\mathbb{P}\big(A'   \xleftrightarrow{(D)} \partial B(C'M) \big).
	\end{split}
	\end{equation}
\end{lemma}

We also cite \cite[Corollary 1.10]{inpreparation} in the following lemma.

 \begin{lemma}\label{lemma_boxtobox}
For any $d\ge 3$ with $d\neq 6$, there exists $c>0$ such that for any $N\ge 1$, $A\subset [\widetilde{B}(N)]^c$, $0<M_1\le M_2\le  cN^{1\boxdot \frac{2}{d-4}}$ and $D\subset  [\widetilde{B}(N)]^c \cup \widetilde{B}(\frac{1}{2}M_1)$, 
		\begin{equation}\label{ineq2_compare_boundtoset}
  	\frac{\mathbb{P}^D\big(A\xleftrightarrow{\ge 0}  B(M_1)\big)}{ \mathbb{P}^D\big(A\xleftrightarrow{\ge 0} B(M_2)\big)}  	\asymp \big(\frac{M_1}{M_2}\big)^{(\frac{d}{2}-1)\boxdot (d-4)}.
	\end{equation}
 \end{lemma}

We record a useful application of these two lemmas as follows. For any $M\ge CN^{1\boxdot \frac{d-4}{2}}$, $A\subset \widetilde{B}(cN^{1\boxdot \frac{2}{d-4}})$ and $A',D\subset [\widetilde{B}(C'M^{1\boxdot \frac{d-4}{2} })]^c$, we have 
 \begin{equation}\label{2.26}
	\begin{split}
	&\mathbb{P}\big(  A\xleftrightarrow{(D)}A', \widetilde{\mathcal{L}}_{1/2}^{D}[N,M]\neq 0 \big)\\ 
\overset{(\ref*{ineq_new_decompose})}{\lesssim} & \big(\frac{N}{M}\big)^{d-2}\mathbb{P}\big( A \xleftrightarrow{(D)}\partial B(c'N) \big) \mathbb{P}\big(A'  \xleftrightarrow{(D)}  \partial B(C''M) \big)\\
	\overset{(\ref*{iso}),(\ref*{ineq2_compare_boundtoset})}{\lesssim} & \big(\frac{N}{M}\big)^{d-2-[(\frac{d}{2}-1)\boxdot (d-4)]} \mathbb{P}^D\big( A \xleftrightarrow{\ge 0}\partial B(c'N) \big) \mathbb{P}^D\big(A'  \xleftrightarrow{\ge 0}  \partial B(c'N) \big)  \\
	\overset{(\ref*{QM_ineq_1}),(\ref*{iso})}{\lesssim} & N^{(\frac{d}{2}-1)\boxdot (d-4)}M^{-[(\frac{d}{2}-1)\boxdot 2]} \mathbb{P}\big(  A\xleftrightarrow{(D)}A'  \big). 
	\end{split}
\end{equation}

For any disjoint subsets $A_0,A_1,...,A_k\subset \widetilde{\mathbb{Z}}^d$, we denote 
\begin{equation}\label{addnew2.30}
	\{A_0\xleftrightarrow{\cdots} A_1,...,A_k\}:= \cap_{1\le i\le k}\{A_0\xleftrightarrow{\cdots} A_i\}.
\end{equation}

\begin{lemma}\label{lemma_Ato_x_boundary}
For any $3\le d\le 5$, there exists $\Cl\label{const_lemma_Ato_x_boundary1}>0$ such that for any $n\ge 1$, $A\subset \widetilde{B}(\Cref{const_lemma_Ato_x_boundary1}^{-1}n)$, $D\subset \widetilde{B}(\Cref{const_lemma_Ato_x_boundary1}^{-1}n)\cup [\widetilde{B}(\Cref{const_lemma_Ato_x_boundary1}n)]^c$, $x\in B(2n)\setminus B(\frac{1}{2}n)$ and $a \in [4,\frac{1}{2}\Cref{const_lemma_Ato_x_boundary1}]$, 
	\begin{equation}
		\mathbb{P}^D\big( A\xleftrightarrow{\ge 0 } x,\partial B(a n)\big)\lesssim (an)^{-\frac{d}{2}+1}\mathbb{P}^D\big( A\xleftrightarrow{\ge 0}  \partial B(n)\big).
	\end{equation}
\end{lemma}
\begin{proof}
	By the isomorphism theorem, it suffices to show 
	 \begin{equation}\label{finaluse243}
		\mathbb{P}\big( A\xleftrightarrow{(D) } x,\partial B(a n)\big)\lesssim (an)^{-\frac{d}{2}+1}\mathbb{P}\big( A\xleftrightarrow{(D)}  \partial B(n)\big).
	\end{equation}
	Applying \cite[Lemma 5.2]{inpreparation}, the left-hand side of (\ref{finaluse243}) is bounded from above by 
	\begin{equation}\label{finaluse244}
	\begin{split}
			C\mathbb{P}\big( A\xleftrightarrow{(D_1) }& \partial B(C_*^{-2}n)\big) \Big[\mathbb{P}\big(x\xleftrightarrow{(D_2)} \partial B(C_*^{-1}n)  , \partial B(a n)  \big)\\
		&\ \ \ \ \ \ \ \ \ \ \  + \mathbb{P}\big(x\xleftrightarrow{(D_2)} \partial B(C_*^{-1}n)  \big)\mathbb{P}\big(\partial B(a n)  \xleftrightarrow{(D_2) } \partial B(C_*^{-1}n)  \big)  \Big]. 
	\end{split}
	\end{equation}
	Here $C_*>0$ is some constant (where we require $\Cref{const_lemma_Ato_x_boundary1}\ge 10C_*^5$), $D_1=D\cap \widetilde{B}(\Cref{const_lemma_Ato_x_boundary1}^{-1}n)$ and $D_2= D\cap [\widetilde{B}(\Cref{const_lemma_Ato_x_boundary1}n)]^c$. Next, we estimate the terms in (\ref{finaluse244}) one by one.

	For $\mathbb{P}\big( A\xleftrightarrow{(D_1) } \partial B(C_*^{-2}n)\big)$, by Lemma \ref{lemma_boxtobox}, we know that it is proportional to $\mathbb{P}\big( A\xleftrightarrow{(D_1) } \partial B(n)\big)$. Moreover, referring to \cite[Proposition 1.8]{inpreparation}, the addition of an absorbing boundary $D_2$ alters the connecting probability between $A$ and $\partial B(n)$ by only a constant factor. Combining these two estimates, we get 
	\begin{equation}
		\mathbb{P}\big( A\xleftrightarrow{(D_1) } \partial B(C_*^{-2}n)\big) \asymp \mathbb{P}\big( A\xleftrightarrow{(D) } \partial B(n)\big). 
	\end{equation}
 Now we estimate the probability of $\{x\xleftrightarrow{(D_2)} \partial B(C_*^{-1}n)  , \partial B(a n)  \}$. Noting that this event implies $\{ x\xleftrightarrow{}\partial B_x(\frac{1}{2}a n) \}$, one has 
	\begin{equation}
		\mathbb{P}\big(x\xleftrightarrow{(D_2)} \partial B(C_*^{-1}n)  , \partial B(a n)  \big) \le \mathbb{P}\big(x\xleftrightarrow{}\partial B_x(\tfrac{1}{2}a n)  \big) \overset{(\ref*{iso}),(\ref*{one_arm_low})}{\lesssim } (an)^{-\frac{d}{2}+1}. 
	\end{equation} 
 Furthermore, since $\widetilde{\mathcal{L}}_{1/2}^{D_2}\le \widetilde{\mathcal{L}}_{1/2}$, the remaining term in (\ref{finaluse244}) is upper-bounded by 
 \begin{equation}\label{finaluse247}
 \begin{split}
 		& \mathbb{P}\big(x\xleftrightarrow{ } \partial B(C_*^{-1}n)  \big)\mathbb{P}\big(\partial B(a n)  \xleftrightarrow{ } \partial B(C_*^{-1}n)  \big)\\
 	\overset{\widetilde{B}_x(\frac{1}{2}n)\subset [\widetilde{B}(C_*^{-1}n)]^c,(\ref*{iso})}{\lesssim }   & \theta_d(\tfrac{1}{2}n )\rho_d(C_*^{-1}n,an ) \overset{(\ref*{one_arm_low}),(\ref*{crossing_low})}{\lesssim } (an)^{-\frac{d}{2}+1}. 
 \end{split}
 \end{equation} 
 Putting (\ref{finaluse244})--(\ref{finaluse247}) together, we obtain (\ref{finaluse243}) and thus complete the proof.
\end{proof}

We conclude this section by citing several special cases of \cite[Lemma 5.3]{inpreparation}.

\begin{lemma} \label{corollary_2.9}
For any $3\le d\le 5$, there exist $C,c>0$ such that the following hold.
\begin{enumerate}
	\item  For any $N\ge 1$, $A\subset \widetilde{B}(N)$, $M\ge CN$, $x,y\in \partial B(M)$ and $D\subset \widetilde{B}(N)\cup [\widetilde{B}(2M)]^c$,
\begin{equation}\label{final_addnew2.33}
	\mathbb{P}\big(A \xleftrightarrow{\ge 0} x,y \big)\lesssim    |x-y|^{-\frac{d}{2}+1} \mathbb{P}\big(A \xleftrightarrow{\ge 0} x  \big). 
\end{equation}
	Especially, by taking $A=\{\bm{0}\}$ and $D=\emptyset$, one has
	\begin{equation}\label{addnew2.33}
		\mathbb{P}\big(\bm{0}\xleftrightarrow{\ge 0} x,y \big)\lesssim    |x-y|^{-\frac{d}{2}+1}M^{2-d}. 	
		\end{equation}

	\item For any $N\ge 1$ and $x\in B(cN)$, 
\begin{equation}\label{addnew2.34}
	\mathbb{P}\big(\partial B(N) \xleftrightarrow{\ge 0} \bm{0} ,   x\big)\lesssim    |x|^{-\frac{d}{2}+1}\theta_d(N). 
\end{equation}
	
\end{enumerate}
\end{lemma}

\section{Existence and equivalence of IICs}\label{section_exi_IIC}

This section is dedicated to proving Theorem \ref{thm_iic}. To achieve this, the key is to establish the uniform convergence stated in the subsequent proposition. Its precise formulation requires the following definition.

\begin{definition}\label{def_admissible}
	We say a family of events $\{\mathsf{A}_h\}_{h\ge 0}$ is admissible if there exist $k\in \mathbb{N}^+$, $\{a_1,...,a_k\}\subset (0,\infty)$, $\{v_1,...,v_k\}\subset   \widetilde{\mathbb{Z}}^d$ and disjoint subsets $\gamma_1,...,\gamma_l\subset \{v_1,...,v_k\}$ such that $\mathsf{A}_h=\mathsf{A}_h^{(\mathrm{i})}(v_1,...,v_k;a_1,...,a_k)\cap \mathsf{A}_h^{(\mathrm{ii})}(\gamma_1,...,\gamma_l)$ for all $h\ge 0$, where $\mathsf{A}_h^{(\mathrm{i})}(\cdot )$ and $\mathsf{A}_h^{(\mathrm{ii})}(\cdot)$ are defined as
\begin{equation}\label{v2_3.1}
	\mathsf{A}^{(\mathrm{i})}(v_1,...,v_k;a_1,...,a_k):=\cap_{1\le i\le k} \big\{ \widehat{\mathcal{L}}_{1/2}^{v_i}+ \widehat{\mathcal{I}}^{\frac{1}{2}h^2}_{v_i} \ge a_i\big\},
\end{equation}
\begin{equation}\label{v2_3.2}
	\begin{split}
		\mathsf{A}^{(\mathrm{ii})}(\gamma_1,...,\gamma_l ):= & \cap_{1\le j\le l}\cap_{v,v'\in \gamma_j}\big\{ v \xleftrightarrow{[h]} v'\big\}.
	\end{split}
\end{equation}
\end{definition}

For any $h\ge 0$, $\mathsf{A}_h$ is increasing in both loops and interlacements. Moreover, by $\{\widehat{\mathcal{L}}_{1/2}^{v_i}+ \widehat{\mathcal{I}}^{\frac{1}{2}h^2}_{v_i} \ge a_i \}\supset \{\widehat{\mathcal{L}}_{1/2}^{v_i} \ge a_i \}$, $\{ v \xleftrightarrow{[h]} v' \}\supset \{ v \xleftrightarrow{} v'\big\}$ and the FKG inequality,
 \begin{equation}\label{def_c_star}
	\begin{split}
		\mathbb{P}\big( \mathsf{A}_h \big) \ge  c_{\star}=&c_{\star}(v_1,...,v_k;a_1,...,a_k;\gamma_1,...,\gamma_l  )   \\
		:=& \prod\nolimits_{1\le i\le k}\mathbb{P}\big(\widehat{\mathcal{L}}_{1/2}^{v_i} \ge a_i \big)  \prod\nolimits_{1\le j\le l} \prod\nolimits_{v,v'\in \gamma_j} \mathbb{P}\big(v \xleftrightarrow{} v' \big)>0.
			\end{split}
\end{equation}
Now we present the key proposition for proving Theorem \ref{thm_iic}.

\begin{proposition}\label{lemma_converge_loop}
	For $d\ge 3$ with $d\neq 6$ and for any admissible  $\{\mathsf{A}_h\}_{h\ge 0}$, there exists $\cl\label{const_iic_loop}(d,\{\mathsf{A}_h\}_{h\ge 0})>0$ such that the following hold:
	\begin{enumerate}
		\item $\big\{\mathbb{P}\big(\mathsf{A}_h\mid \bm{0}\xleftrightarrow{[h]}\partial B(N)\big)\big\}_{0\le h\le \cref{const_iic_loop}}$ uniformly converges as $N\to \infty$;

		\item $\big\{\mathbb{P}\big(\mathsf{A}_h\mid \bm{0}\xleftrightarrow{[h]}x\big)\big\}_{0\le h\le \cref{const_iic_loop}}$ uniformly converges as $x\to \infty$ (where $x\in \mathbb{Z}^d$);

		\item $\big\{\mathbb{P}\big(\mathsf{A}_h\mid \mathrm{cap}(\mathcal{C}^{[h]}(\bm{0}))\ge T \big)\big\}_{0\le h\le \cref{const_iic_loop}}$ uniformly converges as $T\to \infty$.

	\end{enumerate}
	
\end{proposition}

The proof of Proposition \ref{lemma_converge_loop} will be presented in Section \ref{proof_prop_converge_loop}. In the next subsection, we establish the existence and equivalence of Types (1)--(4) IICs assuming Proposition \ref{lemma_converge_loop} holds.

\subsection{Proof of Theorem \ref{thm_iic} assuming Proposition \ref{lemma_converge_loop}}

For brevity, for any $h\ge 0$, we denote the events
\begin{equation}\label{notation_H}
	\mathsf{H}^{\mathbf{N}}_h=\mathsf{H}^{\mathbf{N}}_h(N):=\big\{\bm{0}\xleftrightarrow{[h]}\partial B(N)\big\}, 
\end{equation}
\begin{equation}
\mathsf{H}^{\mathbf{x}}_h= 	\mathsf{H}^{\mathbf{x}}_h(x):=\big\{\bm{0}\xleftrightarrow{[h]}x\big\}, 
\end{equation}
\begin{equation}\label{notation_H3}
	\mathsf{H}^{\mathbf{T}}_h= \mathsf{H}^{\mathbf{T}}_h(T):=\big\{\mathrm{cap}(\mathcal{C}^{[h]}(\bm{0}))\ge T \big)\big\}. 
\end{equation}
We also denote by $\widetilde{\mathsf{H}}^{\mathbf{N}}_h$, $\widetilde{\mathsf{H}}^{\mathbf{x}}_h$ and $\widetilde{\mathsf{H}}^{\mathbf{T}}_h$ the analogues of $\mathsf{H}^{\mathbf{N}}_h$, $\mathsf{H}^{\mathbf{x}}_h$ and $\mathsf{H}^{\mathbf{T}}_h$ obtained by replacing $\xleftrightarrow{[h]}$ and $\mathcal{C}^{[h]}(\bm{0})$ with $\xleftrightarrow{\ge -h}$ and $\mathcal{C}^{\ge -h}(\bm{0})$ respectively. Note that here we choose the mathbf font for $\mathbf{N}, \mathbf{x}, \mathbf{T}$ in superscripts to indicate the types of IICs, but not to mean actual math quantities $N, x, T$. In addition, we use $\diamond \to \infty$ to represent $N\to \infty$, $x\to \infty$, and $T\to \infty$ for $\diamond=\mathbf{N}, \mathbf{x}, \mathbf{T}$ respectively.

Assuming Proposition \ref{lemma_converge_loop}, for any $\diamond\in \{\mathbf{N},\mathbf{x},\mathbf{T}\}$ and any admissible $\{\mathsf{A}_h\}_{h\ge 0}$, the limits $\{f_h^{\diamond}(\mathsf{A}_h):= \lim\limits_{\diamond \to \infty}\mathbb{P}\big(\mathsf{A}_h \mid \mathsf{H}_h^\diamond \big) \}_{0\le h\le \cref{const_iic_loop}} $ exist and are continuous with respect to $h$ on $[0,\cref{const_iic_loop}]$. When $h=0$, $f_0^{\mathbf{N}}(\mathsf{A}_0)$, $f_0^{\mathbf{x}}(\mathsf{A}_0)$ and $f_0^{\mathbf{T}}(\mathsf{A}_0)$ correspond to Type (1), Type (3) and Type (4) IICs respectively for the loop soup $\widetilde{\mathcal{L}}_{1/2}$ (recalling (\ref{iic_type1}), (\ref{iic_type3}) and (\ref{iic_type4})). When $h>0$, for $\diamond\in \{ \mathbf{N},\mathbf{T}\}$, since the event $\mathsf{H}_h^\diamond$ converges to $\bm{0}\xleftrightarrow{[h]}\infty$ (which happens with positive probability) as $\diamond \to \infty$, one has \begin{equation}\label{3.8}
	f_h^{\diamond}(\mathsf{A}_h)= \mathbb{P}\big(\mathsf{A}_h\mid \bm{0}\xleftrightarrow{[h]}\infty\big), \ \ \forall\ \diamond \in \{\mathbf{N},\mathbf{T}\},0<h\le \cref{const_iic_loop}. 
\end{equation}
Meanwhile, by the continuity of $f_h^{\diamond}(\mathsf{A}_h)$, we have 
\begin{equation}\label{3.9}
	f_0^{\diamond}(\mathsf{A}_0)= \lim_{h\downarrow 0} f_h^{\diamond}(\mathsf{A}_h),\ \forall\ \diamond  \in \{\mathbf{N},\mathbf{x},\mathbf{T}\}. 
\end{equation}
Note that $\lim\limits_{h\downarrow 0}  \mathbb{P}\big(\mathsf{A}_h\mid \bm{0}\xleftrightarrow{[h]}\infty\big)$ corresponds to Type (2) IIC for $\widetilde{\mathcal{L}}_{1/2}$ (recall (\ref{iic_type2})). Thus, by (\ref{3.8}) and (\ref{3.9}), we know that Types (1), (2) and (4) IICs for $\widetilde{\mathcal{L}}_{1/2}$ are equivalent. Before discussing Type (3) IIC, we extend these results to the case of GFF level-sets as follows.

\begin{proof}[Proof of the existence and equivalence of (\ref{iic_type1}), (\ref{iic_type2}) and (\ref{iic_type4})] 
We arbitrarily take $m\in \mathbb{N}^+$, $n\in \mathbb{N}$, $\bm{0}\in \{y_1,...,y_m\}\subset \widetilde{\mathbb{Z}}^d$, $\{z_1,...,z_n \}\subset \widetilde{\mathbb{Z}}^d$, $\{a_1,...,a_m\}\subset [0,\infty)\cup \{-\infty\}$ and $\{b_1,...,b_n\}\subset (0,\infty)$. Similar to (\ref{v2_3.1}), we denote  
\begin{equation*}
	\mathsf{F}^{+}=\mathsf{F}^{+}(y_1,...,y_m;a_1,...,a_m):=\cap_{1\le i\le m}\big\{ \widetilde{\phi}_{y_i}\ge a_i\big\},  
	\end{equation*}
\begin{equation*}
	\mathsf{F}^{-}=\mathsf{F}^{-}(z_1,...,z_n;b_1,...,b_n):=\cap_{ 1\le i\le n}\big\{ \widetilde{\phi}_{z_i}< -b_i\big\},
	\end{equation*}
	and $\mathsf{F}:=\mathsf{F}^{+}\cap \mathsf{F}^{-}$. Here $\widetilde{\phi}_v\ge -\infty$ means that there is no requirement on the value of $\widetilde{\phi}_{v}$. Note that any increasing cylinder event $\{\widetilde{\phi}_{v_i}\ge h_i,\forall 1\le i\le k\}$ can be equivalently written as the difference of two cylinder events with the same form as $\mathsf{F}$ (e.g., $\{\widetilde{\phi}_{v}\ge a,\widetilde{\phi}_{w}\ge -b  \}=\{\widetilde{\phi}_{v}\ge a \}\setminus \{\widetilde{\phi}_{v}\ge a,\widetilde{\phi}_{w}<  -b \}$ for $a\ge 0,b>0$). Thus, it suffices to prove the existence and equivalence of the limits of the conditional probabilities of $\mathsf{F}$ under the probability measures in (\ref{iic_type1}), (\ref{iic_type2}) and (\ref{iic_type4}).

	Let $c_\dagger:=\min\{b_1,...,b_n\}>0$ and let $h\in [0,c_\dagger)$ (where we set $c_\dagger=\min \emptyset = \infty$ when $n=0$). We denote by $\Upsilon_y$ (resp. $\Upsilon_z$) the collection of partitions of $\{y_1,...,y_m \}$ (resp. $\{z_1,...,z_n\}$). Recall that a partition of a set is a collection of disjoint subsets whose union is the entire set. For each $\overbar{\gamma}_y \in \Upsilon_y$ (resp. $\overbar{\gamma}_z \in \Upsilon_z$), we denote by $\widehat{\mathsf{F}}_h^+(\overbar{\gamma}_y)$ (resp. $\widehat{\mathsf{F}}_h^-(\overbar{\gamma}_z)$) the event that $\mathsf{F}^{+}$ (resp. $\mathsf{F}^-$) happens and that two points $y_i, y_j$ (resp. $z_i, z_j$) are connected by $\widetilde{E}^{\ge -h}$ (resp. $\widetilde{E}^{\le -h}:=\{v\in \widetilde{\mathbb{Z}}^d: \widetilde{\phi}_v\le -h\}$) if and only if they are in the same element of $\overbar{\gamma}_y$ (resp. $\overbar{\gamma}_z$). Note that $\widehat{\mathsf{F}}^+_h(\overbar{\gamma}_y)$ (resp. $\widehat{\mathsf{F}}_h^-(\overbar{\gamma}_z)$) with different $\overbar{\gamma}_y$ (resp. $\overbar{\gamma}_z$) are disjoint. In addition, since $0\le h< c_\dagger$ (which ensures that no $y_i$ and $z_j$ can be connected by $\widetilde{E}^{\ge -h}$ or $\widetilde{E}^{\le -h}$), one has 
	\begin{equation}\label{v2_3.8}
	\mathsf{F}  =    \cup_{\overbar{\gamma}_y \in \Upsilon_y,\overbar{\gamma}_z \in \Upsilon_z} \big(  \widehat{\mathsf{F}}_{h}^+(\overbar{\gamma}_y)\cap \widehat{\mathsf{F}}_h^{-}(\overbar{\gamma}_z) \big) .	\end{equation}
	Note that the right-hand side of (\ref{v2_3.8}) is a disjoint union. Meanwhile, for any $\overbar{\gamma}_y \in \Upsilon_y$ and $\overbar{\gamma}_z \in \Upsilon_z$, we denote by $\widehat{\mathsf{A}}_h(\overbar{\gamma}_y,\overbar{\gamma}_z) $ the event that $$\mathsf{A}_h^{(\mathrm{i})}(y_1,...,y_m,z_1,...,z_n;\tfrac{1}{2}[(h+a_1)\vee 0]^2,...,\tfrac{1}{2}[(h+a_m)\vee 0]^2,\tfrac{1}{2}(h-b_1)^2,...,\tfrac{1}{2}(h-b_n)^2)$$ happens (recall $\mathsf{A}_h^{(\mathrm{i})}(\cdot)$ in (\ref{v2_3.1})) and that two points $y_i, y_j$ (resp. $z_i, z_j$) are connected by $(\cup \widetilde{\mathcal{L}}_{1/2})\cup (\cup \widetilde{\mathcal{I}}^{\frac{1}{2}h^2})$ if and only if they are in the same element of $\overbar{\gamma}_y$ (resp. $\overbar{\gamma}_z$). Note that by the inclusion-exclusion principle, the probability of $\widehat{\mathsf{A}}_h(\overbar{\gamma}_y,\overbar{\gamma}_z)$ can be written as a combination of additions and subtractions involving probabilities of admissible events (recalling Definition \ref{def_admissible}). Combined with Proposition \ref{lemma_converge_loop}, it implies that for any $\diamond\in \{\mathbf{N},\mathbf{x},\mathbf{T}\}$, $\overbar{\gamma}_y \in \Upsilon_y$ and $\overbar{\gamma}_z \in \Upsilon_z$, the limits
\begin{equation}\label{submit_3.10}
	\big\{f_h^{\diamond}\big(\widehat{\mathsf{A}}_h(\overbar{\gamma}_y,\overbar{\gamma}_z)\big):= \lim\limits_{\diamond \to \infty}\mathbb{P}\big(\widehat{\mathsf{A}}_h(\overbar{\gamma}_y,\overbar{\gamma}_z) \mid \mathsf{H}_h^\diamond \big) \big\}_{0\le h\le c_\dagger\land \cref{const_iic_loop}}
\end{equation}
exist and satisfy the properties presented in (\ref{3.8}) and (\ref{3.9}). I.e., 
\begin{equation}\label{v2_3.9}
	f_h^{\diamond}\big(\widehat{\mathsf{A}}_h(\overbar{\gamma}_y,\overbar{\gamma}_z)\big)= \mathbb{P}\big(\widehat{\mathsf{A}}_h(\overbar{\gamma}_y,\overbar{\gamma}_z) \mid \bm{0}\xleftrightarrow{[h]}\infty\big), \ \ \forall\ \diamond \in \{\mathbf{N},\mathbf{T}\},0<h< c_\dagger\land \cref{const_iic_loop}; 
\end{equation}
\begin{equation}\label{v2_3.10}
	f_0^{\diamond}\big(\widehat{\mathsf{A}}_0(\overbar{\gamma}_y,\overbar{\gamma}_z)\big)= \lim\limits_{h \downarrow 0} f_h^{\diamond}\big(\widehat{\mathsf{A}}_h(\overbar{\gamma}_y,\overbar{\gamma}_z)\big),\ \forall\ \diamond  \in \{\mathbf{N},\mathbf{x},\mathbf{T}\}. 
\end{equation}

When $h=0$, by the isomorphism theorem we have (recall that $\bm{0} \in \{y_1, \ldots, y_m\}$)
\begin{equation}\label{3.13}
	\mathbb{P}\big( \widehat{\mathsf{F}}_0^{+}(\overbar{\gamma}_y) \cap \widehat{\mathsf{F}}_0^{-}(\overbar{\gamma}_z), \widetilde{\mathsf{H}}_0^\diamond \big) = 2^{-|\overbar{\gamma}_y|
	-|\overbar{\gamma}_z|} \mathbb{P}\big(\widehat{\mathsf{A}}_0(\overbar{\gamma}_y,\overbar{\gamma}_z), \mathsf{H}_0^\diamond \big).\end{equation}
Combined with $\mathbb{P}(\widetilde{\mathsf{H}}_0^\diamond)=\frac{1}{2}\mathbb{P}(\mathsf{H}_0^\diamond)$ (by the isomorphism theorem) and (\ref{v2_3.8}), it yields
\begin{equation}
	\begin{split}
		\mathbb{P}\big(\mathsf{F} \mid \widetilde{\mathsf{H}}_0^\diamond  \big) \overset{(\ref*{v2_3.8})}{=}& \sum\nolimits_{\overbar{\gamma}_y \in \Upsilon_y,\overbar{\gamma}_z \in \Upsilon_z} \mathbb{P}\big( \widehat{\mathsf{F}}_0^+(\overbar{\gamma}_y) \cap \widehat{\mathsf{F}}_0^-(\overbar{\gamma}_z) \mid  \widetilde{\mathsf{H}}_0^\diamond   \big)   \\
		\overset{(\ref*{3.13}),\mathbb{P}(\widetilde{\mathsf{H}}_0^\diamond)=\frac{1}{2}\mathbb{P}(\mathsf{H}_0^\diamond)}{=}& \sum\nolimits_{\overbar{\gamma}_y \in \Upsilon_y,\overbar{\gamma}_z \in \Upsilon_z}  2^{-|\overbar{\gamma}_y|
	-|\overbar{\gamma}_z|+1}  \mathbb{P}\big(\widehat{\mathsf{A}}_0(\overbar{\gamma}_y,\overbar{\gamma}_z)\mid    \mathsf{H}_0^\diamond  \big).
	\end{split}
\end{equation}
By taking the limits on both sides as $\diamond \to \infty$ and using (\ref{submit_3.10}), we obtain \begin{equation}\label{3.15}
	\begin{split}
		 \lim\limits_{\diamond \to \infty}\mathbb{P}\big(\mathsf{F} \mid \widetilde{\mathsf{H}}_0^\diamond  \big)=  \sum_{\overbar{\gamma}_y \in \Upsilon_y,\overbar{\gamma}_z \in \Upsilon_z}  2^{-|\overbar{\gamma}_y|
	-|\overbar{\gamma}_z|+1} f_0^{\diamond}\big(\widehat{\mathsf{A}}_0(\overbar{\gamma}_y,\overbar{\gamma}_z)\big), \ \forall \diamond  \in \{\mathbf{N},\mathbf{x},\mathbf{T}\}. 
	\end{split}
\end{equation}

When $h>0$, similar to (\ref{3.13}), we have 
\begin{equation}\label{3.16}
	\mathbb{P}\big( \widehat{\mathsf{F}}_h^{+}(\overbar{\gamma}_y) \cap \widehat{\mathsf{F}}_h^{-}(\overbar{\gamma}_z), \bm{0}\xleftrightarrow{\ge -h}\infty \big) = 2^{-|\overbar{\gamma}_y|
	-|\overbar{\gamma}_z|+1} \mathbb{P}\big(\widehat{\mathsf{A}}_h(\overbar{\gamma}_y,\overbar{\gamma}_z), \bm{0}\xleftrightarrow{[h]}\infty \big),
\end{equation}
where on the right-hand side there is an extra $+1$ term in the exponent (compared to (\ref{3.13})) since the sign of $\widetilde{\phi}_{\cdot}+h$ within the cluster $\mathcal{C}^{[h]}_{\infty}$ must be ``$+$'' (under the coupling presented in Section \ref{subsection_iso_thm}). Therefore, for $\diamond\in \{\mathbf{N},\mathbf{T}\}$, we have 
\begin{equation}
	\begin{split}
		\mathbb{P}\big(\mathsf{F} \mid \bm{0}\xleftrightarrow{\ge -h} \infty  \big)\overset{(\ref*{v2_3.8})}{=}& \sum\nolimits_{\overbar{\gamma}_y \in \Upsilon_y,\overbar{\gamma}_z \in \Upsilon_z} \mathbb{P}\big( \widehat{\mathsf{F}}_h^{+}(\overbar{\gamma}_y) \cap \widehat{\mathsf{F}}_h^{-}(\overbar{\gamma}_z)\mid \bm{0}\xleftrightarrow{\ge -h} \infty \big)   \\
		\overset{(\ref*{232plus}),(\ref*{3.16})}{=}& \sum\nolimits_{\overbar{\gamma}_y \in \Upsilon_y,\overbar{\gamma}_z \in \Upsilon_z}  2^{-|\overbar{\gamma}_y|
	-|\overbar{\gamma}_z|+1}  \mathbb{P}\big(\widehat{\mathsf{A}}_h(\overbar{\gamma}_y,\overbar{\gamma}_z)\mid  \bm{0}\xleftrightarrow{[h]} \infty \big)\\
	\overset{(\ref*{v2_3.9})}{=} & \sum\nolimits_{\overbar{\gamma}_y \in \Upsilon_y,\overbar{\gamma}_z \in \Upsilon_z}  2^{-|\overbar{\gamma}_y|
	-|\overbar{\gamma}_z|+1} f_h^{\diamond}\big(\widehat{\mathsf{A}}_h(\overbar{\gamma}_y,\overbar{\gamma}_z)\big).
	\end{split}
\end{equation}
By taking the limits on both sides as $h\downarrow 0$, we get (recall $\mathbb{P}_{d,\mathrm{IIC}}^{(2)}(\cdot)$ in (\ref{iic_type2}))
\begin{equation}\label{3.18}
\begin{split}
		 \mathbb{P}_{d,\mathrm{IIC}}^{(2)}(\mathsf{F})=&  \sum\nolimits_{\overbar{\gamma}_y \in \Upsilon_y,\overbar{\gamma}_z \in \Upsilon_z}  2^{-|\overbar{\gamma}_y|
	-|\overbar{\gamma}_z|+1} \lim\limits_{h\downarrow 0} f_h^{\diamond}\big(\widehat{\mathsf{A}}_h(\overbar{\gamma}_y,\overbar{\gamma}_z)\big)\\
	\overset{(\ref*{v2_3.10})}{=}&   \sum\nolimits_{\overbar{\gamma}_y \in \Upsilon_y,\overbar{\gamma}_z \in \Upsilon_z}  2^{-|\overbar{\gamma}_y|
	-|\overbar{\gamma}_z|+1} f_0^{\diamond}\big(\widehat{\mathsf{A}}_0(\overbar{\gamma}_y,\overbar{\gamma}_z)\big).
\end{split}
\end{equation}
Note that the left-hand side of (\ref{3.15}) for $\diamond=\mathbf{N}$ (resp. $\diamond=\mathbf{T}$) matches the definition of Type (1) (resp. Type (4)) IIC defined in (\ref{iic_type1}) (resp. (\ref{iic_type4})). Thus, by putting (\ref{3.15}) and (\ref{3.18}) together, we confirm the existence and equivalence of (\ref{iic_type1}), (\ref{iic_type2}) and (\ref{iic_type4}).
\end{proof}

In what follows, we show that Type (3) IIC defined in (\ref{iic_type3}) exists and is equivalent to Type (2) IIC defined in (\ref{iic_type2}). Note that the existence has been proved by (\ref{3.15}). To prove the equivalence, we continue to use the notations introduced above. By taking $\diamond = \mathbf{x}$ in (\ref{3.15}), we have 
\begin{equation}\label{new3.17}
	 \lim\limits_{x \to \infty}\mathbb{P}\big(\mathsf{F} \mid \bm{0}\xleftrightarrow{\ge 0} x  \big)=  \sum\nolimits_{\overbar{\gamma}_y \in \Upsilon_y,\overbar{\gamma}_z \in \Upsilon_z}  2^{-|\overbar{\gamma}_y|
	-|\overbar{\gamma}_z|+1} f_0^{\mathbf{x}}\big(\widehat{\mathsf{A}}_0(\overbar{\gamma}_y,\overbar{\gamma}_z)\big). 
\end{equation}
Now we consider the case $h>0$. For the same reason as in proving (\ref{3.16}), we have 
\begin{equation}\label{new3.18}
	\begin{split}
		&\mathbb{P}\big( \widehat{\mathsf{F}}_h^{+}(\overbar{\gamma}_y) \cap \widehat{\mathsf{F}}_h^{-}(\overbar{\gamma}_z), \bm{0}\xleftrightarrow{\ge -h}x  \big)\\
		=  &2^{-|\overbar{\gamma}_y|
	-|\overbar{\gamma}_z|+1} \mathbb{P}\big(\widehat{\mathsf{A}}_h(\overbar{\gamma}_y,\overbar{\gamma}_z), \bm{0}\xleftrightarrow{[h]}\infty,x\xleftrightarrow{[h]}\infty \big)\\
	&+ \mathbb{P}\big( \widehat{\mathsf{F}}_h^{+}(\overbar{\gamma}_y) \cap \widehat{\mathsf{F}}_h^{-}(\overbar{\gamma}_z), \bm{0}\xleftrightarrow{\ge -h}x , \{\bm{0} \xleftrightarrow{\ge -h} \infty \}^c \big).
	\end{split}
\end{equation}
For the first term on the right-hand side, we have 
\begin{equation}\label{new3.19}
	\begin{split}
		&\big| \mathbb{P}\big(\widehat{\mathsf{A}}_h(\overbar{\gamma}_y,\overbar{\gamma}_z), \bm{0}\xleftrightarrow{[h]}\infty,x\xleftrightarrow{[h]}\infty \big) - \mathbb{P}\big(\widehat{\mathsf{A}}_h(\overbar{\gamma}_y,\overbar{\gamma}_z), \bm{0}\xleftrightarrow{[h]} x \big)\big| \\
		\le &\mathbb{P}\big(  \bm{0}\xleftrightarrow{[h]^c_\infty}x  \big) \le \mathbb{P}\big(  \bm{0}\xleftrightarrow{}x  \big) \overset{ x\to \infty }{\to } 0. 
	\end{split}
\end{equation}
In addition, by the isomorphism theorem, the second term on the right-hand side of (\ref{new3.18}) is bounded from above by 
\begin{equation}\label{new3.20}
	\begin{split}
		\tfrac{1}{2}\mathbb{P}\big(  \bm{0}\xleftrightarrow{[h]^c_\infty}x  \big)\le \tfrac{1}{2}\mathbb{P}\big(  \bm{0}\xleftrightarrow{}x  \big) \overset{ x\to \infty }{\to } 0. 
	\end{split}
\end{equation}
Plugging (\ref{new3.19}) and (\ref{new3.20}) into (\ref{new3.18}), we get 
\begin{equation}\label{new3.21}
\begin{split}
		&\mathbb{P}\big( \widehat{\mathsf{F}}_h^{+}(\overbar{\gamma}_y) \cap \widehat{\mathsf{F}}_h^{-}(\overbar{\gamma}_z), \bm{0}\xleftrightarrow{\ge -h}x  \big)\\
	=&2^{-|\overbar{\gamma}_y|
	-|\overbar{\gamma}_z|+1} \mathbb{P}\big(\widehat{\mathsf{A}}_h(\overbar{\gamma}_y,\overbar{\gamma}_z), \bm{0}\xleftrightarrow{[h]} x  \big) +o_{x}(1).  
\end{split}
\end{equation}
Dividing both sides of (\ref{new3.21}) by $\mathbb{P}(\bm{0}\xleftrightarrow{\ge -h}x)$ (which is at least $ch^2$ according to (\ref{finish2.15})) and taking the limits as $x\to \infty$, we obtain 
\begin{equation}\label{submit3.24}
	\begin{split}
		& \lim\limits_{x\to \infty} \mathbb{P}\big( \widehat{\mathsf{F}}_h^{+}(\overbar{\gamma}_y) \cap \widehat{\mathsf{F}}_h^{-}(\overbar{\gamma}_z) \mid \bm{0}\xleftrightarrow{\ge -h}x  \big)\\
		=&\lim\limits_{x\to \infty} 2^{-|\overbar{\gamma}_y|
	-|\overbar{\gamma}_z|+1} \mathbb{P}\big(\widehat{\mathsf{A}}_h(\overbar{\gamma}_y,\overbar{\gamma}_z) \mid \bm{0}\xleftrightarrow{[h]} x\big)\cdot \frac{\mathbb{P}(\bm{0}\xleftrightarrow{[h]}x)}{\mathbb{P}(\bm{0}\xleftrightarrow{\ge -h}x)}\\
	\overset{(\ref*{submit_3.10})}{=}& \lim\limits_{x\to \infty} 2^{-|\overbar{\gamma}_y|
	-|\overbar{\gamma}_z|+1} f_h^{\mathbf{x}}\big(\widehat{\mathsf{A}}_h(\overbar{\gamma}_y,\overbar{\gamma}_z)\big),
		\end{split}
\end{equation}
where in the last line we also used $\mathbb{P}(\bm{0}\xleftrightarrow{\ge -h}x)\overset{}{=}\mathbb{P}(\bm{0}\xleftrightarrow{[h]}x)-\frac{1}{2}\mathbb{P}(\bm{0}\xleftrightarrow{[h]^c_\infty}x)= \mathbb{P}(\bm{0}\xleftrightarrow{[h]}x)+o_x(1)$ (by the isomorphism theorem). Combined with (\ref{v2_3.8}), it implies  
\begin{equation}\label{finish3.26}
\begin{split}
	 \lim\limits_{x \to \infty}\mathbb{P}\big(\mathsf{F} \mid \bm{0}\xleftrightarrow{\ge -h} x  \big)=  \sum\nolimits_{\overbar{\gamma}_y \in \Upsilon_y,\overbar{\gamma}_z \in \Upsilon_z }2^{-|\overbar{\gamma}_y|
	-|\overbar{\gamma}_z|+1} f_h^{\mathbf{x}}\big(\widehat{\mathsf{A}}_h(\overbar{\gamma}_y,\overbar{\gamma}_z)\big). 
\end{split}
\end{equation}
By taking the limits as $h\downarrow 0$ on both sides of (\ref{finish3.26}) (where by (\ref{v2_3.10}) the limit of the right-hand side equals to the right-hand side of (\ref{new3.17})), we derive 
\begin{equation}\label{new3.25}
\lim\limits_{h\downarrow 0}	\lim\limits_{x \to \infty}\mathbb{P}\big(\mathsf{F} \mid \bm{0}\xleftrightarrow{\ge -h} x  \big)= \lim\limits_{x \to \infty}\mathbb{P}\big(\mathsf{F} \mid \bm{0}\xleftrightarrow{\ge 0} x  \big). 
\end{equation}
To relate the left-hand side with Type (2) IIC, we need the following lemma. 
\begin{lemma}\label{lemma3.3}
For any $d\ge 3$, $h>0$ and any increasing cylinder event $\mathsf{F}$, 
		\begin{equation}
			\lim_{x\to \infty} \mathbb{P}\big( \mathsf{F},  \bm{0}\xleftrightarrow{\ge -h} x \big)  = 	\mathbb{P}\big( \mathsf{F},  \bm{0}\xleftrightarrow{\ge -h} \infty \big)  \mathbb{P}\big(\bm{0}\xleftrightarrow{\ge -h} \infty \big).
		\end{equation}
\end{lemma}

To prove Lemma \ref{lemma3.3}, we need the following comparison lemma, which is also useful for later proofs. For simplicity, we use $\diamond \ge M$ to denote $N\geq M$, $T\geq M$, and $|x| \geq M$ for $\diamond = \mathbf{N}, \mathbf{T}, \mathbf{x}$ respectively.

\begin{lemma}\label{lemma_compare_hdiamond}
	For any $d\ge 3$, there exist $C,c>0$ such that for any event $\mathsf{F}$, $h>0$, $M\ge 10$ and $\diamond \in \{\mathbf{N},\mathbf{x},\mathbf{T}\}$ with $\diamond\ge M$, 
	\begin{equation}
	0\le 	\mathbb{P}\big(\mathsf{F},  \mathsf{H}_h^{\diamond}  \big)- \mathbb{P}\big(\mathsf{F}, \mathsf{H}_h^{\diamond} ,\bm{0}\xleftrightarrow{[h]} \infty   \big)\le Ce^{-ch^2M^{\frac{1}{d}}}.
	\end{equation}
\end{lemma}
\begin{proof}
	The inequality on the left-hand side holds trivially. For the right-hand side, note that $\mathsf{H}_h^{\diamond} \cap \{\bm{0}\xleftrightarrow{[h] } \infty\}^c\subset \mathsf{H}_0^{\diamond}\cap \{ \mathcal{C}(\bm{0})\cap (\cup \widetilde{\mathcal{I}}^{\frac{1}{2}h^2})=\emptyset\}$. Moreover, when $\diamond\ge  M$, $\mathsf{H}_0^{\diamond}$ implies $\{\bm{0}\xleftrightarrow{} \partial B(cM) \}$ when $\diamond \in \{\mathbf{N},\mathbf{x}\}$, and implies $\{\bm{0}\xleftrightarrow{} \partial B(cM^{\frac{1}{d-2}}) \}$ when $\diamond = \mathbf{T} $ (since $\mathrm{cap}(\widetilde{B}(cM^{\frac{1}{d-2}})) \le M$). Therefore, we have 
	\begin{equation}\label{finally_3.28}
		\begin{split}
			&\mathbb{P}\big(\mathsf{F},  \mathsf{H}_h^{\diamond}  \big)- \mathbb{P}\big(\mathsf{F}, \mathsf{H}_h^{\diamond} ,\bm{0}\xleftrightarrow{[h] } \infty   \big)\\
			\le &\mathbb{P}\big(\bm{0}\xleftrightarrow{} \partial B(cM^{\frac{1}{d-2}}), \mathcal{C}(\bm{0})\cap (\cup \widetilde{\mathcal{I}}^{\frac{1}{2}h^2})=\emptyset \big) \\
			\lesssim  &  e^{-ch^2M^{\frac{1}{d-2}}\ln^{-1}(M)}\lesssim  e^{-c'h^2M^{\frac{1}{d}} },
		\end{split}
	\end{equation}
	where the second transition follows from (\ref{2.30}) and the fact that the capacity of any connected set on $\widetilde{\mathbb{Z}}^d$ connecting $\bm{0}$ and $\partial B(m)$ is at least $c(\frac{m}{\ln(m)}\cdot \mathbbm{1}_{d=3}+ m\cdot \mathbbm{1}_{d\ge 4})$ (see e.g., \cite[Lemma 14]{ding2020percolation}). Now we complete the proof of this lemma.
	\end{proof}

	By applying the isomorphism theorem and using the same arguments as in the proof of (\ref{finally_3.28}), we also have 
	\begin{equation}\label{apply_lemme_compare}
		0\le 	\mathbb{P}\big(\mathsf{F},  \widetilde{\mathsf{H}}_h^{\diamond}  \big)- \mathbb{P}\big(\mathsf{F}, \widetilde{\mathsf{H}}_h^{\diamond} ,\bm{0}\xleftrightarrow{\ge -h} \infty   \big)\le Ce^{-ch^2M^{\frac{1}{d}}}.
	\end{equation}

	Now we are ready to prove Lemma \ref{lemma3.3}. 
\begin{proof}[Proof of Lemma \ref{lemma3.3}]
We first bound the probability $ \mathbb{P}\big( \mathsf{F},  \bm{0}\xleftrightarrow{\ge -h} x \big)$ from above. We denote $B_y(n):=y+ B(n)$ for $y\in \mathbb{Z}^d$ and $n>0$. By Lemma \ref{lemma_decoupling_ineq}, we have 	\begin{equation}\label{newadd3.28}
		\begin{split}
			&\mathbb{P}\big( \mathsf{F},  \bm{0}\xleftrightarrow{\ge -h} x \big) \\
			 \le  &\mathbb{P}\big( \mathsf{F},  \bm{0}\xleftrightarrow{\ge -h} \partial B( |x|^{1/2}), x\xleftrightarrow{\ge -h} \partial B_x( |x|^{1/2})  \big)\\
			\le &\mathbb{P}\big( \mathsf{F},  \bm{0}\xleftrightarrow{\ge -h} \partial B( |x|^{1/2})  \big)\mathbb{P}\big( x\xleftrightarrow{\ge -h- |x|^{-\frac{1}{8}}} \partial B_x(|x|^{1/2})  \big)
			+ Ce^{-c|x|^{\frac{2d-5}{4}}}.
		\end{split}
	\end{equation}
	Moreover, by (\ref{apply_lemme_compare}) with $\diamond=\mathbf{N}$ and $N=|x|^{1/2}$, we have 
	\begin{equation}
		\mathbb{P}\big( \mathsf{F},  \bm{0}\xleftrightarrow{\ge -h} \partial B( |x|^{1/2})  \big) \le \mathbb{P}\big( \mathsf{F},  \bm{0}\xleftrightarrow{\ge -h} \infty  \big) + Ce^{-ch^2|x|^{\frac{1}{2d}}}, 
	\end{equation}
		\begin{equation}\label{newfinish_3.34}
	\mathbb{P}\big( x\xleftrightarrow{\ge -h- |x|^{-\frac{1}{8}}} \partial B_x(|x|^{1/2})  \big) \le \mathbb{P}\big(  \bm{0}\xleftrightarrow{\ge -h- |x|^{-\frac{1}{8}}} \infty  \big) + Ce^{-ch^2|x|^{\frac{1}{2d}}}. 
	\end{equation}
	Furthermore, since $\mathbb{P}\big(  \bm{0}\xleftrightarrow{\ge -h} \infty  \big)$ is continuous in $h$ (by (\ref{usenew2.23})), one has 
	\begin{equation}\label{newfinish_3.34_5}
		\limsup_{x\to \infty} \mathbb{P}\big(  \bm{0}\xleftrightarrow{\ge -h- |x|^{-\frac{1}{8}}} \infty  \big) = \mathbb{P}\big(  \bm{0}\xleftrightarrow{\ge -h} \infty  \big). 
	\end{equation}
Putting (\ref{newadd3.28})--(\ref{newfinish_3.34_5}) together, we obtain 
	\begin{equation}\label{submit_334}
		\limsup_{x\to \infty}\mathbb{P}\big( \mathsf{F},  \bm{0}\xleftrightarrow{\ge -h} x  \big)
		 \le  \mathbb{P}\big( \mathsf{F},  \bm{0}\xleftrightarrow{\ge -h} \infty \big)\mathbb{P}\big( \bm{0}\xleftrightarrow{\ge -h} \infty \big).	\end{equation}
	Meanwhile, it follows from (\ref{use2.37}) that 
	\begin{equation}\label{submit_335}
		\liminf_{x\to \infty}\mathbb{P}\big( \mathsf{F},  \bm{0}\xleftrightarrow{\ge -h} x  \big)
		 \ge  \mathbb{P}\big( \mathsf{F},  \bm{0}\xleftrightarrow{\ge -h} \infty \big)\mathbb{P}\big( \bm{0}\xleftrightarrow{\ge -h} \infty \big).
	\end{equation}
	By (\ref{submit_334}) and (\ref{submit_335}), we conclude the proof of this lemma. 
	\end{proof}

For any $h>0$, it follows from Lemma \ref{lemma3.3} that 
\begin{equation}\label{new3.31}
	\begin{split}
		\lim\limits_{x\to \infty }\mathbb{P}\big( \mathsf{F} \mid  \bm{0}\xleftrightarrow{\ge -h} x \big)=&   \frac{\mathbb{P}\big( \mathsf{F},  \bm{0}\xleftrightarrow{\ge -h} \infty \big)  \mathbb{P}\big(\bm{0}\xleftrightarrow{\ge -h} \infty \big)}{\big[  \mathbb{P}\big(\bm{0}\xleftrightarrow{\ge -h} \infty \big)\big]^2 }
		=  \mathbb{P}\big( \mathsf{F} \mid \bm{0}\xleftrightarrow{\ge -h} \infty \big).
	\end{split}
\end{equation}
Plugging (\ref{new3.31}) into (\ref{new3.25}), we obtain that Type (3) IIC defined in (\ref{iic_type3}) is equivalent to Type (2) IIC defined in (\ref{iic_type2}). To sum up, we have established the existence and equivalence of the IICs in (\ref{iic_type1}), (\ref{iic_type2}), (\ref{iic_type3}) and (\ref{iic_type4}) assuming that Proposition \ref{lemma_converge_loop} holds.

Next, we show that $\mathscr{C}^{\ge 0}$ (defined in Theorem \ref{thm_iic}) is almost surely infinite. For any $M>0$, since $\mathcal{C}^{\ge 0}(\bm{0})\subset \widetilde{B}(M)$ is incompatible with the event $\bm{0}\xleftrightarrow{\ge 0} \partial B(N)$ when $N>M$, we have 
\begin{equation*}
 \mathbb{P}_{d,\mathrm{IIC}}(\mathscr{C}^{\ge 0}\subset \widetilde{B}(M) ) =  \lim\limits_{N\to \infty}	\mathbb{P}\big(\mathcal{C}^{\ge 0}(\bm{0})\subset \widetilde{B}(M) \mid \bm{0}\xleftrightarrow{\ge 0} \partial B(N) \big) =0,
\end{equation*}
which implies that $\mathscr{C}^{\ge 0}$ is almost surely infinite.

It remains to establish that $\mathscr{C}^{\ge 0}$ is almost surely one-ended. To this end, note that for any $h>0$, if $\mathcal{C}^{\ge -h}(\bm{0})$ is infinite and not one-ended, then there exist $N\in \mathbb{N}$ and $x_1\neq x_2\in D_N:=\{x\in \mathbb{Z}^d\setminus B(N): \exists x'\in B(N)\ \text{such that}\ x\sim x'\}$ such that $x_1$ and $x_2$ are contained in two distinct infinite clusters of $\mathcal{C}^{\ge -h}(\bm{0})\cap ( \cup_{\{x,x'\}\cap B(N)= \emptyset} I_{\{x,x'\}})$ (we denote this event by $\mathsf{F}_{h,N}(x_1,x_2)$). This observation implies that 
\begin{equation}\label{submit3.38}
	\begin{split}
		& \mathbb{P}_{d,\mathrm{IIC}}(\mathscr{C}^{\ge 0}\ \text{is not one-ended}) \\
\le & \sum\nolimits_{N\in \mathbb{N},x_1\neq x_2\in D_N} \limsup \limits_{h \downarrow 0} \mathbb{P}\big( \mathsf{F}_{h,N}(x_1,x_2) \mid \bm{0}\xleftrightarrow{\ge -h} \infty\big).
	\end{split}
\end{equation}
In what follows, we use a proof by contradiction to show that 
\begin{equation}\label{submit3.39}
	\mathbb{P}\big(\mathsf{F}_{h,N}(x_1,x_2) \big)=0 ,\ \forall h>0, N\in \mathbb{N}, x_1\neq x_2\in D_N. 
\end{equation}
Assume that $\mathbb{P}\big(\mathsf{F}_{h,N}(x_1,x_2) \big)\ge a$ for some positive number $a>0$. Consequently, there exists a sufficiently large $\lambda=\lambda(a,N)>0$ such that 
\begin{equation}\label{submit340}
\begin{split}
		\mathbb{P}\big(\widehat{\mathsf{F}}_{h,N}^\lambda(x_1,x_2) \big):=&\mathbb{P}\big(\mathsf{F}_{h,N}(x_1,x_2)\cap \{\max\nolimits_{y\in D_N} \widetilde{\phi}_y\le \lambda\} \big)  \\
		\ge  &\mathbb{P}\big(\mathsf{F}_{h,N}(x_1,x_2) \big) -\mathbb{P}\big( \exists y\in D_N\ \text{such that}\ \widetilde{\phi}_y> \lambda\big)   >\tfrac{1}{2}a.
\end{split}
	\end{equation}
 Note that $\widehat{\mathsf{F}}_{h,N}^\lambda(x_1,x_2)$ is measurable with respect to $\widehat{\mathcal{F}}_N$, i.e., the $\sigma$-field generated by $\widetilde{\phi}_v$ for $v\in \cup_{\{x,x'\}\cap B(N)=\emptyset}I_{\{x,x'\}}$. Moreover, when $\widehat{\mathsf{F}}_{h,N}^\lambda(x_1,x_2)$ happens, since the boundary conditions on $D_N$ are at most $\lambda$, we know that for any $y\in B(N)$ and $y'\sim y$, the conditional probability of $\{y\xleftrightarrow{\ge -h} y'\}^c$ given $\widehat{\mathcal{F}}_N$ is bounded from below by some $\zeta=\zeta(\lambda,N)\in (0,1)$. Thus, by the FKG inequality and (\ref{submit340}), we have
  \begin{equation}\label{submit_3.41}
 	\begin{split}
 		&\mathbb{P}\big(\widehat{\mathsf{F}}_{h,N}^\lambda(x_1,x_2) , \cap_{y\sim y'\in B(N)} \{y\xleftrightarrow{\ge -h} y'\}^c \big) \\
 		=&\mathbb{E}\big[ \mathbbm{1}_{\widehat{\mathsf{F}}_{h,N}^\lambda(x_1,x_2)}\cdot \mathbb{P}\big( \cap_{y\in B(N),y'\sim y} \{y\xleftrightarrow{\ge -h} y'\}^c \mid  \widehat{\mathcal{F}}_N \big) \big]\\
 		\overset{\text{(FKG)}}{\ge } & \mathbb{E}\big[ \mathbbm{1}_{\widehat{\mathsf{F}}_{h,N}^\lambda(x_1,x_2)}\cdot \prod\nolimits_{y\in B(N),y'\sim y} \mathbb{P}\big(\{y\xleftrightarrow{\ge -h} y'\}^c \mid  \widehat{\mathcal{F}}_N \big) \big] \\
 		\ge &\mathbb{P}\big( \widehat{\mathsf{F}}_{h,N}^\lambda(x_1,x_2) \big) \cdot  \zeta^{CN^{d}} 
 	\overset{(\ref*{submit340})}{\ge }   \tfrac{1}{2}a\zeta^{CN^{2d}} >0. 
 	\end{split}
 \end{equation}
 However, the event on the left-hand side of (\ref{submit_3.41}) implies the occurrence of two distinct infinite clusters in $\widetilde{E}^{\ge -h}$, which happens with probability $0$. This leads to a contradiction with (\ref{submit_3.41}). To sum up, we obtain (\ref{submit3.39}).

Plugging (\ref{submit3.39}) into (\ref{submit3.38}), we confirm that $\mathscr{C}^{\ge 0}$ is almost surely one-ended. In conclusion, we complete the proof of Theorem \ref{thm_iic} assuming Proposition \ref{lemma_converge_loop}.  
  \qed

\subsection{Approximation of events in IICs}

To facilitate the proof of Proposition \ref{lemma_converge_loop} (the precise motivation will be provided in the overview at the beginning of the next subsection), we need to approximate the probabilities of the events involved in the definition of IICs, namely, $\mathsf{A}_h \cap \mathsf{H}_h^{\diamond}$ for $\diamond\in\{\mathbf{N},\mathbf{x},\mathbf{T}\}$ (recall $\mathsf{A}_h$ in Definition \ref{def_admissible} and recall $\mathsf{H}_h^{\diamond}$ in (\ref{notation_H})--(\ref{notation_H3})). A key step (see Lemma \ref{lemma3.4}) is to show that given these events, the occurrence of a large crossing loop remains unlikely.

For $\diamond\in \{\mathbf{N},\mathbf{T}\}$, since $\mathsf{H}_h^{\diamond}\supset \{\bm{0}\xleftrightarrow{[h]} \infty \}$, one has
\begin{equation}\label{renew3.32}
	\mathbb{P}\big( \mathsf{H}_h^{\diamond}\big) \ge \mathbb{P}\big(\bm{0}\xleftrightarrow{[h]} \infty \big) \overset{(\ref*{232plus})}{=}  h\land 1.
\end{equation}
Moreover, for $\diamond=\mathbf{x}$, one has 
\begin{equation}\label{renew3.33}
 \mathbb{P}\big( \mathsf{H}_h^{\mathbf{x}}\big)  \ge \mathbb{P}\big(\bm{0}\xleftrightarrow{[h]}x , \bm{0} \xleftrightarrow{[h]} \infty \big) \overset{(\ref*{newsubmit_224})}{\gtrsim  } (h\land 1)^2.
\end{equation}
Recall $\widetilde{\mathcal{L}}_{1/2}[n,m]$ in (\ref{def_L[n,m]}). For brevity, we denote $A_{ \mathbf{N}}:= B(N)$ and $A_{ \mathbf{x}}:=\{x\}$.

\begin{lemma}\label{lemma3.4}
	For any $d\ge 3$ with $d\neq 6$, there exists $C>0$ such that for any $u\ge C$, $n\ge 1$, $m\ge un^{1\boxdot \frac{d-4}{2}}$, $h\in [0,m^{-d^{5}}]$, and any $\diamond\in \{\mathbf{N},\mathbf{x},\mathbf{T}\}$ with $\diamond \ge m^{d^5}$, 
		\begin{equation}\label{3.32}
	\mathbb{P}\big(\mathsf{H}_h^{\diamond}[n,m] \big):=	\mathbb{P}\big(\mathsf{H}_h^{\diamond} , \widetilde{ \mathcal{L}}_{1/2}[n,m]\neq \emptyset  \big)\lesssim u^{-1}\mathbb{P}\big(\mathsf{H}_h^{\diamond}  \big). 
	\end{equation} 
\end{lemma}
\begin{proof}
	We claim that for any $\diamond\in \{\mathbf{N},\mathbf{x},\mathbf{T}\}$, 
\begin{equation}\label{renew3.35}
	\mathbb{P}\big( \mathsf{F}^{\diamond}  \big) := \mathbb{P}\big(\mathsf{H}_h^{\diamond}, B(m^{d^2})\cap (\cup \widetilde{\mathcal{I}}^{\frac{1}{2}h^2})\neq \emptyset  \big)\lesssim  u^{-1}\mathbb{P}\big(\mathsf{H}_h^{\diamond}  \big).\end{equation}
	 We demonstrate this claim separately for $\diamond \in \{ \mathbf{N},\mathbf{T}\}$ and $\diamond =\mathbf{x}$.

	 \textbf{When $\diamond \in \{ \mathbf{N},\mathbf{T}\}$.} By (\ref{2.30}), (\ref{renew3.32}) and $h\le m^{-d^5}$, we have 
	 \begin{equation}\label{renew3.36}
	 		\mathbb{P}\big(  \mathsf{F}^{\diamond } \big) \lesssim h^2m^{d^2(d-2)} \lesssim u^{-1 }\mathbb{P}\big(\mathsf{H}_h^{\diamond}  \big).	 \end{equation}

	 \textbf{When $\diamond =\mathbf{x}$.} By (\ref{one_arm_low}), (\ref{one_arm_high}), (\ref{2.30}) and (\ref{renew3.33}), one has 
\begin{equation}\label{renew3.37}
\begin{split}
	\mathbb{P}\big(  \mathsf{F}^{\mathbf{x}}, x\xleftrightarrow{} \partial B_x(m^{d^4}) \big)\le &\mathbb{P}\big(B(m^{d^2})\cap (\cup \widetilde{\mathcal{I}}^{\frac{1}{2}h^2})\neq \emptyset   \big)\cdot \theta_d(m^{d^4})\\
	\overset{(\ref*{one_arm_low}), (\ref*{one_arm_high}),(\ref*{2.30})}{\lesssim} &h^2m^{d^2(d-2)} \cdot m^{-d^4[(\frac{d}{2}-1)\boxdot 2] }\\
	\overset{}{\lesssim} & h^2 m^{-d^3}  \overset{(\ref*{renew3.33})}{\lesssim} u^{-1 }\mathbb{P}\big(\mathsf{H}_h^{\mathbf{x}}  \big).
\end{split}	
\end{equation}
Moreover, on $\mathsf{F}^{\mathbf{x}} \cap \{x\xleftrightarrow{} \partial B_x(m^{d^4})\}^c$, we know that $\{B(m^{d^2})\cap (\cup \widetilde{\mathcal{I}}^{\frac{1}{2}h^2})\neq \emptyset\}$ and $\{B_x(m^{d^4})\cap (\cup \widetilde{\mathcal{I}}^{\frac{1}{2}h^2})\neq \emptyset\}$ both happen. If these two events are certified by the same trajectory in $\widetilde{\mathcal{I}}^{\frac{1}{2}h^2}$, then by $B(m^{d^2})\subset B(m^{d^4})$ and (\ref{newadd_2.26}), this probability is bounded from above by  
	 \begin{equation}
	 	Ch^2 |x|^{2-d}\cdot m^{d^4(2d-4)} \overset{|x|\ge m^{d^5}}{\lesssim}  h^2m^{-d^3}  \overset{(\ref*{renew3.33})}{\lesssim} u^{-1 }\mathbb{P}\big(\mathsf{H}_h^{\mathbf{x}}  \big).
	 \end{equation}
 Otherwise, there are two different trajectories intersecting $B(m^{d^2})$ and $ B_x(m^{d^4})$ respectively. By (\ref{2.30}), this happens with probability at most
	 \begin{equation}\label{renew3.39}
	 	Ch^2m^{d^2(d-2)}\cdot h^2m^{d^4(d-2)} \overset{h\le m^{-d^5}}{\lesssim} h^2m^{-d^5}  \overset{(\ref*{renew3.33})}{\lesssim} u^{-1 }\mathbb{P}\big(\mathsf{H}_h^{\mathbf{x}}  \big).
	 \end{equation}
	Combining (\ref{renew3.37})--(\ref{renew3.39}), we get 
	 \begin{equation}\label{renew3.40}
	 	\mathbb{P}\big(  \mathsf{F}^{\mathbf{x}} \big)   \lesssim u^{-1 }\mathbb{P}\big(\mathsf{H}_h^{\mathbf{x}}  \big). 
	 \end{equation}
   By (\ref{renew3.36}) and (\ref{renew3.40}), we confirm the claim in (\ref{renew3.35}).

With (\ref{renew3.35}) at hand, we next prove (\ref{3.32}) for $\diamond \in \{\mathbf{N},\mathbf{x}\}$. It follows from (\ref{2.26}) and $\mathsf{H}_0^{\diamond}\subset \mathsf{H}_h^{\diamond}$ that 
	\begin{equation}\label{renew3.41}
	\begin{split}
				 \mathbb{P}\big(\mathsf{H}_0^{\diamond}[n,m]  \big) \lesssim  & n^{(\frac{d}{2}-1)\boxdot (d-4)}m^{-[(\frac{d}{2}-1)\boxdot 2]}\mathbb{P}\big(\mathsf{H}_0^{\diamond}  \big)
				 \lesssim u^{-1} \mathbb{P}\big(\mathsf{H}_h^{\diamond}  \big).
				 \end{split}
	\end{equation}
		 On $\{\mathsf{H}_h^{\diamond} , \widetilde{ \mathcal{L}}_{1/2}[n,m]\neq \emptyset\}\cap (\mathsf{F}^{\diamond})^c \cap (\mathsf{H}_0^{\diamond})^c$, we have the following observations:
		 \begin{enumerate}
		 	\item [(a)]   $\mathcal{C}(A_\diamond):= \{v\in \widetilde{\mathbb{Z}}^d: v \xleftrightarrow{} A_\diamond \}$ and $\mathcal{C}(\bm{0})$ are disjoint and both intersect $\widetilde{\mathcal{I}}^{\frac{1}{2}h^2}$;

		 	\item [(b)] $\cup \widetilde{\mathcal{I}}^{\frac{1}{2}h^2} \subset [\widetilde{B}(m^{d^2})]^c$ (this must happen on the event $\mathsf{H}_h^{\diamond}\cap  (\mathsf{F}^{\diamond})^c$);

		 	\item [(c)]   $\mathcal{C}(\bm{0})\cap \partial B(m^{d^2})\neq \emptyset$ (by Observations (a) and (b)).

		 \end{enumerate}
		 Building on these observations, we estimate the probability of $\{\mathsf{H}_h^{\diamond} , \widetilde{ \mathcal{L}}_{1/2}[n,m]\neq \emptyset\}\cap (\mathsf{F}^{\diamond})^c\cap (\mathsf{H}_0^{\diamond})^c$ separately in the following cases.

		 \textbf{Case 1:} when $\mathcal{C}(A_\diamond)\cap \partial B(Cm^{1\boxdot \frac{d-4}{2}})\neq \emptyset$ (i.e., $A_\diamond \xleftrightarrow{} B(Cm^{1\boxdot \frac{d-4}{2}})$). In this case, $\bm{0}\xleftrightarrow{} \partial B(m^{d^2})$ and $A_\diamond \xleftrightarrow{} B(Cm^{1\boxdot \frac{d-4}{2}})$ happen disjointly (since they are certified by two different clusters $\mathcal{C}(\bm{0})$ and $\mathcal{C}(A_\diamond)$ respectively; recalling Observations (a) and (c)), the probability for this case, according to the BKR inequality (see \cite[Corollary 3.4]{cai2024high}), is bounded from above by  
      \begin{equation}
      	\begin{split}
      		&\mathbb{P}\big(\bm{0}\xleftrightarrow{} \partial B(m^{d^2})\big) \mathbb{P}\big(A_\diamond \xleftrightarrow{} B(Cm^{1\boxdot \frac{d-4}{2}})\big) \\
      		= & \frac{\theta_d(m^{d^2})}{\theta_d(Cm^{ 1\boxdot  \frac{d-4}{2}})}\cdot  \mathbb{P}\big(\bm{0}\xleftrightarrow{} \partial B(Cm^{1\boxdot \frac{d-4}{2}})\big) \mathbb{P}\big(A_\diamond \xleftrightarrow{} B(Cm^{1\boxdot \frac{d-4}{2}})\big)\\
      		\overset{(\ref*{QM_ineq_1})}{\lesssim } &  \frac{\theta_d(m^{d^2})}{\theta_d(Cm^{1\boxdot \frac{d-4}{2}})}\cdot  m^{(1\boxdot \frac{d-4}{2})\cdot [0\boxdot (d-6) ]}  \mathbb{P}\big(\mathsf{H}_0^{\diamond} \big)\lesssim  u^{-1} \mathbb{P}\big(\mathsf{H}_h^{\diamond}  \big).
      		     	\end{split}
      \end{equation}

		\textbf{Case 2:} when $\mathcal{C}(A_\diamond)\cap \partial B(Cm^{1\boxdot \frac{d-4}{2}})= \emptyset$. Suppose that $\mathcal{C}(A_\diamond)$ and $\widetilde{\mathcal{I}}^{\frac{1}{2}h^2}$ are sampled. Then the following event (denoted by $\mathsf{G}$) happens: $\mathcal{C}(A_\diamond)$ and $\cup \widetilde{\mathcal{I}}^{\frac{1}{2}h^2}$ intersect each other (by Observation (a)), and are both contained in $[\widetilde{B}(Cm^{1\boxdot \frac{d-4}{2}})]^c$ (by the assumption $\mathcal{C}(A_\diamond)\cap \partial B(Cm^{1\boxdot \frac{d-4}{2}})= \emptyset$ and Observation (b)). Hence, $\mathcal{C}(A_\diamond)$ does not contain any loop in $\widetilde{\mathcal{L}}_{1/2}[n,m]$. Moreover, by Observation (a), we know that $\bm{0}$ is connected to $\cup \widetilde{\mathcal{I}}^{\frac{1}{2}h^2}$ by $\cup \widetilde{\mathcal{L}}_{1/2}$ without using any loop included in $\mathcal{C}(A_\diamond)$. To sum up, the probability of this case is at most 
	\begin{equation}\label{3.36}
		\begin{split}
			&\mathbb{E}\Big[\mathbbm{1}_{\mathsf{G}} \cdot \mathbb{P}\big( \bm{0}\xleftrightarrow{(\mathcal{C}(A_\diamond))} \cup \widetilde{\mathcal{I}}^{\frac{1}{2}h^2}, \widetilde{ \mathcal{L}}_{1/2}[n,m]\neq \emptyset   \mid \mathcal{C}(A_\diamond), \widetilde{\mathcal{I}}^{\frac{1}{2}h^2} \big)  \Big]\\
			\overset{ (\ref*{2.26}) }{\lesssim}  &n^{(\frac{d}{2}-1)\boxdot (d-4)}m^{-[(\frac{d}{2}-1) \boxdot 2]} \mathbb{E}\Big[\mathbbm{1}_{\mathsf{G}} \cdot \mathbb{P}\big( \bm{0}\xleftrightarrow{(\mathcal{C}(A_\diamond))} \cup \widetilde{\mathcal{I}}^{\frac{1}{2}h^2} \mid \mathcal{C}(A_\diamond), \widetilde{\mathcal{I}}^{\frac{1}{2}h^2} \big)  \Big]\\
			\lesssim  & u^{-1} \mathbb{P}\big(\mathsf{H}_h^{\diamond}  \big), 
		\end{split}
	\end{equation}	
	where in the first inequality we used $\mathsf{G}\subset \{ \mathcal{C}(A_\diamond)\cup (\cup \widetilde{\mathcal{I}}^{\frac{1}{2}h^2})\subset  [\widetilde{B}(Cm^{1\boxdot \frac{d-4}{2}})]^c\}$ to satisfy the conditions of (\ref{2.26}). Putting (\ref{renew3.35}) and (\ref{renew3.41})--(\ref{3.36}) together, we obtain (\ref{3.32}) for $\diamond \in \{ \mathbf{N},\mathbf{x}\}$.

Now we prove (\ref{3.32}) for $\diamond = \mathbf{T}$. By taking the limit in (\ref{3.32}) for $\diamond=\mathbf{N}$ as $N\to \infty$ and using $\{\bm{0}\xleftrightarrow{[h]} \infty\}\subset \mathsf{H}_h^{\mathbf{T}}$, we have  
\begin{equation}\label{newuse3.37}
	\mathbb{P}\big( \bm{0}\xleftrightarrow{[h]} \infty, \widetilde{ \mathcal{L}}_{1/2}[n,m]\neq \emptyset  \big)\lesssim  u^{-1}\mathbb{P}\big(\mathsf{H}_h^{\mathbf{T}} \big). 
\end{equation}
Therefore, since $\{\mathsf{H}_h^{\mathbf{T}} , \widetilde{ \mathcal{L}}_{1/2}[n,m]\neq \emptyset  \}\cap \{\bm{0}\xleftrightarrow{[h]} \infty \}^c\subset \mathsf{H}_0^{\mathbf{T}}[n,m]$, it suffices to prove 
\begin{equation}\label{submit355}
	\mathbb{P}\big(\mathsf{H}_0^{\mathbf{T}}[n,m] \big) \lesssim u^{-1}\mathbb{P}\big(\mathsf{H}_h^{\mathbf{T}} \big). 
\end{equation}
In what follows, we establish (\ref{submit355}) separately in the cases $3\le d\le 5$ and $d\ge 7$.

\textbf{When $3\le d\le 5$.} Noting that $\mathsf{H}_h^{\mathbf{T}} \subset \{\bm{0} \xleftrightarrow{[h]} \partial B(cT^{\frac{1}{d-2}})\}$ for some constant $c>0$ (since $\mathrm{cap}(B(cT^{\frac{1}{d-2}}))< T$), we have 
 \begin{equation}\label{3.37}
	\begin{split}
		\mathbb{P}\big(\mathsf{H}_0^{\mathbf{T}}[n,m]  \big)
		\le &\mathbb{P}\big( \bm{0}\xleftrightarrow{}\partial B(cT^{\frac{1}{d-2}})) , \widetilde{ \mathcal{L}}_{1/2}[n,m]\neq \emptyset  \big) \\
		\overset{(\ref*{3.32})\ \text{for}\ \diamond =\mathbf{N},N=cT^{\frac{1}{d-2}}}{\lesssim} &  u^{-1}\mathbb{P}\big( \bm{0}\xleftrightarrow{}\partial B(cT^{\frac{1}{d-2}}) \big)\\
		\overset{(\ref*{one_arm_low}),(\ref*{iic_type4new}),(\ref*{iso})}{\asymp } &u^{-1} \mathbb{P}(\mathsf{H}_0^{\mathbf{T}}) \overset{\mathsf{H}_0^{\mathbf{T}}\subset \mathsf{H}_h^{\mathbf{T}}}{\le }u^{-1} \mathbb{P}(\mathsf{H}_h^{\mathbf{T}}).			\end{split}
\end{equation}

\textbf{When $d\ge 7$.} By (\ref{crossing_high}) and the inclusion $\mathsf{H}_0^{\mathbf{T}}\cap \{\bm{0}\xleftrightarrow{\ge 0}\partial B(T^{\frac{1}{4}})\}^c \subset \{\mathrm{cap}(\mathcal{C}^{\ge 0}(\bm{0})\cap \widetilde{B}(T^{\frac{1}{4}}))\ge T\}$, we know that 
	\begin{equation}\label{use_241}
		\begin{split}
		\mathbb{P}\big(\mathsf{H}_0^{\mathbf{T}}[n,m] \big)\le & \mathbb{P}\big(\bm{0}\xleftrightarrow{} \partial B(T^{\frac{1}{4}}) , \widetilde{ \mathcal{L}}_{1/2}[n,m]\neq \emptyset  \big)	\\
				&+\mathbb{P}\big(\mathrm{cap}(\mathcal{C}^{\ge 0}(\bm{0})\cap \widetilde{B}(T^{\frac{1}{4}}))\ge T,\widetilde{ \mathcal{L}}_{1/2}[n,m]\neq \emptyset   \big).					\end{split}
	\end{equation}
	For the same reasons as in proving (\ref{3.37}), the first term on the right-hand side of (\ref{use_241}) satisfies that 
\begin{equation}\label{newuse3.41}
	\mathbb{P}\big(\bm{0}\xleftrightarrow{} \partial B(T^{\frac{1}{4}}) , \widetilde{ \mathcal{L}}_{1/2}[n,m]\neq \emptyset  \big) \lesssim  u^{-1} \mathbb{P}(\mathsf{H}_h^{\mathbf{T}}). 
\end{equation}
For the second term, we arbitrarily fix a point $z\in \partial B(2dT^{\frac{1}{4}})$. By (\ref{new2.20}), we know that arbitrarily given $\mathcal{C}^{\ge 0}(\bm{0})$ satisfying $\mathrm{cap}(\mathcal{C}^{\ge 0}(\bm{0})\cap \widetilde{B}(T^{\frac{1}{4}}))\ge T$, an independent Brownian motion with starting point $z$ will hit $\mathcal{C}^{\ge 0}(\bm{0})\cap \widetilde{B}(T^{\frac{1}{4}})$ with probability at least $cT^{\frac{2-d}{4}}$. This implies that 
\begin{equation}\label{3.41}
	\begin{split}
	&T^{\frac{2-d}{4}}\mathbb{P}\big(\mathrm{cap}(\mathcal{C}^{\ge 0}(\bm{0})\cap \widetilde{B}(T^{\frac{1}{4}}))\ge T,\widetilde{ \mathcal{L}}_{1/2}[n,m]\neq \emptyset   \big)\\
	\lesssim  & \mathbb{P}\times \widetilde{\mathbb{P}}_z\big(\tau_{\mathcal{C}^{\ge 0}(\bm{0})\cap \widetilde{B}(T^{\frac{1}{4}})}<\infty , \widetilde{ \mathcal{L}}_{1/2}[n,m]\neq \emptyset \big).
		\end{split}
\end{equation}
Moreover, the right-hand side of (\ref{3.41}) is at most 
\begin{equation}\label{3.42}
	\begin{split}
		&\mathbb{P}\times \widetilde{\mathbb{P}}_z\big(\tau_{\mathcal{C}^{\ge 0}(\bm{0})\cap [\widetilde{B}(T^{\frac{1}{4}})\setminus \widetilde{B}(Cm^{\frac{d-4}{2}})]}<\infty , \widetilde{ \mathcal{L}}_{1/2}[n,m]\neq \emptyset \big) \\
		&+ \widetilde{\mathbb{P}}_z\big(\tau_{\widetilde{B}(Cm^{\frac{d-4}{2}})}<\infty \big),	\end{split}
\end{equation}
where the second term, by (\ref{new2.20}) and $T\ge m^{2d^2}$, is at most 
\begin{equation}
	C'm^{\frac{d-4}{2}\cdot (d-2)}T^{\frac{2-d}{4}}\lesssim  u^{-1} T^{\frac{4-d}{4}}.  \end{equation}
In addition, applying the union bound, the first term in (\ref{3.42}) is at most 
\begin{equation}\label{3.44}
	\begin{split}
		 &\sum_{x\in B(T^{\frac{1}{4}})\setminus B(Cm^{\frac{d-4}{2}})} \mathbb{P}(\mathsf{H}_0^{\mathbf{x}}, \widetilde{ \mathcal{L}}_{1/2}[n,m]\neq \emptyset ) \cdot  \widetilde{\mathbb{P}}_z\big( \tau_{\widetilde{B}_x(1)}<\infty\big)\\
		\overset{(\ref*{3.32})\ \text{for}\ \diamond =\mathbf{x},(\ref*{new2.20})}{\lesssim} 	& u^{-1}T^{\frac{2-d}{4}} \sum_{x\in B(T^{\frac{1}{4}})\setminus B(Cm^{\frac{d-4}{2}})} |x|^{2-d} 
		\overset{}{\lesssim }   u^{-1} T^{\frac{4-d}{4}},
	\end{split}
\end{equation}
where in the last inequality we used the fact that for any $l\ge 1$ and $\alpha\neq d$ (for convenience, we set $0^{-a}:=1$ for $a>0$), 
\begin{equation}\label{newcite3.52}
\begin{split}
		\sum\nolimits_{x\in B(l)} |x|^{-\alpha} =& \sum\nolimits_{0\le j\le l}\sum\nolimits_{x\in \partial B(j)} |x|^{-\alpha}\\
		 \lesssim &  \sum\nolimits_{0\le j\le l} j^{d-1}\cdot j^{-\alpha} \lesssim l^{ (d-\alpha) \vee 0 }.\end{split}
\end{equation} 
Putting (\ref{3.41})--(\ref{3.44}) together, we obtain  
\begin{equation*}
	\begin{split}
		& \mathbb{P}\big(\mathrm{cap}(\mathcal{C}^{\ge 0}(\bm{0})\cap \widetilde{B}(T^{\frac{1}{4}}))\ge T,\widetilde{ \mathcal{L}}_{1/2}[n,m]\neq \emptyset   \big) \\
		\lesssim  & u^{-1}T^{-\frac{1}{2}}\overset{(\ref*{iic_type4new}),\mathsf{H}_0^{\mathbf{T}}\subset \mathsf{H}_h^{\mathbf{T}} }{\lesssim } u^{-1}  \mathbb{P}(\mathsf{H}_h^{\mathbf{T}}). 
	\end{split}
\end{equation*}
Combined with (\ref{use_241}) and (\ref{newuse3.41}), it implies that for any $d\ge 7$,
\begin{equation}\label{newuse3.46}
	\mathbb{P}\big(\mathsf{H}_0^{\mathbf{T}}[n,m]  \big)
		\lesssim  u^{-1} \mathbb{P}(\mathsf{H}_h^{\mathbf{T}}).
\end{equation}

By (\ref{3.37}) and (\ref{newuse3.46}), we obtain (\ref{submit355}) and thus prove (\ref{3.32}) for $\diamond = \mathbf{T}$. Now we complete the proof of this lemma.
\end{proof}

To simplify the exposition, we first record some notations as follows. 
	
\begin{itemize}

 \item  Recall Definition \ref{def_admissible}. Arbitrarily take admissible events $\{\mathsf{A}_h\}_{h\ge 0}$ with $$\mathsf{A}_h=\mathsf{A}_h^{(\mathrm{i})}(v_1,...,v_j;a_1,...,a_j)\cap \mathsf{A}_h^{(\mathrm{ii})}(\gamma_1,...,\gamma_l),$$
where $j\in \mathbb{N}^+$, $\{a_1,...,a_j\}\subset (0,\infty)$, $\{v_1,...,v_j\}\subset   \widetilde{\mathbb{Z}}^d$, and subsets $\gamma_1,...,\gamma_l\subset \{v_1,...,v_j\}$ are disjoint.

	\item Let $n_0>0$ be a large number such that $v_1,...,v_j\in \widetilde{B}(n_0)$,  and let $\lambda >1$ be sufficiently large. For each $i\in \mathbb{N}^+$, we inductively define 
\begin{equation}\label{def_ni}
	n_{i}= n_{i}(n_0, \lambda):=\lambda^2 i^{4}n_{i-1}^{1\boxdot \frac{d-4}{2}}.
\end{equation}

    \item   Let $K$ be a large integer, and we denote 
\begin{equation}\label{def_m}
	m=m(n_0,\lambda, K):=n_{6K+10}.
\end{equation}
    We assume that $0\le h\le m^{-d^5}$ and $\diamond \ge m^{d^5}$ for all $\diamond\in \{\mathbf{N},\mathbf{x},\mathbf{T}\}$.

      \item  For any $h\ge 0$ and any integer $k\in [0,K]$, we define the following events: 
   \begin{equation}\label{def_F_diamond1}
   	\mathsf{F}_{k,h}^{\diamond, (1)}:=\mathsf{H}_h^{\diamond} \cap \big[ \big\{ \widetilde{\mathcal{L}}_{1/2}[n_{6k},n_{6k+1}]\neq 0\big\}\cup \big\{\widetilde{\mathcal{L}}_{1/2}[n_{6k+4},n_{6k+5}]\neq 0  \big\}\big],
   \end{equation}
    \begin{equation}\label{def_F_diamond2}
\begin{split}
		\mathsf{F}_{k,h}^{\diamond, (2)}:= &\big[  \mathsf{H}_h^{\diamond}  \circ \big\{ B(n_{6k+1})\xleftrightarrow{} \partial B(n_{6k+2})\big\}\big]\\
		&\cup \big[ \mathsf{H}_h^{\diamond} \circ \big\{ B(n_{6k+2})\xleftrightarrow{} \partial B(n_{6k+3})\big\}\big]. 
\end{split}
\end{equation}
Here $\mathsf{A}_1\circ \mathsf{A}_2$ represents the event that $\mathsf{A}_1$ and $\mathsf{A}_2$ both happen and are certified by disjoint collections of loops (the precise definition is provided in \cite[Section 3.3]{cai2024high}, which serves the purpose for applying BKR inequality). We also define 
\begin{equation}\label{def_event_F3}
	\mathsf{F}^{\diamond,(3)}_{h}:= \mathsf{H}_h^{\diamond} \cap \big\{\widetilde{B}(m)\cap (\cup \widetilde{\mathcal{I}}^{\frac{1}{2}h^2})\neq \emptyset \big\}. 
\end{equation}
We then denote
\begin{equation}
	\mathsf{G}_h^{\diamond}:= (\mathsf{A}_h\cap \mathsf{H}_h^{\diamond})\setminus \big(\cup_{0\le k\le K}\mathsf{F}^{\diamond,(1)}_{k,h}\cup  \mathsf{F}^{\diamond,(2)}_{k,h} \cup \mathsf{F}^{\diamond,(3)}_{h}  \big).
\end{equation}

\end{itemize}

It directly follows from Lemma \ref{lemma3.4} that 
\begin{equation}\label{use3.49}
	\mathbb{P}\big(\mathsf{F}^{\diamond,(1)}_{k,h} \big) \lesssim   \lambda^{-2}k^{-4}\mathbb{P}\big(\mathsf{H}_h^{\diamond}  \big). 
\end{equation}
Moreover, by the BKR inequality and the isomorphism theorem, one has 
\begin{equation}\label{finaluse3.50}
\begin{split}
	\mathbb{P}\big(\mathsf{F}^{\diamond,(2)}_{k,h} \big) \le  &2\big[\rho_d(n_{6k+1},n_{6k+2})+ \rho_d(n_{6k+2},n_{6k+3}) 
	\big]\mathbb{P}\big(\mathsf{H}_h^{\diamond}  \big) \\
	\overset{(\ref*{crossing_low}), (\ref*{crossing_high})}{\lesssim} & \lambda^{-1}k^{-2} \mathbb{P}\big(\mathsf{H}_h^{\diamond}  \big). 
\end{split}
\end{equation}
Also note that by (\ref{renew3.35}), we have 
\begin{equation}\label{use3.51}
	\mathbb{P}\big(\mathsf{F}^{\diamond,(3)}_{h} \big)\lesssim \lambda^{-1}k^{-2} \mathbb{P}\big(\mathsf{H}_h^{\diamond}  \big). 
\end{equation}
Combining (\ref{use3.49})--(\ref{use3.51}), we get 
\begin{equation}\label{use3.53}
	\mathbb{P}\big( \mathsf{G}_h^{\diamond}\big) \ge \mathbb{P}\big(\mathsf{A}_h, \mathsf{H}_h^{\diamond} \big)-  C\lambda^{-1}\mathbb{P}\big( \mathsf{H}_h^{\diamond}\big).
\end{equation}

For brevity, we denote $\mathbf{B}_k:=\widetilde{B}(n_{6k+3})\setminus \widetilde{B}(n_{6k+1})$ and $\widehat{\mathbf{B}}_k:=\widetilde{B}(n_{6k+4})\setminus \widetilde{B}(n_{6k})$ for $0\le k\le K-1$, and denote $\mathbf{B}_{K}:= \widetilde{\mathbb{Z}}^d\setminus \widetilde{B}(n_{6K+1})$ and $\widehat{\mathbf{B}}_{K}:= \widetilde{\mathbb{Z}}^d\setminus \widetilde{B}(n_{6K})$. For any $0\le k\le K$ and $h\ge 0$, we define the following exploration process inductively:
\begin{itemize}
	\item Step $0$: We set $\mathcal{C}_0=\partial B(n_{6k+2})$.
	
	\item Step $j$ ($j\ge 1$): Given $\mathcal{C}_{j-1}$, we define $\mathcal{C}_{j}$ as the union of $\mathcal{C}_{j-1}$ and the ranges of all loops in $\widetilde{\mathcal{L}}_{1/2}$ and trajectories in $\widetilde{\mathcal{I}}^{\frac{1}{2}h^2}$ that intersect $\mathcal{C}_{j-1}\cap \mathbf{B}_k$. If $\mathcal{C}_{j}=\mathcal{C}_{j-1}$, we stop the construction and define $\mathfrak{C}_{k,h}:=\mathcal{C}_{j}$; otherwise (i.e., $\mathcal{C}_{j-1}\subsetneq \mathcal{C}_{j}$), we proceed to the next step.

\end{itemize}

We denote by $\widehat{\mathfrak{C}}_{k,h}$ the analogue of $\mathfrak{C}_{k,h}$ obtained by the same construction but restricting the loops to $\widetilde{\mathcal{L}}_{1/2}\cdot \mathbbm{1}_{\mathrm{ran}(\widetilde{\ell})\subset \widehat{\mathbf{B}}_k}$. Moreover, if there is exactly one cluster in $\widehat{\mathfrak{C}}_{k,h}$ crossing $\mathbf{B}_k$, then we denote this cluster by $\widehat{\mathcal{C}}_{k,h}$; otherwise (i.e., if there is no crossing cluster or if there are multiple crossing clusters), we set $\widehat{\mathcal{C}}_{k,h}=\emptyset$ for completeness. For brevity, we denote $\widehat{\mathfrak{C}}_{k}:=\widehat{\mathfrak{C}}_{k,0}$ and $\widehat{\mathcal{C}}_{k}=\widehat{\mathcal{C}}_{k,0}$ for $0\le k\le K-1$, and denote $\widehat{\mathfrak{C}}_{K}:=\widehat{\mathfrak{C}}_{K,h}$ and $\widehat{\mathcal{C}}_{K}=\widehat{\mathcal{C}}_{K,h}$ (here we consider the clusters $\{ \widehat{\mathfrak{C}}_{k}\}_{0\le k\le K-1}$ without $\widetilde{\mathcal{I}}^{\frac{1}{2}h^2}$ because on the event $\mathsf{F}^{\diamond,(3)}_{h}$, no trajectory in $\widetilde{\mathcal{I}}^{\frac{1}{2}h^2}$ intersects $\cup_{0\le k\le K-1} \widehat{\mathbf{B}}_k$). Note that $\{\widehat{\mathfrak{C}}_{k}\}_{0\le k\le K}$ are independent since $\widehat{\mathfrak{C}}_{k}\subset \widehat{\mathbf{B}}_k$ for all $0\le k\le K$ and $\{\widehat{\mathbf{B}}_k\}_{0\le k\le K}$ are disjoint. For any $0\le k\le K$, let $\widecheck{\mathfrak{C}}_{k}:= \widehat{\mathfrak{C}}_k\cap \mathbf{B}_k$. For convenience, we set $\widehat{\mathcal{C}}_{-1}:=\{\bm{0}\}$ and $\widecheck{\mathfrak{C}}_{-1}:=\emptyset$. An illustration is provided in Figure \ref{fig1}.

According to the construction, on $\mathsf{G}_h^{\diamond}$, the following events happen: 
\begin{enumerate}
	\item[(a)]   $\mathfrak{C}_{k,h}=\widehat{\mathfrak{C}}_{k }$ for all $0\le k\le K$. Thanks to the definition of $\mathsf{F}_{k,h}^{\diamond, (1)}$ (see (\ref{def_F_diamond1})), on the event $\mathsf{G}_h^{\diamond}$, no loop crosses $B(n_{6k+1})\setminus B(n_{6k})$ or $B(n_{6k+5})\setminus B(n_{6k+4})$.

	\item[(b)] For each $0\le k\le K$, there is exactly one cluster in $\mathfrak{C}_{k,h}$ crossing $\mathbf{B}_k$, which exactly equals $\widehat{\mathcal{C}}_{k}$. This is because $\mathsf{F}_{k,h}^{\diamond, (2)}$ (see (\ref{def_F_diamond2})) does not occur on $\mathsf{G}_h^{\diamond}$.

	\item[(c)] For $-1\le k\le K-1$, $\widehat{\mathcal{C}}_{k}$ and $\widehat{\mathcal{C}}_{k+1}$ are connected by the loops in $\widetilde{\mathcal{L}}_{1/2}$ disjoint from $\widecheck{\mathfrak{C}}_k\cup \widecheck{\mathfrak{C}}_{k+1}$. This follows from Events (a) and (b).

	\item[(d)]  The event $\mathsf{A}_h$ is certified by $\widehat{\mathfrak{C}}_0$ and by loops of $\widetilde{\mathcal{L}}_{1/2}$ that are disjoint from $\widecheck{\mathfrak{C}}_0$ and contained in $\widehat{\mathbf{B}}_0$ (we denote the collection of these loops by $\mathfrak{L}_0$). This is because on $\mathsf{G}_h^{\diamond}$, the event $\mathsf{F}_{0,h}^{\diamond,(1)}\cap \mathsf{F}_{h}^{\diamond,(3)}$ does not occur, implying that the set $\{v_1,...,v_j\}\subset \widetilde{B}(n_0)$ is disjoint from $\cup \widetilde{\mathcal{I}}^{\frac{1}{2}h^2}\cup \widehat{\mathfrak{C}}_{0}$, and thus  $\{\widehat{\mathcal{L}}_{1/2}^{v_i}+ \widehat{\mathcal{I}}^{\frac{1}{2}h^2}_{v_i}\}_{1\le i \le j}$ depend only on $\mathfrak{L}_0$; in addition, for any $v_{i_1},v_{i_2}$ contained in one of $\gamma_1,...,\gamma_l$, $\{v_{i_1} \xleftrightarrow{[h]} v_{i_2}\}$ happens if and only if $v_{i_1}$ and $v_{i_2}$ are connected by $(\cup \mathfrak{L}_0) \cup \widehat{\mathfrak{C}}_0$.

	\item[(e)] When $\diamond \in \{\mathbf{N},\mathbf{x}\}$, $\widehat{\mathcal{C}}_K\cap A_\diamond\neq \emptyset$; when $\diamond=\mathbf{T}$, $\mathrm{cap}(\widehat{\mathcal{C}}_K)\ge T-Cm^{d-2}$ (this is because by (\ref{2.8}), $\mathcal{C}^{\ge 0}(\bm{0}) \subset \widehat{\mathcal{C}}_K\cup \widetilde{B}(m)$ and $\mathrm{cap}(\widetilde{B}(m))\le Cm^{d-2}$, one has $\mathrm{cap}(\widehat{\mathcal{C}}_K)\ge \mathrm{cap}(\mathcal{C}^{\ge 0}(\bm{0}))-Cm^{d-2}$). We denote this event by $\widehat{\mathsf{H}}_{h,K}^{\diamond}$

\end{enumerate}

	\textbf{P.S.:} (1) The loops used to connect $\widehat{\mathcal{C}}_{k}$ and $\widehat{\mathcal{C}}_{k+1}$ (in Event (c)) are contained in $\widetilde{\mathbf{B}}_k^+$ (where $\widetilde{\mathbf{B}}_{-1}^+:=\widetilde{B}(n_{2})$ and $\widetilde{\mathbf{B}}_k^+:=\widetilde{B}(n_{6k+8})\setminus \widetilde{B}(n_{6k+2})$ for $k\ge 0$), and thus for different $k$ these different collections of loops are independent of each other (since $\{\widetilde{\mathbf{B}}_k^+\}_{-1\le k\le K-1}$ are disjoint); (2) Conditioned on $\cap_{0\le k\le K} \{\widehat{\mathfrak{C}}_{k }=D_k\} $ (for some $D_0,...,D_k\subset \widetilde{\mathbb{Z}}^d$), the distribution of loops that are not used in the construction of $\{\widehat{\mathfrak{C}}_{k}\}_{0\le k\le K}$ is given by $\widetilde{\mathcal{L}}_{1/2}^{\cup_{0\le k\le K} (D_k\cap \mathbf{B}_k)}$.

\begin{figure}[h]
	\centering
	\includegraphics[width=0.8\textwidth]{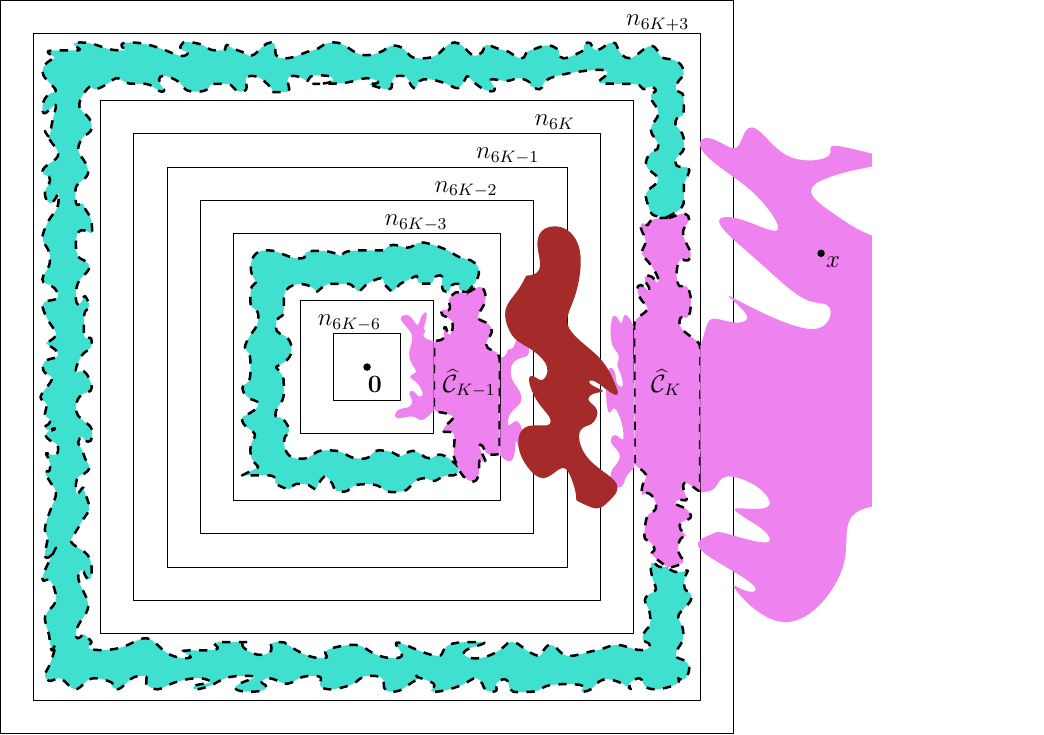}
	\caption{In this illustration, we consider the case $\diamond=\mathbf{x}$. The two pink regions represent $\widehat{\mathcal{C}}_{K-1}$ and $\widehat{\mathcal{C}}_{K}$ respectively, where $\widehat{\mathcal{C}}_{K}$ intersects $x$. The union of the green and pink regions inside $B(n_{6K-2})$ (resp. outside $B(n_{6K})$) represents $\widehat{\mathfrak{C}}_{K-1}$ (resp. $\widehat{\mathfrak{C}}_{K}$). The two regions surrounded by dashed curves represent $\widecheck{\mathfrak{C}}_{K-1}$ and $\widecheck{\mathfrak{C}}_{K}$ respectively. Given $\widehat{\mathfrak{C}}_{K-1}$ and $\widehat{\mathfrak{C}}_{K-1}$, the dashed curves serve as absorbing boundaries for loops. The red region is a loop cluster that connects $\widehat{\mathcal{C}}_{K-1}$ and $\widehat{\mathcal{C}}_{K}$, and consists of loops disjoint from $\widecheck{\mathfrak{C}}_{K-1}\cup \widecheck{\mathfrak{C}}_{K}$.
	 }\label{fig1}
\end{figure}

Consequently, we obtain that 
\begin{equation}\label{use3.55}
	\mathbb{P}\big( 	\mathsf{G}_h^{\diamond}\big)\le \mathbb{I}( \mathsf{A}_h , n_0,\lambda,K,\diamond), 
\end{equation}
where the integral $\mathbb{I}(\cdot )$ is defined as 
\begin{equation*}
 	\int_{\widehat{\mathcal{C}}_k\neq \emptyset,\forall 0\le k\le K; \widehat{\mathsf{H}}_{h,K}^{\diamond}}  \mathbb{P}\big(\mathsf{A}_h, \bm{0} \xleftrightarrow{(\widecheck{D}_{0})}D_0  \big) \prod_{k=0}^{K-1} \mathbb{P}\big(D_k \xleftrightarrow{(\widecheck{D}_k\cup \widecheck{D}_{k+1})}D_{k+1}  \big)d\widehat{\mathfrak{p}}_0(D_0)\cdots d\widehat{\mathfrak{p}}_{K}(D_k).
\end{equation*}
Here, $\{\widehat{\mathfrak{p}}_k\}_{0\le k\le K}$ are the probability measures of $\{\widehat{\mathfrak{C}}_k\}_{0\le k\le K}$, and $\widecheck{D}_{k}:=D_k\cap \mathbf{B}_k $. Note that the occurrence of $\{ \mathsf{A}_h, \bm{0} \xleftrightarrow{(\widecheck{\mathfrak{C}}_{0})}\widehat{\mathfrak{C}}_0\}$, $\cap_{0\le k\le K-1}\{\widehat{\mathfrak{C}}_k \xleftrightarrow{(\widecheck{\mathfrak{C}}_k\cup \widecheck{\mathfrak{C}}_{k+1})} \widehat{\mathfrak{C}}_{k+1}  \}$ and $\widehat{\mathsf{H}}_{h,K}^{\diamond}$ implies $\mathsf{A}_h \cap \mathsf{H}_h^{\diamond}$ when $\diamond\in \{\mathbf{N},\mathbf{x}\}$, and implies $\mathsf{A}_h\cap \{\mathrm{cap}(\mathcal{C}^{[h]}(\bm{0}))\ge T-Cm^d\}$ when $\diamond= \mathbf{T}$. Furthermore, when $\diamond= \mathbf{T}$, since $\mathbb{P}\big(\mathrm{cap}(\mathcal{C}^{[h]}(\bm{0}))\ge a\big)= \mathbb{P}(\bm{0}\xleftrightarrow{[h]}\infty)+\mathbb{P}\big(\mathrm{cap}(\mathcal{C}(\bm{0}))\ge a, \mathcal{C}(\bm{0})\cap (\cup \widetilde{\mathcal{I}}^{\frac{1}{2}h^2}) =\emptyset \big)$, one has
 \begin{equation*}
 	\begin{split}
 		0\le & \mathbb{P}\big(\mathrm{cap}(\mathcal{C}^{[h]}(\bm{0}))\ge T-Cm^d\big) - \mathbb{P}\big(\mathsf{H}^{\mathbf{T}}_{h} \big)\\
 		=& \mathbb{P}\big(T-Cm^d \le \mathrm{cap}(\mathcal{C}(\bm{0}))\le T , \mathcal{C}(\bm{0})\cap (\cup \widetilde{\mathcal{I}}^{\frac{1}{2}h^2}) =\emptyset  \big)\\
 		\le & \mathbb{P}\big(T-Cm^d \le \mathrm{cap}(\mathcal{C}(\bm{0}))\le T  \big)\\
 			\overset{(\ref*{iic_type4new})}{\lesssim } &  m^dT^{-1} \mathbb{P}(\mathsf{H}^{\mathbf{T}}_{0}) \overset{}{\lesssim }  \lambda^{-1} \mathbb{P}(\mathsf{H}^{\mathbf{T}}_{h}),
 	\end{split}
 \end{equation*}
 where in the last inequality we used $\mathsf{H}^{\mathbf{T}}_{0}\subset \mathsf{H}^{\mathbf{T}}_{h}$, $T \ge m^{d^5}$ and $m\ge \lambda$ (recall (\ref{def_ni})). In conclusion, we have 
\begin{equation}\label{finaluse_375}
	\mathbb{I}(\{\mathsf{A}_h\}_{h\ge 0}, n_0,\lambda,K,\diamond) \le \mathbb{P}\big( \mathsf{A}_h \cap \mathsf{H}_h^{\diamond}\big)+  \mathbbm{1}_{\diamond=\mathbf{T}}\cdot  C\lambda^{-1}  \mathbb{P}(\mathsf{H}^{\diamond}_{h}).
\end{equation}
Combined with (\ref{use3.53}) and (\ref{use3.55}), it yields an approximation of $\mathsf{A}_h \cap \mathsf{H}_h^{\diamond}$, as incorporated in the next lemma (where we keep using the aforementioned notations).

\begin{lemma}\label{lemma_approx_theta}
	For any $d\ge 3$ with $d\neq 6$, $\delta>0$ and any admissible $\{\mathsf{A}_h\}_{h\ge 0}$, there exists $\Cl\label{const_approx_theta}(d,\{\mathsf{A}_h\}_{h\ge 0})>0$ such that for any $n_0,K\ge \Cref{const_approx_theta}$, $0\le h\le m^{-d^5}$, $\lambda\ge \Cref{const_approx_theta}\delta^{-1}$, and any $\diamond\in \{\mathbf{N},\mathbf{x},\mathbf{T}\}$ with $\diamond\ge m^{d^5}$, 
	 	\begin{equation}
		\big|\mathbb{I}(\mathsf{A}_h, n_0,\lambda,K,\diamond) - \mathbb{P}\big( \mathsf{A}_h \cap \mathsf{H}_h^{\diamond}\big)  \big| \le \delta \cdot \mathbb{P}( \mathsf{H}_h^{\diamond}).	\end{equation}

	\end{lemma}

\subsection{Proof of Proposition \ref{lemma_converge_loop}}\label{proof_prop_converge_loop}

Before delving into the details, we provide an overview of the proof of Proposition \ref{lemma_converge_loop}. The uniform convergence for $h\in [0,c]$ (as stated in Proposition \ref{lemma_converge_loop}) is established separately based on whether $h$ exceeds a certain threshold. When $h$ exceeds this threshold, the large cluster $\mathcal{C}^{[h]}(\bm{0})$ typically contains the random interlacements $\widetilde{\mathcal{I}}^{\frac{1}{2}h^2}$. Consequently, the probabilities of $\mathsf{A}_h\cap \mathsf{H}_h^{\diamond}$ and $\mathsf{H}_h^{\diamond}$ are approximately equal to those of their intersections with $\{\bm{0}\xleftrightarrow{[h]} \infty \}$ (see Lemma \ref{lemma_compare_hdiamond}), implying that $\mathbb{P}\big(\mathsf{A}_h\mid \mathsf{H}_h^{\diamond}\big)\to \mathbb{P}\big(\mathsf{A}_h\mid \bm{0}\xleftrightarrow{[h]} \infty\big)$ as $\diamond \to \infty$ with a uniform rate (ensured by the aforementioned threshold).

 When $h$ is below the threshold, the random interlacements $\widetilde{\mathcal{I}}^{\frac{1}{2}h^2}$ typically do not intersect a box around $\bm{0}$ with diameter depending on this threshold (through (\ref{2.30})). Thus, with high probability, the configuration of $\mathcal{C}^{[h]}(\bm{0})$ near $\bm{0}$ is determined by the loop soup $\widetilde{\mathcal{L}}_{1/2}$. Inspired by \cite{basu2017kesten}, we intend to construct the loop cluster near $\bm{0}$ by first sampling sub-clusters contained in disjoint annuli and then connecting them sequentially (as shown in the last subsection). In the case of Bernoulli percolation, these sub-clusters are automatically independent, which aligns perfectly with the conditions required for an analysis using Hopf's contraction theorem (see \cite[Theorem 1]{hopf1963inequality}). However, due to the large loops in $\widetilde{\mathcal{L}}_{1/2}$, this independence no longer holds in our case. To this end, we need to show that the occurrence of a large loop is rare on the event $\mathsf{H}_h^{\diamond}$ (as shown by Lemma \ref{lemma3.4}). Based on the aforementioned construction for the loop cluster near $\bm{0}$, Hopf's contraction theorem shows that the oscillation of the density decays exponentially with the number of sub-clusters (which exceeds a uniform bound determined by the diameter of the region disjoint from $\widetilde{\mathcal{I}}^{\frac{1}{2}h^2}$), thereby deriving the uniform convergence of $\mathbb{P}\big(\mathsf{A}_h\mid \mathsf{H}_h^{\diamond}\big)$ as $\diamond \to \infty$.

We next carry out the proof details. It suffices to show that for any $d\ge 3$ with $d\neq 6$ and any admissible $\{\mathsf{A}_h\}_{h\ge 0}$, there exist $C_\ddagger(d,\{\mathsf{A}_h\}_{h\ge 0})>0,c_\ddagger(d,\{\mathsf{A}_h\}_{h\ge 0})\in (0,1)$ such that for any $\diamond\in\{\mathbf{N},\mathbf{x},\mathbf{T}\}$, there is $\zeta(d,\{\mathsf{A}_h\}_{h\ge 0},\diamond)\in [0,1]$ such that 
\begin{equation}\label{new3.39}
		\big| \mathbb{P}\big(\mathsf{A}_h \mid \mathsf{H}_h^{\diamond} \big)- \zeta \big| \le \epsilon
	\end{equation} 
	holds for all $\epsilon\in (0,e^{-100d^2})$, $\diamond \ge M_\ddagger:=\mathrm{exp}(C_\ddagger e^{\epsilon^{-10}})$ and $h\in [0,c_\ddagger]$. As illuminated in the preceding heuristic discussion, we prove this separately for two cases: $h\ge \mathrm{exp}(-C_\ddagger e^{\epsilon^{-2}})$ and  $0\le h< \mathrm{exp}(-C_\ddagger e^{\epsilon^{-2}})$.

	  \textbf{Case 1:} $h\ge \mathrm{exp}(-C_\ddagger e^{\epsilon^{-2}})$. In this case, we take $\zeta=\mathbb{P}\big(\mathsf{A}_h \mid  \mathsf{H}_h^{\diamond,\infty} \big)$, where we denote $\mathsf{H}_h^{\diamond,\infty}:= \mathsf{H}_h^{\diamond }\cap \{\bm{0}\xleftrightarrow{[h]}\infty\}$. Note that $\mathsf{H}_h^{\diamond,\infty}=\{\bm{0}\xleftrightarrow{[h]}\infty \}$ when $\diamond \in \{\mathbf{N}, \mathbf{T}\}$. It follows from Lemma \ref{lemma_compare_hdiamond} that 
	  \begin{equation}\label{newfinish3.78}
	  	   \mathbb{P}\big(\mathsf{A}_h, \mathsf{H}_h^{\diamond} \big) -  \mathbb{P}\big(\mathsf{A}_h , \mathsf{H}_h^{\diamond,\infty} \big),   \mathbb{P}\big( \mathsf{H}_h^{\diamond} \big) -  \mathbb{P}\big(  \mathsf{H}_h^{\diamond,\infty} \big)\in \big[0, Ce^{-ch^2M_\ddagger^{\frac{1}{d}}} \big].  
	  \end{equation}
  Note that for any $a,a',b,b',\delta>0$ with $a-a',b-b'\in [0,\delta b']$ and $a'\le b'$, one has $\frac{a'}{b'} -\delta \le  \frac{a'}{b'}\cdot (1-\delta) \le \frac{a' }{b'-\delta  b'}  \le  \frac{a}{b}\le \frac{a'+\delta b'}{b'}= \frac{a'}{b'}+\delta$. I.e.,  
	  \begin{equation}\label{newfinish3.79}
	  	\big| \frac{a}{b} -\frac{a'}{b'}  \big| \le \delta. 
	  \end{equation}
	   In addition, by (\ref{newsubmit_224}) and $h\ge \mathrm{exp}(-C_\ddagger e^{\epsilon^{-2}})$, we have 
	    \begin{equation}\label{newuse_3.42}
	  	\begin{split}
 \mathbb{P}\big(\mathsf{H}_h^{\diamond,\infty} \big) \ge c'h^2\ge (\tfrac{1}{2}\epsilon )^{-1} \cdot  Ce^{-ch^2M_\ddagger^{\frac{1}{d}}}. 
	  	\end{split}
	  \end{equation}
	   Combining (\ref{newfinish3.78})--(\ref{newuse_3.42}), we get
	  	\begin{equation}\label{new3.43}
		\big|   \mathbb{P}\big(\mathsf{A}_h \mid  \mathsf{H}_h^{\diamond} \big) -   \mathbb{P}\big(\mathsf{A}_h \mid  \mathsf{H}_h^{\diamond,\infty} \big)\big|\le \tfrac{1}{2}\epsilon. 
			\end{equation}

	For $\diamond \in \{\mathbf{N},\mathbf{T}\}$, since $\mathsf{H}_h^{\diamond,\infty}=\{\bm{0}\xleftrightarrow{[h]}\infty \}$, we obtain (\ref{new3.39}) from (\ref{new3.43}), with $\zeta=\mathbb{P}\big(\mathsf{A}_h \mid   \bm{0}\xleftrightarrow{[h]} \infty \big)$.

	  For $\diamond = \mathbf{x}$, we aim to show that $\mathbb{P}\big(\mathsf{A}_h, \mathsf{H}_h^{\mathbf{x},\infty} \big)\approx\mathbb{P}\big(\mathsf{A}_h, \bm{0}\xleftrightarrow{[h]}\infty \big)\cdot \mathbb{P}( \bm{0}\xleftrightarrow{[h]}\infty )$. It follows from (\ref{newsubmit_223}) that
	  \begin{equation}\label{3.65}
	\mathbb{P}\big(\mathsf{A}_h, \mathsf{H}_h^{\mathbf{x},\infty} \big) \ge \mathbb{P}\big(\mathsf{A}_h, \bm{0}\xleftrightarrow{[h]}\infty \big)\cdot \mathbb{P}( \bm{0}\xleftrightarrow{[h]}\infty ).
	\end{equation}	
	We next control the approximation error when reversing (\ref{3.65}). Suppose that $\mathsf{A}_h=\mathsf{A}_h^{(\mathrm{i})}(v_1,...,v_k;a_1,...,a_k)\cap \mathsf{A}_h^{(\mathrm{ii})}(\gamma_1,...,\gamma_l)$ for all $h\ge 0$ (recalling Definition \ref{def_admissible}). In addition, by taking a sufficiently large $C_\ddagger$, we can ensure that $\{v_1,...,v_k\}\subset \widetilde{B}(M_\ddagger^{\frac{1}{10d}})$. We define a truncated version of $\mathsf{A}_h^{(\mathrm{ii})}$ as follows:
	  \begin{equation*}
	  	\overbar{\mathsf{A}}_h^{(\mathrm{ii})}:= \cap_{1\le j\le l}\cap_{v,v'\in \gamma_j}\Big\{ v \xleftrightarrow{\big( (\cup \widetilde{\mathcal{L}}_{1/2})\cup (\cup \widetilde{\mathcal{I}}^{\frac{1}{2}h^2}) \big) \cap \widetilde{B}(|x|^{1/4})} v' \Big \}. 
	  \end{equation*}
	  Let $\overbar{\mathsf{A}}_h := \mathsf{A}_h^{(\mathrm{i})}\cap \overbar{\mathsf{A}}_h^{(\mathrm{ii})}$. Note that $\overbar{\mathsf{A}}_h$ and $\{\bm{0}\xleftrightarrow{[h]}\partial B(|x|^{1/4}) \}$ only depend on loops (in $\widetilde{\mathcal{L}}_{1/2}$) and trajectories (in $\widetilde{\mathcal{I}}^{\frac{1}{2}h^2}$) intersecting $\widetilde{B}(|x|^{1/4})$. Meanwhile, the event $\{x\xleftrightarrow{[h]} \partial B_x(|x|^{1/4})\}$ only depends on loops and trajectories intersecting $\widetilde{B}_x(|x|^{1/4})$. Therefore, given that $\mathfrak{L}=\mathfrak{I}=\emptyset$, where $\mathfrak{L}$ (resp. $\mathfrak{I}$) is the collection of loops in $\widetilde{\mathcal{L}}_{1/2}$ (resp. trajectories in $\widetilde{\mathcal{I}}^{\frac{1}{2}h^2}$) intersecting both $\widetilde{B}(|x|^{1/4})$ and $\widetilde{B}_x(|x|^{1/4})$, the events $\overbar{\mathsf{A}}_h\cap \{\bm{0}\xleftrightarrow{[h]}\partial B(|x|^{1/4}) \}$ and $\{x\xleftrightarrow{[h]} \partial B_x(|x|^{1/4})\}$ are certified by disjoint collections of loops and trajectories. Consequently, we obtain
	  \begin{equation}\label{newadd_3.48}
	  	\begin{split}
	  		\mathbb{P}\big(\overbar{\mathsf{A}}_h ,  \mathsf{H}_h^{\mathbf{x},\infty} \big) \le &\mathbb{P}\big(\overbar{\mathsf{A}}_h ,  \bm{0}\xleftrightarrow{[h]}\partial B(|x|^{1/4}), x\xleftrightarrow{[h]} \partial B_x(|x|^{1/4}) \big)  \\
	  		\le &\mathbb{P}\big(\overbar{\mathsf{A}}_h ,  \bm{0}\xleftrightarrow{[h]}\partial B(|x|^{1/4}) \big)\mathbb{P}\big( x\xleftrightarrow{[h]} \partial B_x(|x|^{1/4}) \big)\\
	  		&+ \mathbb{P}\big(\mathfrak{L}\neq \emptyset\big)+\mathbb{P}\big(\mathfrak{I}\neq \emptyset\big).	 	  	\end{split}
	  \end{equation}
	 Moreover, it follows from Lemma \ref{lemma_compare_hdiamond} (with $\diamond=\mathbf{N}$ and $N=|x|^{1/4}\ge M_\ddagger^{1/4}$) that 
	 \begin{equation}\label{newadd_3.49}
	0\le 	\mathbb{P}\big(\overbar{\mathsf{A}}_h ,  \bm{0}\xleftrightarrow{[h]}\partial B(|x|^{1/4}) \big) - \mathbb{P}\big(\overbar{\mathsf{A}}_h ,  \bm{0}\xleftrightarrow{[h]}\infty \big) \lesssim  e^{-ch^2M_\ddagger^{\frac{1}{4d}}}, 
		 \end{equation} 
	  \begin{equation}
	  	  0\le 	\mathbb{P}\big(x\xleftrightarrow{[h]} \partial B_x(|x|^{1/4})  \big) - \mathbb{P}\big( \bm{0}\xleftrightarrow{[h]}\infty \big)  \lesssim e^{-ch^2M_\ddagger^{\frac{1}{4d}}}. 
	  \end{equation}
	  	  In addition, since each loop in $\mathfrak{L}$ intersects both $B(|x|^{1/4})$ and $\partial B(\frac{1}{2}|x|)$, it follows from \cite[Corollary 2.2]{cai2024high} that 
	  \begin{equation}
	  	 \mathbb{P}\big(\mathfrak{L}\neq \emptyset\big) \lesssim  |x|^{2-d}\cdot |x|^{\frac{1}{4}(d-2)} \lesssim  M_\ddagger^{\frac{3(2-d)}{4}}. 
	  \end{equation}
	  Meanwhile, by (\ref{newadd_2.26}) and $h<1$, we have 
	  \begin{equation}\label{newadd_3.52}
	  	\mathbb{P}\big(\mathfrak{I}\neq \emptyset\big) \lesssim h^2|x|^{2-d}\cdot |x|^{\frac{1}{4}(2d-4)} \lesssim  M_{\ddagger}^{\frac{2-d}{2}}. 
	  		  \end{equation} 
Combining (\ref{newadd_3.48})--(\ref{newadd_3.52}) with $\overbar{\mathsf{A}}_h\subset  \mathsf{A}_h$, we get 
	  \begin{equation}\label{newaddto_3.53}
	  \begin{split}
	  		\mathbb{P}\big(\overbar{\mathsf{A}}_h ,  \mathsf{H}_h^{\mathbf{x},\infty} \big) \le &\mathbb{P}\big(\mathsf{A}_h, \bm{0}\xleftrightarrow{[h]}\infty \big)\cdot \mathbb{P}( \bm{0} \xleftrightarrow{[h]}\infty )
	  		+C(e^{-ch^2M_{\ddagger}^{\frac{1}{4d}}}+ M_{\ddagger}^{\frac{2-d}{2}} ).
	  			  \end{split}
	  \end{equation}

    Next, we estimate the probability of $(\mathsf{A}_h \setminus  \overbar{\mathsf{A}}_h)\cap \mathsf{H}_h^{\mathbf{x},\infty}$. On the event $(\mathsf{A}_h \setminus  \overbar{\mathsf{A}}_h)\cap \mathsf{H}_h^{\mathbf{x},\infty}$, there exist $1\le i_1<i_2\le k$ such that $v_{i_1}$ and $v_{i_2}$ are contained in two distinct clusters $\mathcal{C}_{i_1}$ and $\mathcal{C}_{i_2}$ of $[(\cup \widetilde{\mathcal{L}}_{1/2})\cup (\cup \widetilde{\mathcal{I}}^{\frac{1}{2}h^2})]\cap \widetilde{B}(|x|^{1/4})$ that intersect $\partial B(|x|^{1/4})$ (we denote this event by $\mathsf{A}_{i_1,i_2}$). For $j\in \{1,2\}$, we denote the event 
    \begin{equation*}
    	\mathsf{D}_j:=\big\{ \mathcal{C}_{i_j} \cap  (\cup \widetilde{\mathcal{I}}^{\frac{1}{2}h^2}) =\emptyset  \big\}. 
    \end{equation*}
   For the same reason as in proving (\ref{finally_3.28}), since $\mathcal{C}_{i_j}$ intersects $\partial B(|x|^{1/4})$, we know that $\mathrm{cap}(\mathcal{C}_{i_j})\gtrsim |x|^{1/4}\ln^{-1}(|x|)\gtrsim M_{\ddagger}^{\frac{1}{4d}}$. Combined with (\ref{2.30}), it implies that  
     \begin{equation}\label{newuse3.74}
\mathbb{P}\big( \mathsf{D}_j\big)\le Ce^{-ch^2M_{\ddagger}^{\frac{1}{4d}}}, \ \ \forall j\in \{1,2\}. 
  	  \end{equation} 
 Moreover, on $\mathsf{A}_{i_1,i_2}\cap \mathsf{D}_1^c\cap \mathsf{D}_2^c$, $\mathcal{C}_{i_1}$ and $\mathcal{C}_{i_2}$ contain two disjoint clusters of $(\cup \widetilde{\mathcal{I}}^{\frac{1}{2}h^2})\cap  \widetilde{B}(|x|^{1/4})$ that intersect $\widetilde{B}(\tfrac{1}{2}|x|^{1/4})$. In fact, by the local uniqueness of the random interlacements (see \cite[Proposition 1]{rath2011transience}), the probability of the occurrence of such two clusters in $\widetilde{\mathcal{I}}^{\frac{1}{2}h^2}$ is at most $e^{-|x|^{c'}}\le e^{-M_\ddagger^{c'}}$. Combined with (\ref{newuse3.74}), it implies 
\begin{equation}\label{3.75}
	\begin{split}
		\mathbb{P}\big( (\mathsf{A}_h \setminus  \overbar{\mathsf{A}}_h)\cap \mathsf{H}_h^{x,\infty}\big) \le k^2\big(Ce^{-ch^2M_\ddagger^{\frac{1}{4d}}}+e^{-M_{\ddagger}^{c'}}\big).
	\end{split}
\end{equation}
In addition, by (\ref{232plus}) and the assumption that $h\ge \mathrm{exp}(-C_\ddagger e^{\epsilon^{-2}})$, we have 
\begin{equation}\label{newfinish3.93}
	C(e^{-ch^2M_{\ddagger}^{\frac{1}{4d}}}+ M_{\ddagger}^{\frac{2-d}{2}} )+k^2\big(Ce^{-ch^2M_\ddagger^{\frac{1}{4d}}}+e^{-M_\ddagger^{c'}}\big) \le \tfrac{1}{2}\epsilon \cdot \big[ \mathbb{P}( \bm{0}\xleftrightarrow{[h]}\infty ) \big]^2. 
\end{equation}
Putting (\ref{3.65}), (\ref{newaddto_3.53}), (\ref{3.75}) and (\ref{newfinish3.93}) together, we obtain 
\begin{equation}\label{3.76}
	  \mathbb{P}\big(\mathsf{A}_h, \mathsf{H}_h^{\mathbf{x},\infty} \big) -\mathbb{P}\big(\mathsf{A}_h, \bm{0}\xleftrightarrow{[h]}\infty \big)\cdot \mathbb{P}( \bm{0}\xleftrightarrow{[h]}\infty )\in \Big[0, \tfrac{1}{2}\epsilon \cdot \big[ \mathbb{P}( \bm{0}\xleftrightarrow{[h]}\infty ) \big]^2 \Big] .
\end{equation}
By taking $\mathsf{A}_h$ as the sure event $\Omega$, we also have 
\begin{equation}\label{3.77}
	  \mathbb{P}\big( \mathsf{H}_h^{\mathbf{x},\infty} \big)- \big[\mathbb{P}( \bm{0}\xleftrightarrow{[h]}\infty )\big]^2 \in \Big[0, \tfrac{1}{2}\epsilon \cdot \big[ \mathbb{P}( \bm{0}\xleftrightarrow{[h]}\infty ) \big]^2 \Big] .
\end{equation}
Combining (\ref{newfinish3.79}), (\ref{3.76}) and (\ref{3.77}), we get 
\begin{equation*}
	\big|   \mathbb{P}\big(\mathsf{A}_h \mid  \mathsf{H}_h^{\mathbf{x},\infty} \big) - \mathbb{P}\big(\mathsf{A}_h \mid \bm{0} \xleftrightarrow{[h]}\infty \big)    \big|\le \tfrac{1}{2}\epsilon, 
\end{equation*}
which together with (\ref{new3.43}) implies (\ref{new3.39}) for $\diamond =\mathbf{x}$, thereby completing the proof of (\ref{new3.39}) for Case 1.

\textbf{Case 2:} $0\le h < \mathrm{exp}(-C_\ddagger e^{\epsilon^{-2}})$. The subsequent proof is mainly inspired by \cite{basu2017kesten}. This approach was built on an application of Hopf's contraction theorem (see \cite[Theorem 1]{hopf1963inequality}). Our proof only needs the following special case of this theorem.

\begin{lemma}\label{lemma_hopf}
Let $\mathbf{X}$ and $\mathbf{Y}$ be two spaces, where $\mathbf{Y}$ is equipped with a $\sigma$-finite measure $\mu_{\mathbf{Y}}$. Let $\mathcal{K}:\mathbf{X} \times \mathbf{Y}\to (0,\infty)$ be a positive kernel. Define the linear operator $\mathcal{K}f(\cdot)$ as (where $f:\mathbf{Y}\to [0,\infty)$ is a function on $\mathbf{Y}$)
\begin{equation*}
	\mathcal{K}f(x):= \int_{\mathbf{Y}}\mathcal{K}(x,y)f(y)d\mu_{\mathbf{Y}}(y). 
\end{equation*}
If there exists $\kappa>0$ such that
\begin{equation}\label{kappa}
	\frac{\mathcal{K}(x,y)\mathcal{K}(x',y')}{\mathcal{K}(x,y')\mathcal{K}(x',y)}\le \kappa^2, \ \ \forall x,x'\in \mathbf{X}\ \text{and}\ y,y'\in \mathbf{Y},
\end{equation}
then for any two functions $f_1:\mathbf{Y}\to [0,\infty)$ and $f_2:\mathbf{Y}\to (0,\infty)$, one has  
\begin{equation*}
\mathrm{osc}\big(\frac{\mathcal{K}f_1}{\mathcal{K}f_2}\big)	\le  \frac{\kappa-1}{\kappa+1}\cdot  \mathrm{osc}\big(\frac{f_1}{f_2}\big), 
\end{equation*}
	where $\mathrm{osc}(g)$ for $g:\mathbf{Z}\to [0,\infty)$ is defined as $\mathrm{osc}(g):=\sup_{z_1,z_2\in \mathbf{Z}} | g(z_1) - g(z_2) |$.
	
\end{lemma}

Recall $c_{\star}$ and $\Cref{const_approx_theta}$ in (\ref{def_c_star}) and Lemma \ref{lemma_approx_theta} respectively, and require that $C_\ddagger\ge  c_{\star}^{-1}e^{10d\Cref{const_approx_theta}}$. We take $n_0=\lambda = C_\ddagger\epsilon^{-1}$ and $K=\lfloor C_\ddagger \ln^2(1/\epsilon) \rfloor $. By (\ref{def_ni}) we have 
\begin{equation*}
	n_i\le \lambda^2K^4n_{i-1}^{1\boxdot \frac{d-4}{2} }\le n_{i-1}^{2d}.
\end{equation*}
Therefore, by letting $C_\ddagger$ be sufficiently large, one has 
\begin{equation}
	m= n_{6K+10}\le n_0^{(2d)^{6K+10}} < \mathrm{exp}(C_\ddagger e^{\epsilon^{-1}}),\end{equation}
and hence, $0\le h< \mathrm{exp}(-C_\ddagger e^{\epsilon^{-2}}) \le m^{-d^5}$ and $M_\ddagger\ge m^{d^5}$. Recall the integral $\mathbb{I}(\cdot )$ below (\ref{use3.55}). Thus, by Lemma \ref{lemma_approx_theta} with $\delta=\tfrac{1}{10} \epsilon c_{\star}$, we have 
\begin{equation}\label{3.92}
\begin{split}
	&\big|\mathbb{I}( \mathsf{A}_h , n_0,\lambda,K,\diamond) - \mathbb{P}\big( \mathsf{A}_h \cap \mathsf{H}_h^{\diamond}\big)  \big|\\
	\le & \tfrac{1}{10} \epsilon  \cdot c_{\star} \mathbb{P}( \mathsf{H}_h^{\diamond})\overset{(\text{FKG}),(\ref*{def_c_star})}{\le }\tfrac{1}{10} \epsilon 	 \cdot  \mathbb{P}\big( \mathsf{A}_h \cap \mathsf{H}_h^{\diamond}\big).
\end{split}	
		\end{equation}
As a special case when $\mathsf A_h=\Omega$, we get that
\begin{equation}\label{3.93}
		\big|\mathbb{I}(\Omega, n_0,\lambda,K,\diamond) - \mathbb{P}\big(  \mathsf{H}_h^{\diamond}\big)  \big| \le \tfrac{1}{10} \epsilon  \cdot   \mathbb{P}( \mathsf{H}_h^{\diamond}).
		\end{equation}
Combining (\ref{3.92}) and (\ref{3.93}), we get 
\begin{equation}\label{renew3.95}
	\Big|\mathbb{P}\big( \mathsf{A}_h \mid  \mathsf{H}_h^{\diamond}\big) - \frac{\mathbb{I}(\mathsf{A}_h, n_0,\lambda,K,\diamond) }{\mathbb{I}(\Omega, n_0,\lambda,K,\diamond) } \Big| \le \frac{\epsilon}{2}.
\end{equation}

For brevity, we denote $\mathsf{A}^{[1]}_h=\mathsf{A}_h$ and $\mathsf{A}^{[2]}_h=\Omega$. For $j\in \{1,2\}$ and $1 \le J\le K$, we define the function $f^{[j]}_J(D_{J})$ as 
\begin{equation*}
\begin{split}
		\int_{\widehat{\mathcal{C}}_k\neq \emptyset,\forall 0\le k\le J-1;\widehat{\mathsf{H}}_{h,K}^{\diamond}}  &\mathbb{P}\big(\mathsf{A}_h^{[j]} ,\bm{0} \xleftrightarrow{(\widecheck{D}_{0})}D_{0} \big)\\
		&\cdot  \prod\nolimits_{0\le k\le J-1} \mathbb{P}\big(D_k \xleftrightarrow{(\widecheck{D}_k\cup \widecheck{D}_{k+1})}D_{k+1}  \big)d\widehat{\mathfrak{p}}_0(D_0)\cdots d\widehat{\mathfrak{p}}_{J-1}(D_{J-1}).
\end{split}
\end{equation*}
We consider $\mathbb{P}\big(D_k \xleftrightarrow{(\widecheck{D}_k\cup \widecheck{D}_{k+1})}D_{k+1}  \big)$ as a kernel with input $(D_k,D_{k+1})$, which by the quasi-multiplicativity in (\ref{QM_ineq_1}) satisfies the condition in (\ref{kappa}) with some $\kappa>0$ depending only on $d$ (note that on the event $\{\widehat{\mathcal{C}}_k\neq \emptyset,\widehat{\mathcal{C}}_{k+1}\neq \emptyset \}$, one has $\widehat{\mathfrak{C}}_k\subset \widetilde{B}(n_{6k+4})$ and $\widehat{\mathfrak{C}}_{k+1}\subset [\widetilde{B}(n_{6k+6})]^c$, so that the conditions required for (\ref{QM_ineq_1}) are satisfied here). Thus, applying Lemma \ref{lemma_hopf} for $K-1$ times, we have 
\begin{equation}\label{3.96}
	\begin{split}
		\mathrm{osc}\big(\frac{f^{[1]}_K}{f^{[2]}_K} \big)\le  \big( \frac{\kappa-1}{\kappa+1} \big)^{K-1} \mathrm{osc}\big(\frac{f^{[1]}_1}{f^{[2]}_1} \big) \le \big( \frac{\kappa-1}{\kappa+1} \big)^{K-1}< \frac{\epsilon}{2}.
	\end{split}
\end{equation}
where the second inequality follows from $\frac{f^{[1]}_1}{f^{[2]}_1} \in [0,1]$, and the last inequality is ensured by $K=\lfloor C_\ddagger \ln^2(1/\epsilon) \rfloor >\log_{\frac{\kappa+1}{\kappa-1}}(2\epsilon^{-1})+1$ (where the inequality holds as long as $C_\ddagger$ is sufficiently large). By (\ref{3.96}), there exists $\zeta\in [0,1]$ such that 
\begin{equation*}
	\Big|\frac{\mathbb{I}(\mathsf{A}_h, n_0,\lambda,K,\diamond) }{\mathbb{I}(\Omega, n_0,\lambda,K,\diamond) } -\zeta  \Big|< \frac{\epsilon}{2}. 
\end{equation*}
Combined with (\ref{renew3.95}), it confirms (\ref{new3.39}) in Case 2.

In conclusion, we now complete the proof of Proposition \ref{proof_prop_converge_loop} and consequently, establish Theorem \ref{thm_iic}.    \qed

\section{Two point function under conditioning of connectivity}\label{section_iic_twopoint}

This section focuses on the proof of Theorem \ref{thm_1.4}. Specifically, in Section \ref{section3.1_lower} we prove the lower bound in Theorem \ref{thm_1.4} by analyzing the harmonic average for a specific explored cluster. Subsequently, in Section \ref{section3.2_upper} we establish the upper bound from the perspective of the loop soup by using cluster decomposition arguments.

\subsection{Proof of the lower bound}\label{section3.1_lower}

For any $n\ge 1$, we define $\mathcal{C}^{\ge 0}_{n}$ as the positive cluster of $\bm{0}$ obtained by exploration restricted to $\widetilde{B}(n)$. I.e., 
\begin{equation}\label{def_3.1}
	\mathcal{C}^{\ge 0}_{n}:= \big\{v\in \widetilde{B}(n): v\xleftrightarrow{\widetilde{E}^{\ge 0}\cap \widetilde{B}(n)} \bm{0} \big\}.
\end{equation}
Based on $\mathcal{C}^{\ge 0}_{n}$, we denote the following harmonic average (recall $\mathcal{H}_{v}(\cdot)$ in (\ref{def_Hv})):
\begin{equation}\label{finaluse_def_3.1}
	\mathcal{H}^{\ge 0}_{n}:=|\partial \mathcal{B}(dn)|^{-1}\sum\nolimits_{y\in \partial \mathcal{B}(dn)}\mathcal{H}_{y}(\mathcal{C}^{\ge 0}_{n}).
\end{equation}
 Note that $\mathcal{H}^{\ge 0}_{n}>0$ if and only if $\bm{0}\xleftrightarrow{\ge 0} \partial B(n)$. For any $A\subset \mathbb{Z}^d\setminus B(2dn)$, since $\{\bm{0}\xleftrightarrow{\ge 0} x\}= \{\mathcal{C}^{\ge 0}_{n}\xleftrightarrow{\ge 0}x\}$ for all $x\in \mathbb{Z}^d\setminus B(2dn)$, we have 
 \begin{equation}\label{newineq_3.13}
\begin{split}
	\mathbb{E}\big[ \mathrm{vol}(\mathcal{C}(\bm{0})\cap A)\mid \mathcal{F}_{\mathcal{C}^{\ge 0}_{n}}  \big] =&  \sum\nolimits_{x\in A} \mathbb{P}\big(\mathcal{C}^{\ge 0}_{n}\xleftrightarrow{\ge 0} x\mid \mathcal{F}_{\mathcal{C}^{\ge 0}_{n}}  \big) \\
	\overset{\text{Lemma}\ \ref*{newlemma3.1}}{\asymp} &   \sum\nolimits_{x\in A}   \big(  |x|^{2-d}n^{d-2}\mathcal{H}^{\ge 0}_{n} \big)\land 1 \\
	\ge  &   \big( \mathcal{H}^{\ge 0}_{n}\land 1\big) n^{d-2} \sum\nolimits_{x\in A} |x|^{2-d}. 
\end{split}
\end{equation}
 In order to facilitate the application of the second moment method, we define 
\begin{equation}
	\mathbf{Q}(A):= \mathbb{E}\big[\big( \mathrm{vol}(\mathcal{C}(\bm{0})\cap A) \big)^2\big] = \sum\nolimits_{x,y\in A} \mathbb{P}\big( \bm{0} \xleftrightarrow{\ge 0} x,y \big). 
\end{equation}

\begin{lemma}\label{lemma_Q_BN}
	For any $d\ge 3$ with $d\neq 6$ and any $N\ge 1$, 
	\begin{equation}\label{ineq_lemma_QBN}
		\mathbf{Q}(\partial B(N))\lesssim N^{(\frac{d}{2}+1)\boxdot 4}. 
	\end{equation}
\end{lemma}
\begin{proof}
When $3\le d \le 5$, it follows from (\ref{addnew2.33}) that 
\begin{equation}\label{newineq_3.38}
	\begin{split}
		\mathbf{Q}(\partial B(N)) \lesssim &  \sum\nolimits_{x,y\in \partial B(N)}|x-y|^{-\frac{d}{2}+1} |x|^{2-d}\\
		\lesssim & |\partial B(N)| \cdot N^{2-d} \max_{x\in\partial B(N) }  \sum\nolimits_{ y\in \partial B(N)} |x-y|^{-\frac{d}{2}+1}\\
		\overset{}{\lesssim} & N\cdot N^{\frac{d}{2}}=N^{\frac{d}{2}+1},
	\end{split}
\end{equation}
where the inequality on the last line follows from a straightforward computation (see e.g., \cite[Inequality (4.4)]{cai2024high}). Meanwhile, when $d\ge 7$, it has been proved in \cite[Lemma 4.1]{cai2024high} that $\mathbf{Q}(\partial B(N)) \lesssim  N^4$. Now we complete the proof of this lemma. 
\end{proof}


We denote the conditional expectation of $\mathbf{Q}(A)$ given $\mathcal{F}_{\mathcal{C}^{\ge 0}_{n}}$ by 
\begin{equation}
	\mathcal{Q}^{\ge 0}_{n}(A):=  \mathbb{E}\big[\mathbf{Q}(A) \mid \mathcal{F}_{\mathcal{C}^{\ge 0}_{n}}  \big].
\end{equation}
Note that $\mathcal{Q}^{\ge 0}_{n}(A)>0$ if and only if $\bm{0}\xleftrightarrow{\ge 0} \partial B(n)$.

%
%
%


\begin{definition}
	For any $n\ge 1$, $A_1,...,A_k\subset \mathbb{Z}^d\setminus B(2dn)$ ($k\in \mathbb{N}^+$) and $a,b\ge 0$, we say $\mathcal{C}^{\ge 0}_{n}$ is $(a,b)$-good in $A_1,...,A_k$ if the following event happens:
	\begin{equation}\label{defKab}
		\mathsf{K}_{a,b}(n;A_1,...,A_k):=  \mathsf{H}_{a}(n) \cap \big[  \cap_{i=1}^{k}  \mathsf{Q}_b(n,A_i) \big],
	\end{equation}
	where the events $\mathsf{H}_{a}$ and $\mathsf{Q}_b$ are defined as 
	\begin{equation}\label{event_Ha}
		\mathsf{H}_{a}(n):=\big\{ 	\mathcal{H}^{\ge 0}_{n} \ge a n^{2-d} [\theta_d(n)]^{-1} \big\} ,  
	\end{equation}
	\begin{equation}\label{newQb}
		 \mathsf{Q}_b(n,A_i):=\big\{   \mathcal{Q}^{\ge 0}_{n}(A_i)\le b[\theta_d(n)]^{-1}  \mathbf{Q}(A_i)\big\}.
	\end{equation}
	\end{definition}

 For any $N\ge d^2n$, it follows from \cite[Lemma 2.8]{inpreparation} that
\begin{equation}\label{addnew_3.8}
	\mathbb{P}\big([\mathsf{H}_{a}(n)]^c, \bm{0}\xleftrightarrow{\ge 0} \partial B(N) \big) \lesssim a\cdot \theta_d(N). 
\end{equation}
Especially, by taking $N=d^2n$ in (\ref{addnew_3.8}), we have
\begin{equation}\label{lower_Ha}
	\begin{split}
		\mathbb{P}\big(\mathsf{H}_{a}(n) \big)\ge  & \mathbb{P}\big(\mathsf{H}_{a}(n), \bm{0}\xleftrightarrow{\ge 0} \partial B(d^2n) \big) \\
		=& \mathbb{P}\big(\bm{0}\xleftrightarrow{\ge 0} \partial B(d^2n) \big)- \mathbb{P}\big([\mathsf{H}_{a}(n)]^c, \bm{0}\xleftrightarrow{\ge 0} \partial B(d^2n) \big)\\
		\ge & (1-Ca)\theta_d(d^2n)\gtrsim \theta_d(n) 
			\end{split}
\end{equation}
for all sufficiently small $a>0$ (i.e., $0<a<c$ for some constant $c>0$). Moreover, for any $n\ge 1$ and  $A\subset \mathbb{Z}^d\setminus B(2dn)$, by (\ref{newQb}) one has 
\begin{equation}\label{finaluse_upper_Qb}
	\mathbf{Q}(A) \ge \mathbb{E}\big[ \mathcal{Q}^{\ge 0}_{n}(A) \mathbbm{1}_{[\mathsf{Q}_b(n,A)]^c} \big]  \ge b [\theta_d(n)]^{-1}   \mathbf{Q}(A) \cdot \mathbb{P}\big([\mathsf{Q}_b(n,A)]^c\big),
	\end{equation}
which implies that 
\begin{equation}\label{upper_Qb}
	\mathbb{P}\big([\mathsf{Q}_b(n,A)]^c\big) \le b^{-1} \theta_d(n). 
\end{equation}
Combining (\ref{lower_Ha}) and (\ref{upper_Qb}), we know that there exist $\cl\label{const_K_1}(d),\Cl\label{const_K_2}(d)>0$ such that for any $n\ge 1$, $k\in \mathbb{N}^+$ and $A_1,...,A_k\subset \mathbb{Z}^d\setminus B(2dn)$, 
\begin{equation}\label{newineq_3.24}
		\mathbb{P}\big(\mathsf{K}_{\cref{const_K_1},\Cref{const_K_2}k}(n;A_1,...,A_k) \big) \gtrsim  \theta_d(n). 
\end{equation}
For simplicity, we abbreviate $\mathsf{K}(n;A_1,...,A_k):= \mathsf{K}_{\cref{const_K_1},\Cref{const_K_2}k}(n;A_1,...,A_k)$.




Now we apply the second moment method to lower-bound $\mathrm{vol}(\mathcal{C}(\bm{0})\cap A_i)$ conditioned on $\mathsf{K}(n;A_1,...,A_k)$. To this end, we denote 
\begin{equation}\label{def_R_nA}
	\mathbf{R}(n,A):= [\theta_d(n) ]^{-1} \sum\nolimits_{x\in A} |x|^{2-d}, \ \ \forall n\ge 1,A\subset \mathbb{Z}^d.  
\end{equation}
On $\mathsf{K}(n;A_1,...,A_k)$, for each $1\leq i\leq k$, it follows from (\ref{newineq_3.13}) that
 \begin{equation}
		\mathbb{E}\big[ \mathrm{vol}(\mathcal{C}(\bm{0})\cap A_i)\mid \mathcal{F}_{\mathcal{C}^{\ge 0}_{n}}  \big] \gtrsim \mathbf{R}(n,A_i).  
\end{equation}
Thus, by the Paley-Zygmund inequality, we have
\begin{equation}\label{newineq3.27}
\begin{split}
	\mathbb{P}\big(\mathrm{vol}(\mathcal{C}^{\ge 0}(\bm{0})\cap A_i) \ge c\mathbf{R}(n,A_i) \mid \mathcal{F}_{\mathcal{C}^{\ge 0}_{n}} \big)
	\gtrsim   \frac{ \big[ \mathbf{R}(n,A_i) \big]^2 }{ \mathcal{Q}^{\ge 0}_{n}(A_i) }  \gtrsim \frac{ \big(\sum\limits_{x\in A} |x|^{2-d} \big)^2 }{  \theta_d(n) k \mathbf{Q}(A_i) }.
\end{split}
\end{equation}

With the preparations above, we are ready to establish the desired lower bound.

\begin{proof}[Proof of the lower bound in Theorem \ref{thm_1.4}]
 For any $M\ge 1$, $y\in \partial B(M)$ and $N\ge CM$, by taking $n=d^{-2}M$, $A_1=\{y\}$ and $A_2=\partial B(N)$ in (\ref{newineq_3.24}), one has 
 \begin{equation}\label{newineq3.28}
 	\begin{split}
 	\mathbb{P}(\mathsf{A}):= 		\mathbb{P}\big(\mathsf{K}(d^{-2}M;\{y\},\partial B(N)) \big) \gtrsim  \theta_d(d^{-2}M). 
 	\end{split}
 \end{equation}
Moreover, on the event $\mathsf{A}$, by (\ref{newineq3.27}) (and the simple fact that the volume of a set is lower-bounded by a positive number implies that this set is non-empty), we have 
\begin{equation}
	\begin{split}
		\mathbb{P}\big( \bm{0}\xleftrightarrow{\ge 0} y  \mid \mathcal{F}_{\mathcal{C}^{\ge 0}_{d^{-2}M}} \big) \gtrsim M^{2-d} [\theta_d(d^{-2}M)]^{-1}, 
	\end{split}
\end{equation}
\begin{equation}\label{newineq3.30}
	\begin{split}
		\mathbb{P}\big( \bm{0}\xleftrightarrow{\ge 0} \partial B(N)  \mid \mathcal{F}_{\mathcal{C}^{\ge 0}_{d^{-2}M}} \big)  \gtrsim \frac{N^2}{\theta_d(d^{-2}M)\mathbf{Q}(\partial B(N)) }. 
	\end{split}
\end{equation}
By applying the FKG inequality and using (\ref{newineq3.28})--(\ref{newineq3.30}), we have 
\begin{equation}\label{add3.37}
	\begin{split}
		&\mathbb{P}\big( \bm{0}\xleftrightarrow{\ge 0} y, \bm{0}\xleftrightarrow{\ge 0} \partial B(N)  \big)\\
		\ge & \mathbb{E}\Big[\mathbbm{1}_{\mathsf{A}}\cdot  \mathbb{P}\Big( \bm{0}\xleftrightarrow{\ge 0} y, \partial B(N)  \mid \mathcal{F}_{\mathcal{C}^{\ge 0}_{d^{-2}M}} \Big) \Big] \\
		\overset{(\mathrm{FKG})}{\ge } & \mathbb{E}\Big[\mathbbm{1}_{\mathsf{A}}\cdot  \mathbb{P}\Big( \bm{0}\xleftrightarrow{\ge 0} y  \mid \mathcal{F}_{\mathcal{C}^{\ge 0}_{d^{-2}M}} \Big) \cdot \mathbb{P}\Big( \bm{0}\xleftrightarrow{\ge 0}   \partial B(N)  \mid \mathcal{F}_{\mathcal{C}^{\ge 0}_{d^{-2}M}} \Big) \Big]\\
		\overset{(\ref*{newineq3.28})\text{--}(\ref*{newineq3.30})}{\gtrsim } & \frac{M^{2-d}N^2}{\theta_d(d^{-2}M)\mathbf{Q}(\partial B(N)) }\overset{(\ref*{one_arm_low}), (\ref*{one_arm_high}), \text{Lemma}\ \ref*{lemma_Q_BN}}{\gtrsim } M^{-[(\frac{d}{2}-1)\boxdot (d-4)]}\theta_d(N). \end{split}
\end{equation}
Dividing both sides by $\theta_d(N)$, we conclude the desired lower bound. 
\end{proof}


\subsection{Proof of the upper bound}\label{section3.2_upper}

For any $M\ge 1$, $y\in \partial B(M)$ and $N\ge CM$, note that $\{\bm{0}\xleftrightarrow{\ge 0} y, \partial B(N)\}\subset \{ \partial B(N) \xleftrightarrow{\ge 0 } \bm{0}, y \}$. Therefore, when $3\le d\le 5$, it directly follows from (\ref{addnew2.34}) that 
\begin{equation}
	\begin{split}
		\mathbb{P}\big( \bm{0}\xleftrightarrow{\ge 0} y ,\partial B(N)  \big) \le \mathbb{P}\big(  \partial B(N) \xleftrightarrow{\ge 0 } \bm{0}, y \big) \lesssim   M^{-\frac{d}{2}+1} \theta_d(N),
	\end{split}
\end{equation}
which implies the upper bound in Theorem \ref{thm_1.4}.

In what follows, we focus on the case when $d\ge 7$. By the isomorphism theorem, it suffices to prove that 
\begin{equation}\label{submit425}
	\mathbb{P}\big( \bm{0}\xleftrightarrow{ } y,\partial B(N)  \big) \lesssim M^{4-d}N^{-2}. 
\end{equation}
Let $\widetilde{\mathcal{C}}_*$ denote the cluster of $\cup \widetilde{\mathcal{L}}_{1/2}\cdot \mathbbm{1}_{\mathrm{ran}(\widetilde{\ell})\cap \mathbb{Z}^d \subset \mathcal{B}(N/2)}$ containing $\bm{0}$. We also denote by $\mathfrak{C}_*$ the union of loop clusters intersecting $\partial B(N)$ and composed of loops not included in $\widetilde{\mathcal{C}}_*$. On the event $ \{\bm{0}\xleftrightarrow{(\partial \mathcal{B}(N/2))} y \}\cap \{  \bm{0}\xleftrightarrow{ }  \partial B(N)  \}$, one has $\bm{0},y\in \widetilde{\mathcal{C}}_*$ and $\widetilde{\mathcal{C}}_*\cap \mathfrak{C}_* \neq \emptyset$. Therefore, by the tree expansion argument (see e.g., \cite[Lemma 3.5]{cai2024high}), there exists $\widetilde{\ell}\in \widetilde{\mathcal{L}}_{1/2}$ with $\mathrm{ran}(\widetilde{\ell})\cap \mathbb{Z}^d \subset \mathcal{B}(N/2)$ such that $\mathrm{ran}(\widetilde{\ell})\xleftrightarrow{} \bm{0}$, $\mathrm{ran}(\widetilde{\ell})\xleftrightarrow{} y$, and $\mathrm{ran}(\widetilde{\ell})\xleftrightarrow{} \mathfrak{C}_*$ (which implies $\mathrm{ran}(\widetilde{\ell}) \xleftrightarrow{} \partial B(N)$) happen disjointly. Thus, by the BKR inequality (see \cite[Corollary 3.4]{cai2024high}) and \cite[Lemma 2.3]{cai2024high}, we have 
\begin{equation}\label{newineq_3.37}
	\begin{split}
		&\mathbb{P}\big( \bm{0}\xleftrightarrow{(\partial \mathcal{B}(N/2))} y,\bm{0}\xleftrightarrow{ }  \partial B(N)  \big) \\
		\le &\sum_{w_1,w_2,w_3\in \mathcal{B}(N/2)}  \mathbb{P}\big(\exists \widetilde{\ell}\in \widetilde{\mathcal{L}}_{1/2}\ \text{with}\ \widetilde{B}_{w_i}(1)\cap  \mathrm{ran}(\widetilde{\ell})\neq \emptyset, \forall 1\le i\le 3  \big)\\
		&\ \ \ \ \ \ \ \ \ \ \ \ \ \ \ \ \ \ \cdot \mathbb{P}(\widetilde{B}_{w_1}(1)\xleftrightarrow{} \bm{0} ) \mathbb{P}(\widetilde{B}_{w_2}(1)\xleftrightarrow{} y ) \mathbb{P}(\widetilde{B}_{w_3}(1)\xleftrightarrow{} \partial B(N) ) \\
		\lesssim & \sum_{w_1,w_2,w_3\in \mathcal{B}(N/2)}  |w_1-w_2|^{2-d } |w_2-w_3|^{2-d } |w_3-w_1|^{2-d }\\
		&\ \ \ \ \ \ \ \ \ \ \ \ \ \ \ \ \ \ \cdot  |w_1|^{2-d}|w_2-y|^{2-d} N^{-2}\\
		\lesssim  &  M^{4-d}N^{-2}, 
	\end{split}
\end{equation}
where the last inequality follows from straightforward computations (e.g., by applying \cite[Inequality (4.11) and Lemma 4.3]{cai2024high} in turn).




Next, we estimate the probability of $\{\bm{0}\xleftrightarrow{(\partial \mathcal{B}(N/2))} y \}^c \cap \{  \bm{0}\xleftrightarrow{ } y , \partial B(N)   \} $. To this end, we employ the notation $\widetilde{\mathcal{C}}_*$ as defined in the previous paragraph. When $\{\bm{0}\xleftrightarrow{(\partial \mathcal{B}(N/2))} y \}^c \cap \{  \bm{0}\xleftrightarrow{ } y, \partial B(N)   \} $ happens, we know that $y\notin \widetilde{\mathcal{C}}_*$, and that there exists a loop $\widetilde{\ell}_*$ that intersects $\widetilde{\mathcal{C}}_*$ and can be connected to $y$ without using any loop included in $\widetilde{\mathcal{C}}_*$. Therefore, there exist $w_1\in \mathcal{B}(N/2), w_2 \in \mathbb{Z}^d$ such that $\widetilde{\ell}_*$ intersects $\widetilde{B}_{w_i}(1)$ for $i\in \{1,2\}$ and that $\{\widetilde{B}_{w_1}(1)\xleftrightarrow{(\partial \mathcal{B}(N/2))} \bm{0}\} \circ \{\widetilde{B}_{w_2}(1)\xleftrightarrow{} y \}$ happens without using $\widetilde{\ell}_*$. This further implies that there exists $w_3\in  \partial \mathcal{B}(N/2)$ such that $\widetilde{\ell}_*$ contains a sub-path starting from $w_3$ and hitting $\widetilde{B}_{w_1}(1)$ before $\partial \mathcal{B}(N/2+1)$. In addition, by \cite[Inequalities (6.21) and (6.22)]{inpreparation}, the total loop measure of such loops is at most 
\begin{equation}
	C|w_1-w_2|^{2-d}|w_2-w_3|^{2-d}|w_3-w_1|^{2-d} (\tfrac{N}{2}-|w_1|)^{-1}. 
\end{equation}
Meanwhile, it follows from \cite[Inequality (6.24)]{inpreparation} that 
\begin{equation}
	\mathbb{P}\big(\widetilde{B}_{w_1}(1)\xleftrightarrow{\partial \mathcal{B}(N/2)} \bm{0} \big)\lesssim (\tfrac{N}{2}-|w_1|)|w_1|^{1-d}. 
\end{equation}
To sum up, by the BKR inequality (see \cite[Corollary 3.4]{cai2024high}) we have 
\begin{equation}\label{newineq_3.40}
	\begin{split}
	&	\mathbb{P}\big(\{\bm{0}\xleftrightarrow{(\partial \mathcal{B}(N/2))} y \}^c \cap \{  \bm{0}\xleftrightarrow{ } y , \partial B(N)\} \big)\\
		\lesssim & \sum\nolimits_{i\in \{1,2,3\}} \mathbb{I}_i:= \sum\nolimits_{i\in \{1,2,3\}}\mathbb{S}(D_i^{(1)}\times D_i^{(2)}\times \partial \mathcal{B}(N/2)),
	\end{split}
\end{equation}
where $\mathbb{S}(\cdot )$ and $D_i^{(j)}$ for $i\in \{1,2,3\},j\in \{1,2\}$ are defined as
\begin{equation*}
	\mathbb{S}(\mathfrak{D}):= \sum_{(w_1,w_2,w_3)\in \mathfrak{D}}|w_1-w_2|^{2-d}|w_2-w_3|^{2-d}|w_3-w_1|^{2-d} |w_1|^{1-d}|w_2-y|^{2-d},
\end{equation*} 
\begin{equation*}
(D_i^{(1)},D_i^{(2)}):= 	\left\{
	\begin{array}{lr}
		\big(B(d^{-2}N), B_y(d^{-2}N) \big)  \ \ & \text{when}\ i=1;\\
		\big(\mathcal{B}(N/2)\setminus B(d^{-2}N),\mathbb{Z}^d\big) \ \ &\text{when}\ i=2;\\
		\big(B(d^{-2}N), [B_y(d^{-2}N)]^c\big)\ \ &\text{when}\ i=3.
	\end{array}
	\right.
\end{equation*}
Next, we estimate $\mathbb{I}_i$ for $i\in \{1,2,3\}$ separately.



 \textbf{For $\mathbb{I}_1$:} In this case, since $|w_i-w_3|\ge |w_3|-|w_1|\gtrsim N$ for $i\in \{1,2\}$, we have
 \begin{equation}\label{finish4.47}
 	\begin{split}
 		\mathbb{I}_1 \lesssim & N^{4-2d}|\partial \mathcal{B}(N/2)|\sum\nolimits_{(w_1,w_2)\in D_1^{(1)}\times D_1^{(2)}} |w_1-w_2|^{2-d}|w_1|^{1-d}|w_2-y|^{2-d}\\
 		\overset{ }{\lesssim} & N^{3-d} \sum\nolimits_{w_1\in B(d^{-2}N)}  |w_1-y|^{4-d}|w_1|^{1-d},
 		 	\end{split}
 \end{equation}
 where we used \cite[Lemma 4.3]{cai2024high} in the last inequality (which is nothing but a straightforward computation). Note that $|w_1-y|\gtrsim M$ for all $w_1\in B(cM)$. Therefore, when restricting the sum on the right-hand side of (\ref{finish4.47}) to $B(cM)$, we have
 \begin{equation}\label{finish4.48}
 \begin{split}
 		& \sum\nolimits_{w_1\in B(cM)}  |w_1-y|^{4-d}|w_1|^{1-d}\\
 		\lesssim  &M^{4-d}\sum\nolimits_{w_1\in B(cM)} |w_1|^{1-d}  \overset{(\ref*{newcite3.52})}{\lesssim } M^{5-d}. 
 \end{split}
 \end{equation}
 In addition, when restricting the same sum to $B(d^{-2}N)\setminus B(cM)$ (note that $|w_1|\gtrsim M$ for all $w_1\in B(d^{-2}N)\setminus B(cM)$), it follows from a straightforward computation (e.g., by applying \cite[(4.10)]{cai2024high}) that  
 \begin{equation}\label{finish4.49}
 \begin{split}
 	&\sum\nolimits_{w_1\in B(d^{-2}N)\setminus B(cM)}  |w_1-y|^{4-d}|w_1|^{1-d}\\
 	\lesssim 	&M^{-1}\sum\nolimits_{w_1\in B(d^{-2}N)\setminus B(cM)}  |w_1-y|^{4-d}|w_1|^{2-d}\lesssim M^{5-d}. 
 \end{split}
 \end{equation}
 Putting (\ref{finish4.47})--(\ref{finish4.49}) together, we get 
 \begin{equation}
 	\mathbb{I}_1 \lesssim  N^{3-d}M^{5-d}.  
 \end{equation}

 \textbf{For $\mathbb{I}_2$:} Since $|w_1|\gtrsim N$ for all $w_1\in [B(d^{-2}N)]^c$, one has 
 \begin{equation}
 	\begin{split}
 		\mathbb{I}_2 \lesssim & N^{1-d} \sum\nolimits_{(w_1,w_2,w_3)\in  D_2^{(1)}\times D_2^{(2)}\times \partial \mathcal{B}(N/2)} |w_1-w_2|^{2-d}|w_2-w_3|^{2-d}\\
 		&\cdot |w_3-w_1|^{2-d} |w_2-y|^{2-d}\\
 		\overset{}{\lesssim} &  N^{1-d} \sum\nolimits_{w_3\in \partial \mathcal{B}(N/2) }  |w_3-y|^{2-d} \\
 		\lesssim & N^{1-d} \cdot |\partial \mathcal{B}(N/2)|\cdot N^{2-d} \lesssim N^{2-d}, 
 	\end{split}
 \end{equation}
  where in the second inequality we used \cite[Inequality (4.11)]{cai2024high} (another straightforward computation).

 \textbf{For $\mathbb{I}_3$:} Noting that $|w_2-y|\gtrsim N$ for all $w_2\in [B_y(d^{-2}N)]^c$, we have 
 \begin{equation}\label{newineq_3.43}
 	\begin{split}
 			\mathbb{I}_3 \lesssim &  N^{2-d} \sum\nolimits_{(w_1,w_2,w_3)\in  D_3^{(1)}\times D_3^{(2)}\times \partial \mathcal{B}(N/2)} |w_1-w_2|^{2-d}|w_2-w_3|^{2-d}\\
 			&\cdot |w_3-w_1|^{2-d} |w_1|^{1-d}\\
 			\overset{ }{\lesssim} & N^{2-d} \sum\nolimits_{w_1\in B(d^{-2}N), w_3 \in  \partial \mathcal{B}(N/2)}  |w_3-w_1|^{6-2d} |w_1|^{1-d} \overset{ }{\lesssim} N^{2-d},
 			 	\end{split} 
 \end{equation}
 where for the inequalities in the last line we used \cite[Lemma 4.3]{cai2024high} and \cite[Lemma 2.21]{inpreparation} respectively (which are straightforward computations again).



Recall that $N\ge CM$. Thus, putting (\ref{newineq_3.40})--(\ref{newineq_3.43}) together, we obtain
\begin{equation}
\begin{split}
	&\mathbb{P}\big(\{\bm{0}\xleftrightarrow{(\partial \mathcal{B}(N/2))} y \}^c \cap \{  \bm{0}\xleftrightarrow{ } y,\partial B(N)\} \big) \\
\lesssim & N^{3-d}M^{5-d} + N^{2-d} \lesssim M^{4-d}N^{-2}. 
\end{split}	
\end{equation}
Combined with (\ref{newineq_3.37}), it implies (\ref{submit425}), and thus confirms Theorem \ref{thm_1.4}.   \qed

\section{Typical volumes of critical clusters}\label{section_order_typical_volume}


This section is mainly devoted to the proofs of Theorems \ref{thm_1.2} and \ref{thm_1.3}, whose details are provided in Sections \ref{subsection_proof_thm1.2} and \ref{subsection_proof_thm1.3} respectively.

\subsection{Proof of Theorem \ref{thm_1.2}}\label{subsection_proof_thm1.2}
Arbitrarily take $\epsilon>0$, $M\ge 1$ and $N\ge CM$. To establish Theorem \ref{thm_1.2}, it suffices to prove the following two inequalities:
\begin{equation}\label{newineq4.1}
	\widehat{\mathbb{P}}_{d,N}\big(  \mathcal{V}_{M}^{\ge 0} \ge \Cref{const_typical_volume_2}M^{(\frac{d}{2}+1)\boxdot 4} \big) \le \tfrac{1}{2}\epsilon, 
\end{equation} 
\begin{equation}\label{newineq4.2}
	\widehat{\mathbb{P}}_{d,N}\big(  \mathcal{V}_{M}^{\ge 0}\le \cref{const_typical_volume_3}M^{(\frac{d}{2}+1)\boxdot 4}  \big)\le \tfrac{1}{2}\epsilon,
\end{equation} 
where the constants $\Cref{const_typical_volume_2}(d,\epsilon),\cref{const_typical_volume_3}(d,\epsilon)>0$ will be determined later.

The first inequality (\ref{newineq4.1}) can be derived from the Markov's inequality and Theorem \ref{thm_1.4}. Precisely, we denote by $\widehat{\mathbb{E}}_{d,N}$ the expectation under $\widehat{\mathbb{P}}_{d,N}$. By the upper bound in Theorem \ref{thm_1.4}, we have 
\begin{equation}
	\begin{split}
		\widehat{\mathbb{E}}_{d,N}\big[ \mathcal{V}_{M}^{\ge 0}\big]\lesssim \sum\nolimits_{x\in B(M)} |x|^{-[(\frac{d}{2}-1)\boxdot (d-4)]} \overset{(\ref*{newcite3.52})}{\lesssim } M^{(\frac{d}{2}+1)\boxdot 4}. 
	\end{split}
\end{equation}
Thus, by taking a sufficiently large $\Cref{const_typical_volume_2}(d,\epsilon)$ and applying the Markov's inequality, we obtain (\ref{newineq4.1}): 
\begin{equation}\label{newineq4.4}
		\widehat{\mathbb{P}}_{d,N}\big(  \mathcal{V}_{M}^{\ge 0} \ge \Cref{const_typical_volume_2}M^{(\frac{d}{2}+1)\boxdot 4} \big)\le \tfrac{\widehat{\mathbb{E}}_{d,N}\big[ \mathcal{V}_{M}^{\ge 0}\big]}{\Cref{const_typical_volume_2}M^{(\frac{d}{2}+1)\boxdot 4}} \le \tfrac{1}{2}\epsilon. 
\end{equation}

Now we prove (\ref{newineq4.2}) separately for the cases $3\le d\le 5$ and $d\ge 7$. Without loss of generality, we assume that $\epsilon>0$ is sufficiently small and $M\ge e^{\epsilon^{-1}}$.

 \textbf{When $3\le d\le 5$.} Before delving into the details, we first provide some heuristics for this proof. Specifically, by the construction involved in the proof of Lemma \ref{lemma_approx_theta}, we know that the critical cluster connecting $\bm{0}$ and $\partial B(N)$ can be approximately sampled in the following two steps: 
 \begin{enumerate}
 
 	\item   Take a sequence of concentric annuli with increasing radii, and then sample the loop clusters crossing these annuli. On the event $\{\bm{0}\xleftrightarrow{} \partial B(N)\}$, for $3\le d\le 5$, if the outer radius of an annulus exceeds a large constant times its inner radius, then the cluster crossing this annulus is typically unique.

\textbf{P.S.:} For high-dimensional cases (i.e., $d\ge 7$), to ensure the uniqueness of the crossing cluster, according to (\ref{crossing_high}), the outer radius of the annulus must exceed the inner radius to the power of $\frac{d-4}{2}$ (which is strictly larger than $1$). Since the largest annulus must be contained in $B(N)\setminus B(cM)$ (to allow consideration of the connecting event between two clusters with distance $c'M$, which is necessary for obtaining a potentially large cluster volume; see the discussion below), this condition may be violated when $N$ is of the same order as $M$ (as allowed in Theorem \ref{thm_1.2}). This is why a different approach is needed for the cases when $d\ge 7$.

 \item   Connect these clusters sequentially using the loops not included in them.

 \end{enumerate}
Note that the connecting events for different pairs of neighboring clusters in Step (2) are independent since they are certified by disjoint collections of loops. Moreover, given these events, for every pair of neighboring clusters, there is a uniformly positive conditional probability that the sub-cluster connecting them has a significant volume (see Lemma \ref{finaluse_lemma_volume}). Thus, by considering sufficiently many annuli, we conclude that conditioned on hitting $\partial B(N)$, it is rare for the critical cluster $\mathcal{C}^{\ge 0}(\bm{0})$ to have a small volume.

 Building on the preceding heuristics, we now provide a detailed proof of (\ref{newineq4.2}) for $3\le d\le 5$. By the isomorphism theorem, it suffices to prove that 
 \begin{equation}\label{finaluse_5.5}
\mathbb{P}\big( \mathsf{H}^{\mathcal{V}} \big):=  	\mathbb{P} \big(  \mathcal{V}_{M} \le \cref{const_typical_volume_3}M^{\frac{d}{2}+1} , \bm{0}\xleftrightarrow{} \partial B(N) \big)\le \tfrac{1}{2}\epsilon\cdot \mathbb{P} \big(  \bm{0}\xleftrightarrow{} \partial B(N) \big),
 \end{equation}
 where $ \mathcal{V}_{M}:=\mathrm{vol}(\mathcal{C}(\bm{0})\cap B(M))$. Recall the notations $n_i=n_i(n_0,\lambda)$ and $m=m(n_0,\lambda,K)$ defined in (\ref{def_ni}) and (\ref{def_m}) respectively. Here we set $K=\lfloor \ln^{2}(1/\epsilon) \rfloor$ and require $m=M$. In addition, for each $j\in \{1,2\}$, we define $\mathsf{F}_{k}^{\mathcal{V},(j)}$ as the analogue of $\mathsf{F}_{k,h}^{\diamond,(j)}$ (see (\ref{def_F_diamond1}) and (\ref{def_F_diamond2})) obtained by replacing $\mathsf{H}_h^\diamond$ with $\mathsf{H}^{\mathcal{V}}$. Therefore, by $\mathsf{H}^{\mathcal{V}}\subset \mathsf{H}_0^{\mathbf{N}}=\{\bm{0} \xleftrightarrow{} \partial B(N)\}$, one has $\mathsf{F}_{k}^{\mathcal{V},(j)}\subset \mathsf{F}_{k,h}^{\mathbf{N},(j)}$ for all $j\in \{1,2\}$. Thus, it follows from (\ref{use3.49}) and (\ref{finaluse3.50}) that 
   \begin{equation}
 	\mathbb{P}\big(\mathsf{F}_{k}^{\mathcal{V},(1)}\big)   \overset{(\ref*{use3.49})}{\lesssim} \lambda^{-2} k^{-4} \mathbb{P}\big( \mathsf{H}_0^{\mathbf{N}} \big)\ \  \text{and}\ \ 	\mathbb{P}\big(\mathsf{F}_{k}^{\mathcal{V},(2)}\big) \overset{(\ref*{finaluse3.50})}{\lesssim} \lambda^{-1} k^{-2} \mathbb{P}\big( \mathsf{H}_0^{\mathbf{N}} \big). 
 	\end{equation}
 As a result, by taking $\lambda = C\epsilon^{-1}$ with a sufficiently large constant $C>0$, we know that in order to obtain (\ref{finaluse_5.5}), it suffices to show 
\begin{equation}\label{finaluse_5.7}
		\mathbb{P}\big( \mathsf{G}^{\mathcal{V}}\big) := \mathbb{P}\big( \mathsf{H}^{\mathcal{V}}  \setminus (\cup_{0\le k\le K} \mathsf{F}_{k}^{\mathcal{V},(1)}\cup \mathsf{F}_{k}^{\mathcal{V},(2)})\big)  \le  \tfrac{1}{4}\epsilon\cdot \mathbb{P} \big(  \bm{0}\xleftrightarrow{} \partial B(N) \big). 
\end{equation}

Recall the notations $\widehat{\mathfrak{C}}_k, \widehat{\mathcal{C}}_k$ and $\widecheck{\mathfrak{C}}_k$ for $0\le k\le K$ under the construction below (\ref{use3.53}). In addition, we denote $\widehat{\mathfrak{C}}_{-1}=\{\bm{0}\}$, $\widehat{\mathfrak{C}}_{K+1}=\partial B(N)$ and $\widecheck{\mathfrak{C}}_{-1}=\widecheck{\mathfrak{C}}_{K+1}=\emptyset$. For any $A,D\subset \widetilde{\mathbb{Z}}^d$ and $n\ge 2$, we denote 
\begin{equation}\label{finaluse_58}
	V(A,D,n):= \big\{x\in B(2n)\setminus B(\tfrac{1}{2}n): x\xleftrightarrow{(D)} A \big\}. 
\end{equation}
For $-1\le k\le K$, we define $V_k:=V(\widehat{\mathfrak{C}}_k ,\widecheck{\mathfrak{C}}_k\cup \widecheck{\mathfrak{C}}_{k+1} , n_{6k+5})$ and consider the event $\mathsf{V}_k:=\{ |V_k| \le  \cref{const_typical_volume_3}M^{\frac{d}{2}+1}\}$. By the same argument as in proving (\ref{use3.55}), we know that
 \begin{equation}\label{finaluse_59}
	\mathsf{G}^{\mathcal{V}} \subset  \big( \cap_{0\le k\le K}\{ \widehat{\mathcal{C}}_k\neq \emptyset \} \big)  \cap \big( \cap_{-1\le k\le K} \{  \widehat{\mathfrak{C}}_k\xleftrightarrow{(\widecheck{\mathfrak{C}}_k\cup \widecheck{\mathfrak{C}}_{k+1} )}  \widehat{\mathfrak{C}}_{k+1}\}\big) .
\end{equation}
In addition, the event on the right-hand side of (\ref{finaluse_59}) implies that $V_k\subset \mathcal{C}(\bm{0})$ for all $-1\le k\le K$, which further yields that $\mathsf{G}^{\mathcal{V}} \subset   \cap_{-1\le k\le K} \mathsf{V}_k$. Conseuqently, similar to (\ref{use3.55}), we obtain that 
\begin{equation}\label{finaluse_5.10}
	\mathbb{P}\big( \mathsf{G}^{\mathcal{V}}\big) \le \mathbb{I}^{\mathcal{V}}(d,M,N,\epsilon), 
\end{equation} 
where $\mathbb{I}^{\mathcal{V}}(\cdot )$ is defined as (let $D_{-1}=\{\bm{0}\}$, $D_{K+1}=\partial B(N)$ and $\widecheck{D}_{-1}=\widecheck{D}_{K+1}=\emptyset$)
\begin{equation*}
 	\int_{\widehat{\mathcal{C}}_k\neq \emptyset,\forall 0\le k\le K}   \prod_{k=-1}^{K} \mathbb{P}\big(\mathsf{V}_k, D_k \xleftrightarrow{(\widecheck{D}_k\cup \widecheck{D}_{k+1})}D_{k+1}  \big)d\widehat{\mathfrak{p}}_0(D_0)\cdots d\widehat{\mathfrak{p}}_{K}(D_k).
\end{equation*}

The following lemma establishes that the sub-cluster connecting the clusters in two neighboring annuli is likely to have a significant volume. 
\begin{lemma}\label{finaluse_lemma_volume}
	For any $3\le d\le 5$, there exist $\Cl\label{const_finaluse_lemma_volume1},\cl\label{const_finaluse_lemma_volume2},\cl\label{const_finaluse_lemma_volume3}>0$ such that for any $n\ge 1$,  $A_1\subset \widetilde{B}(\Cref{const_finaluse_lemma_volume1}^{-1}n)$, $A_2\in [\widetilde{B}(\Cref{const_finaluse_lemma_volume1}n)]^c$ and  $D\subset \widetilde{B}(\Cref{const_finaluse_lemma_volume1}^{-1}n)\cup [\widetilde{B}(\Cref{const_finaluse_lemma_volume1}n)]^c$, 
	\begin{equation}\label{ineq_finaluse_lemma_volume}
	\mathbb{P}\big(A_1 \xleftrightarrow{(D)} A_2, |V(A_1,D,n)|\ge \cref{const_finaluse_lemma_volume2} n^{\frac{d}{2}+1} \big) \ge \cref{const_finaluse_lemma_volume3} \mathbb{P}\big(A_1 \xleftrightarrow{(D)} A_2  \big). 	\end{equation}
\end{lemma}
 Before presenting its proof, we first show how to derive (\ref{finaluse_5.7}) from Lemma \ref{finaluse_lemma_volume}. In fact, recalling that we set $\lambda=C\epsilon^{-1}$ and $m=M$ in (\ref{def_ni}) and (\ref{def_m}), one has $n_0\ge  c_*(d,\epsilon)M$ for some constant $c_*(d,\epsilon)>0$. As a result, by taking $\cref{const_typical_volume_3}=\frac{1}{2}\cref{const_finaluse_lemma_volume2}  c_*^{\frac{d}{2}+1}$, we have $\cref{const_typical_volume_3}M^{\frac{d}{2}+1}\le \frac{1}{2}\cref{const_finaluse_lemma_volume2} n_0^{\frac{d}{2}+1}$. Thus, for each $-1\le k\le K$, by Lemma \ref{finaluse_lemma_volume} one has
 \begin{equation}\label{finaluse512}
 	\mathbb{P}\big(\mathsf{V}_k, D_k \xleftrightarrow{(\widecheck{D}_k\cup \widecheck{D}_{k+1})}D_{k+1}  \big) \le (1-\cref{const_finaluse_lemma_volume3} ) \mathbb{P}\big( D_k \xleftrightarrow{(\widecheck{D}_k\cup \widecheck{D}_{k+1})}D_{k+1}  \big). \end{equation}
 Plugging (\ref{finaluse512}) into (\ref{finaluse_5.10}), we obtain (\ref{finaluse_5.7}) as follows (recall $\mathbb{I}(\cdot)$ below (\ref{use3.55})): 
 \begin{equation}
 \begin{split}
 		\mathbb{P}\big( \mathsf{G}^{\mathcal{V}}\big) \le & (1-\cref{const_finaluse_lemma_volume3} )^{K+2}\mathbb{I}( \Omega , n_0,\lambda,K,\mathbf{N}) \\
 		\overset{K=\lfloor \ln^{2}(1/\epsilon) \rfloor, (\ref*{finaluse_375})}{\le }  &  \tfrac{1}{4}\epsilon\cdot \mathbb{P} \big(  \bm{0}\xleftrightarrow{} \partial B(N) \big). 
 \end{split}
 \end{equation}
 To sum up, we confirm (\ref{newineq4.2}) for $3\le d\le 5$ assuming Lemma \ref{finaluse_lemma_volume} holds.

  To establish Lemma \ref{finaluse_lemma_volume}, we need the following lemma as preparation. Before stating this lemma, we first introduce some notations as follows. For any $n\ge 1$ and $A\subset \widetilde{B}(d^{-4}n)$, we denote the cluster
  \begin{equation}\label{finaluse_514}
  	 \mathcal{C}^{\mathrm{in}}(A,n):=\big\{v\in \widetilde{B}(d^{-2}n): v\xleftrightarrow{\widetilde{E}^{\ge 0}\cap \widetilde{B}(d^{-2}n)} A\ \text{or}\ v\xleftrightarrow{\widetilde{E}^{\le 0}\cap \widetilde{B}(d^{-2}n)} A \big\}. 
  \end{equation}  
 Based on this cluster, we consider the harmonic average 
 \begin{equation}\label{finaluse_515}
 		\mathcal{H}^{\mathrm{in}}_D(A,n):=|\partial \mathcal{B}(d^{-1}n)|^{-1}\sum\nolimits_{y\in \partial \mathcal{B}(d^{-1}n)}\mathcal{H}_y(\mathcal{C}^{\mathrm{in}}(A,n) \cup D).
 \end{equation}
 Similar to $V(\cdot )$ in (\ref{finaluse_58}), we denote  
 \begin{equation}
 	V^{\mathrm{in}}(A,n):= \big\{x\in B(2n)\setminus B(\tfrac{1}{2}n): x\xleftrightarrow{\ge 0} A  \big\}, 
 \end{equation}
 \begin{equation}
	\widehat{V}^{\mathrm{in}}_{\lambda}(A,n):= \big\{x\in B(2n)\setminus B(\tfrac{1}{2}n): x\xleftrightarrow{\widetilde{E}^{\ge 0}\cap \widetilde{B}(d^{-2}\lambda n)} A  \big\}, \ \ \forall \lambda\ge d^4.
\end{equation}
 Note that $\widehat{V}^{\mathrm{in}}_{\lambda}(A,n)$ is measurable with respect to $\mathcal{F}_{\mathcal{C}^{\mathrm{in}}(A,\lambda n)}$.

 \begin{lemma}\label{lemma_prepare_volume}
 For any $3\le d\le 5$, there exist $\Cl\label{const_lemma_prepare_volume1},\Cl\label{const_lemma_prepare_volume2}, \cl\label{const_lemma_prepare_volume4},\cl\label{const_lemma_prepare_volume3},\cl\label{const_lemma_prepare_volume5}>0$ such that for any $n\ge 1$, $A\subset \widetilde{B}(\Cref{const_lemma_prepare_volume1}^{-1}n)$ and $D\subset \widetilde{B}(\Cref{const_lemma_prepare_volume1}^{-1}n)\cup [\widetilde{B}(\Cref{const_lemma_prepare_volume1} n)]^c$, 
 	\begin{equation}\label{finaluse518}
 	\begin{split}
 		& \mathbb{P}^D\big(  | \widehat{V}^{\mathrm{in}}_{\Cref{const_lemma_prepare_volume2}}(A,n)  | \ge \cref{const_lemma_prepare_volume4}n^{\frac{d}{2}+1},  \mathcal{H}_D^{\mathrm{in}}(A, \Cref{const_lemma_prepare_volume2} n)\ge \cref{const_lemma_prepare_volume3}n^{-\frac{d}{2}+1} , \{A\xleftrightarrow{\le 0} \partial B(d^{-2}n)\}^c \big)\\
 		  \gtrsim &  \mathbb{P}^D\big(A \xleftrightarrow{\ge 0} \partial B(n) \big). 
 	\end{split}
 	\end{equation}
 \end{lemma}
 \begin{proof}
 	 According to \cite[Lemma 6.1]{inpreparation}, it is known that when $\Cref{const_lemma_prepare_volume1}$ is sufficiently large, there exist $c_*>0$ such that 
 	  \begin{equation} \label{finaluse_5.14}
 \begin{split}
 	 	\mathbb{P}^{D}\big(\mathsf{F}^{\mathrm{in}} \big):= & \mathbb{P}^{D}\big(\mathcal{H}^{\mathrm{in}}_D(A,n)\ge c_* n^{-\frac{d}{2}+1}, \{A\xleftrightarrow{\le 0} \partial B(d^{-2}n)\}^c  \big) \\
 	 	\gtrsim   &  \mathbb{P}^{D}\big( A\xleftrightarrow{\ge 0} \partial B(n)  \big). 
 \end{split}
 \end{equation}
 Note that $\mathsf{F}^{\mathrm{in}}$ is measurable with respect to $\mathcal{F}_{ \mathcal{C}^{\mathrm{in}}(A,n)}$.

Similar to (\ref{newQb}), for any $\lambda \ge  1$ and $b>0$, we consider the event 
 \begin{equation*}\label{finaluse_5.15}
 	 \mathsf{Q}^{\mathrm{in}}_{b}(\lambda):= \big\{ \mathbb{E}^{D}\big[  	 |V^{\mathrm{in}}(A,\lambda n)  |^2 \mid \mathcal{F}_{ \mathcal{C}^{\mathrm{in}}(A,n)} \big] \le b(\lambda n)^{d+2}  \big\}.
 \end{equation*}
 As an extension of (\ref{upper_Qb}), we claim that 
 \begin{equation}\label{finaluse_claim521}
 	\mathbb{P}^D\big( [ \mathsf{Q}^{\mathrm{in}}_{b}(\lambda)]^c \big) \lesssim b^{-1} \mathbb{P}^D\big(A\xleftrightarrow{\ge 0}  \partial B(n) \big). 
 \end{equation}
 To see this, by applying the AM-QM inequality and using (\ref{final_addnew2.33}), we have 
 \begin{equation}\label{finaluse522}
 	\begin{split}
 		&\mathbb{E}^{D}\big[  	| V^{\mathrm{in}}(A,\lambda n) |^2  \big] \\
 		\overset{(\text{AM-QM})}{\lesssim }  & \lambda n \sum\nolimits_{\frac{1}{2}\lambda n\le n'\le 2\lambda n}  \mathbb{E}^D\big[  \big(\sum\nolimits_{x\in \partial B(n')}\mathbbm{1}_{x\xleftrightarrow{\ge 0}A} \big)^2 \big]\\
 		=& \lambda n \sum\nolimits_{\frac{1}{2}\lambda n\le n'\le 2\lambda n}  \sum\nolimits_{x,y\in \partial B(n')}\mathbb{P}^D\big(  A\xleftrightarrow{\ge 0} x,y \big) \\
 		\overset{(\ref*{final_addnew2.33})}{\lesssim} & \lambda n \sum\nolimits_{\frac{1}{2}\lambda n\le n'\le 2\lambda n}  \sum\nolimits_{x \in \partial B(n')}\mathbb{P}^D\big(  A\xleftrightarrow{\ge 0} x \big) \sum\nolimits_{y\in \partial B(n')}|x-y|^{-\frac{d}{2}+1}\\
 		\lesssim &  (\lambda n)^{\frac{d}{2}+1}  \sum\nolimits_{x \in B(2\lambda n)\setminus B(\frac{1}{2}\lambda n)} \mathbb{P}^D\big(  A\xleftrightarrow{\ge 0} x \big). 
 	\end{split}
 \end{equation}
 In addition, it follows from \cite[Proposition 1.9]{inpreparation} and Lemma \ref{lemma_boxtobox} that  
 \begin{equation}\label{finaluse523}
 	\mathbb{P}^D\big(  A\xleftrightarrow{\ge 0} x \big) \asymp (\lambda n)^{-\frac{d}{2}+1}\mathbb{P}^D\big(  A\xleftrightarrow{\ge 0} \partial B(n) \big), \ \ \forall  x \in B(2\lambda n)\setminus B(\tfrac{1}{2}\lambda n). 
 \end{equation}
 Plugging (\ref{finaluse523}) into (\ref{finaluse522}), we obtain 
 \begin{equation}
 	\mathbb{E}^{D}\big[  	| V^{\mathrm{in}}(A,\lambda n) |^2  \big] \lesssim  (\lambda n)^{d+2} \mathbb{P}^D\big(  A\xleftrightarrow{\ge 0} \partial B(n) \big). 
 \end{equation}
 As in (\ref{finaluse_upper_Qb}) and (\ref{upper_Qb}), this bound further implies the claim (\ref{finaluse_claim521}).

 Note that $V^{\mathrm{in}}(A,n) \subset 
   \widehat{V}^{\mathrm{in}}_{\lambda}(A,n) \cup   \widecheck{V}^{\mathrm{in}}_{\lambda}(A,n)$, where $\widecheck{V}^{\mathrm{in}}_{\lambda}(A,n)$ is defined as 
 \begin{equation}\label{newfinal_524}
 	\widecheck{V}^{\mathrm{in}}_{\lambda}(A,n):= \big\{x\in B(2n)\setminus B(\tfrac{1}{2}n): x\xleftrightarrow{\ge 0} A , \partial B(d^{-2}\lambda n) \big\}. 
 \end{equation}
For any $\lambda\ge d^4$ and $a_1,a_2>0$, we consider the event 
 \begin{equation}\label{def_widehat_V_a}
 	\widehat{\mathsf{V}}^{\mathrm{in}}_{a_1,a_2}(\lambda ):= \big\{ \mathbb{P}^D\big( |\widecheck{V}^{\mathrm{in}}_{\lambda}(A,n)| \ge a_1n^{\frac{d}{2}+1}  \mid \mathcal{F}_{ \mathcal{C}^{\mathrm{in}}(A,n)}\big) \ge a_2 \big\}. 
 \end{equation}
For the probability of this event, we claim that 
 \begin{equation}\label{finaluse_claim525}
 	\mathbb{P}^{D}\big( \widehat{\mathsf{V}}^{\mathrm{in}}_{a_1,a_2}(\lambda ) \big) \lesssim (a_1a_2)^{-1}\lambda^{-\frac{d}{2}+1} \mathbb{P}^D\big(A\xleftrightarrow{\ge 0}  \partial B(n) \big).  
 \end{equation}
  To achieve this, for any $x\in B(2n)\setminus B(\tfrac{1}{2}n)$, by $\big\{x\xleftrightarrow{\ge 0} A , \partial B(d^{-2}\lambda n)\big\}\subset \big\{A\xleftrightarrow{\ge 0} x , \partial B(d^{-2}\lambda n)\big\}$ and Lemma \ref{lemma_Ato_x_boundary}, we have  
 \begin{equation}
 \begin{split}
 	 	  \mathbb{P}^{D}\big(x\xleftrightarrow{\ge 0} A , \partial B(d^{-2}\lambda n) \big)  
 	\overset{(\ref*{iso})}{\lesssim   }   &	\mathbb{P}\big(A\xleftrightarrow{(D)} x , \partial B(d^{-2}\lambda n)\big)\\
 	\overset{\text{Lemma}\ \ref*{lemma_Ato_x_boundary},(\ref*{iso})}{\lesssim }   & (\lambda n)^{-\frac{d}{2}+1} \mathbb{P}^{D}\big(A \xleftrightarrow{\ge 0}  \partial B(   n) \big). 
 \end{split}
 \end{equation}
 Consequently, by summing over all $x\in B(2n)\setminus B(\tfrac{1}{2}n)$, one has 
 \begin{equation}\label{newfinal_528}
 	\mathbb{E}^{D}\big[ |\widecheck{V}^{\mathrm{in}}_{\lambda}(A,n)| \big] \lesssim \lambda^{-\frac{d}{2}+1} n^{\frac{d}{2}+1} \mathbb{P}^{D}\big(A \xleftrightarrow{\ge 0}  \partial B(   n) \big). 
 \end{equation}
 Meanwhile, it follows from the definition of $\widehat{\mathsf{V}}^{\mathrm{in}}_{a_1,a_2}(\lambda )$ in (\ref{def_widehat_V_a}) that 
 \begin{equation}
 	\begin{split}
 		\mathbb{E}^{D}\big[ | \widecheck{V}^{\mathrm{in}}_{\lambda}(A,n)| \big] \ge &\mathbb{E}^{D}\Big[ \mathbbm{1}_{\widehat{\mathsf{V}}^{\mathrm{in}}_{a_1,a_2}(\lambda )} \mathbb{E}\big[ |\widecheck{V}^{\mathrm{in}}_{\lambda}(A,n)| \mid  \mathcal{F}_{ \mathcal{C}^{\mathrm{in}}(A,n)} \big] \Big] \\
 		\gtrsim  & 	\mathbb{P}^{D}\big( \widehat{\mathsf{V}}^{\mathrm{in}}_{a_1,a_2}(\lambda ) \big)  \cdot a_1a_2n^{\frac{d}{2}+1}. 
 	\end{split}
 \end{equation}
 Combined with (\ref{newfinal_528}), it implies the claim (\ref{finaluse_claim525}).

 By taking a sufficiently large $C_\ddagger>0$ and using (\ref{finaluse_5.14}) and (\ref{finaluse_claim521}), one has  
 \begin{equation} \label{finaluse_530}
 \begin{split}
 	 \mathbb{P}^{D}\big( \mathsf{F}^{\mathrm{in}} \cap  \mathsf{Q}^{\mathrm{in}}_{C_\ddagger}(1)   \big)    
 			  \ge     c_{\vartriangle}  \mathbb{P}^D\big(A\xleftrightarrow{\ge 0}  \partial B(n) \big), 
 \end{split}
 \end{equation}
 for some constant $c_{\vartriangle}>0$. It follows from Lemma \ref{newlemma3.1} that on the event $\mathsf{F}^{\mathrm{in}}$,
\begin{equation}
 	\mathbb{E}^{D}\big[ | V^{\mathrm{in}}(A, n) | \mid \mathcal{F}_{ \mathcal{C}^{\mathrm{in}}(A,n)} \big] \gtrsim   n^{\frac{d}{2}+1}.
 \end{equation}
Therefore, on the event $\mathsf{F}^{\mathrm{in}} \cap  \mathsf{Q}^{\mathrm{in}}_{C_\ddagger}(1) $, by the Paley-Zygmund inequality, we know that there exist constants $c_\dagger, c_{\star}>0$ such that 
 \begin{equation}\label{finaluse5.32}
 	\begin{split}
 		\mathbb{P}^D\big( |V^{\mathrm{in}}(A,  n) | \ge c_\dagger n^{\frac{d}{2}+1}   \mid \mathcal{F}_{ \mathcal{C}^{\mathrm{in}}(A,n)} \big) \ge  c_{\star}. 
 	\end{split}
 \end{equation}
  Meanwhile, by (\ref{finaluse_claim525}), we can choose a sufficiently large $\Cref{const_lemma_prepare_volume2}>0$ such that 
  \begin{equation}
  	\mathbb{P}^{D}\big( \widehat{\mathsf{V}}^{\mathrm{in}}_{\frac{1}{2}c_\dagger,\frac{1}{2}c_{\star} }(\Cref{const_lemma_prepare_volume2} ) \big) \le \tfrac{1}{2}c_{\vartriangle} \mathbb{P}^D\big(A\xleftrightarrow{\ge 0}  \partial B(n) \big).  
  \end{equation} 
Combined with (\ref{finaluse_530}), it yields that 
\begin{equation}\label{finaluse_534}
	 \mathbb{P}^{D}\big( \mathsf{F}^{\mathrm{in}} \cap \mathsf{Q}^{\mathrm{in}}_{C_\ddagger}(1) \cap \big[\widehat{\mathsf{V}}^{\mathrm{in}}_{\frac{1}{2}c_\dagger,\frac{1}{2}c_{\star} }(\Cref{const_lemma_prepare_volume2})\big]^c  \big)    
 			  \ge   \tfrac{1}{2}  c_{\vartriangle}  \mathbb{P}^D\big(A\xleftrightarrow{\ge 0}  \partial B(n) \big).  
\end{equation}
Furthermore, when the event on the left-hand side of (\ref{finaluse_534}) occurs, by the inclusion $V^{\mathrm{in}}(A,n)\subset 
  \widehat{V}^{\mathrm{in}}_{\Cref{const_lemma_prepare_volume2}}(A,n)  \cup \widecheck{V}^{\mathrm{in}}_{\Cref{const_lemma_prepare_volume2}}(A,n)$, we have  
 \begin{equation}\label{finaluse_5.35}
 	\begin{split}
 		& \mathbb{P}^D\big( |\widehat{V}^{\mathrm{in}}_{\Cref{const_lemma_prepare_volume2}}(A,n) | \ge \tfrac{1}{2}c_\dagger n^{\frac{d}{2}+1}   \mid \mathcal{F}_{ \mathcal{C}^{\mathrm{in}}(A,n)} \big) \\
 		\ge & \mathbb{P}^D\big( | V^{\mathrm{in}}(A,n) | \ge  c_\dagger n^{\frac{d}{2}+1}   \mid \mathcal{F}_{ \mathcal{C}^{\mathrm{in}}(A,n)} \big) \\
 		&-\mathbb{P}^D\big( |\widecheck{V}^{\mathrm{in}}_{\Cref{const_lemma_prepare_volume2}}(A,n) | \ge \tfrac{1}{2}c_\dagger n^{\frac{d}{2}+1}   \mid \mathcal{F}_{ \mathcal{C}^{\mathrm{in}}(A,n)} \big)\\
 		\overset{ }{\ge } &  c_{\star} - \tfrac{1}{2}c_{\star}= \tfrac{1}{2}c_{\star},
 	\end{split}
 \end{equation}
  where in the last inequality we used (\ref{finaluse5.32}) and the definition of $\widehat{\mathsf{V}}^{\mathrm{in}}_{a_1,a_2}(\lambda)$ in (\ref{def_widehat_V_a}).

For the same reason as in proving (\ref{finaluse_530}), by taking a large $C_{\ddagger}'>0$, we have 
 \begin{equation}\label{finaluse5.36}
 \begin{split}
 		\mathbb{P}^{D}\big( \mathsf{K}^{\mathrm{in}} \big):= &\mathbb{P}^{D}\big( \mathsf{F}^{\mathrm{in}} \cap  \mathsf{Q}^{\mathrm{in}}_{C_\ddagger}(1) \cap  \big[\widehat{\mathsf{V}}^{\mathrm{in}}_{\frac{1}{2}c_\dagger,\frac{1}{2}c_{\star} }(\Cref{const_lemma_prepare_volume2})\big]^c  \cap  \mathsf{Q}^{\mathrm{in}}_{C_\ddagger'}(\Cref{const_lemma_prepare_volume2})  \big)   	\\
 			  \gtrsim   & \mathbb{P}^D\big(A\xleftrightarrow{\ge 0}  \partial B(n) \big). 
 \end{split}
 \end{equation}
In addition, similar to (\ref{finaluse5.32}), there exist $c_{\clubsuit},c_{\spadesuit}>0$ such that on $\mathsf{F}^{\mathrm{in}} \cap \mathsf{Q}^{\mathrm{in}}_{C_\ddagger'}(\Cref{const_lemma_prepare_volume2}) $, 
 \begin{equation}\label{finaluse_537}
 	\begin{split}
 		\mathbb{P}^D\big( |V^{\mathrm{in}}(A, \Cref{const_lemma_prepare_volume2} n) | \ge c_{\clubsuit} n^{\frac{d}{2}+1}   \mid \mathcal{F}_{ \mathcal{C}^{\mathrm{in}}(A,n)} \big)\ge c_{\spadesuit}. 
 	\end{split}
 \end{equation}
 Note that $\mathcal{F}_{ \mathcal{C}^{\mathrm{in}}(A, n)}\subset \mathcal{F}_{ \mathcal{C}^{\mathrm{in}}(A,\Cref{const_lemma_prepare_volume2} n)}$. We consider the event 
 \begin{equation}
 	\mathsf{J}^{\mathrm{in}}:= \big\{ \mathbb{P}^D\big( |V^{\mathrm{in}}(A, \Cref{const_lemma_prepare_volume2} n) | \ge c_{\clubsuit} n^{\frac{d}{2}+1}   \mid \mathcal{F}_{ \mathcal{C}^{\mathrm{in}}(A,\Cref{const_lemma_prepare_volume2}n)} \big) \ge \tfrac{1}{2} c_{\spadesuit} \big\}. 
 \end{equation} 
 In fact, (\ref{finaluse_537}) implies that on $\mathsf{F}^{\mathrm{in}} \cap [\mathsf{Q}^{\mathrm{in}}_{C_\ddagger'}(\Cref{const_lemma_prepare_volume2})]^c $, 
 \begin{equation}\label{finaluse5.39}
 	\mathbb{P}^D\big(	\mathsf{J}^{\mathrm{in}} \mid \mathcal{F}_{ \mathcal{C}^{\mathrm{in}}(A, n)} \big)\ge \tfrac{1}{2}c_{\spadesuit},   
 \end{equation}
where the ``implication'' can be derived from the following calculation:
 \begin{equation}
 	\begin{split}
 		& \mathbb{P}^D\big( |V^{\mathrm{in}}(A, \Cref{const_lemma_prepare_volume2} n) |\ge c_{\clubsuit} n^{\frac{d}{2}+1}   \mid \mathcal{F}_{ \mathcal{C}^{\mathrm{in}}(A,n)} \big)\\
 		\le & \mathbb{P}^D\big( |V^{\mathrm{in}}(A, \Cref{const_lemma_prepare_volume2} n) | \ge c_{\clubsuit} n^{\frac{d}{2}+1}, 	\mathsf{J}^{\mathrm{in}}   \mid \mathcal{F}_{ \mathcal{C}^{\mathrm{in}}(A,n)} \big)\\
 		& +\mathbb{P}^D\big( | V^{\mathrm{in}}(A, \Cref{const_lemma_prepare_volume2} n) | \ge c_{\clubsuit} n^{\frac{d}{2}+1},(\mathsf{J}^{\mathrm{in}} )^c  \mid \mathcal{F}_{ \mathcal{C}^{\mathrm{in}}(A,n)} \big)\\
 		\le & \mathbb{P}^D\big(	\mathsf{J}^{\mathrm{in}} \mid \mathcal{F}_{ \mathcal{C}^{\mathrm{in}}(A, n)} \big) + \tfrac{1}{2} c_{\spadesuit} . 
 	\end{split}
 \end{equation}
 One the event $\mathsf{J}^{\mathrm{in}}$, one has 
 \begin{equation}
 	\mathbb{E}^D\big[ | V^{\mathrm{in}}(A, \Cref{const_lemma_prepare_volume2} n) |   \mid  \mathcal{F}_{ \mathcal{C}^{\mathrm{in}}(A,\Cref{const_lemma_prepare_volume2} n)} \big] \gtrsim n^{\frac{d}{2}+1}, 
 \end{equation}
 which, by the fact that $\mathbb{E}^D\big[ |V^{\mathrm{in}}(A, \Cref{const_lemma_prepare_volume2} n) |   \mid  \mathcal{F}_{ \mathcal{C}^{\mathrm{in}}(A,\Cref{const_lemma_prepare_volume2} n)} \big]\asymp n^d\cdot  \mathcal{H}_D^{\mathrm{in}}(A, \Cref{const_lemma_prepare_volume2} n)$ (derived from Lemma \ref{newlemma3.1}), further implies that  
 \begin{equation}
	\mathcal{H}_D^{\mathrm{in}}(A, \Cref{const_lemma_prepare_volume2} n) \ge c_{\heartsuit}n^{-\frac{d}{2}+1} 
\end{equation}
 for some constant $c_{\heartsuit}>0$. Combined with (\ref{finaluse5.39}), it yields that on $\mathsf{F}^{\mathrm{in}} \cap [\mathsf{Q}^{\mathrm{in}}_{C_\ddagger'}(\Cref{const_lemma_prepare_volume2})]^c $, 
 \begin{equation}\label{finaluse5.43}
 		\mathbb{P}^D\big(	\mathcal{H}_D^{\mathrm{in}}(A, \Cref{const_lemma_prepare_volume2} n) \ge c_{\heartsuit}n^{-\frac{d}{2}+1}  \mid \mathcal{F}_{ \mathcal{C}^{\mathrm{in}}(A, n)} \big)\ge \tfrac{1}{2}c_{\spadesuit}.  
 \end{equation}

 To sum up, we know that on the event $ \mathsf{K}^{\mathrm{in}}$, 
\begin{equation*}
	\begin{split}
		&\mathbb{P}\big( | \widehat{V}^{\mathrm{in}}_{\Cref{const_lemma_prepare_volume2}}(A,n)| \ge \tfrac{1}{2}c_\dagger  n^{\frac{d}{2}+1},  \mathcal{H}_D^{\mathrm{in}}(A, \Cref{const_lemma_prepare_volume2} n)\ge c_{\heartsuit} n^{-\frac{d}{2}+1}   \mid \mathcal{F}_{ \mathcal{C}^{\mathrm{in}}(A, n)} \big) \\
		\overset{\text{(FKG)}}{\ge } & \mathbb{P}\big( |\widehat{V}^{\mathrm{in}}_{\Cref{const_lemma_prepare_volume2}}(A,n)| \ge \tfrac{1}{2}c_\dagger  n^{\frac{d}{2}+1}  \mid \mathcal{F}_{ \mathcal{C}^{\mathrm{in}}(A, n)} \big)\mathbb{P}\big(  \mathcal{H}_D^{\mathrm{in}}(A, \Cref{const_lemma_prepare_volume2} n)\ge c_{\heartsuit} n^{-\frac{d}{2}+1}   \mid \mathcal{F}_{ \mathcal{C}^{\mathrm{in}}(A, n)} \big)\\
		\overset{(\ref*{finaluse_5.35}), (\ref*{finaluse5.43})}{\gtrsim } & \tfrac{1}{4}c_{\star} c_{\spadesuit}.  
	\end{split}
\end{equation*} 
By taking the integrals on both sides on $\mathsf{K}^{\mathrm{in}}$ (note that $\mathsf{K}^{\mathrm{in}}\subset \{A\xleftrightarrow{\le 0} \partial B(d^{-2}n)\}^c$) and using (\ref{finaluse5.36}), we confirm the desired bound (\ref{finaluse518}).  
\end{proof}

 Now we are ready to establish Lemma \ref{finaluse_lemma_volume}.

 \begin{proof}[Proof of Lemma \ref{finaluse_lemma_volume}]
For (\ref{ineq_finaluse_lemma_volume}), by the isomorphism theorem, it suffices to show  \begin{equation}\label{final544}
	\mathbb{P}^{D}\big(A_1 \xleftrightarrow{\ge 0} A_2, |V^{\ge 0}(A_1,n)|\ge \cref{const_finaluse_lemma_volume2} n^{\frac{d}{2}+1} \big) \ge \cref{const_finaluse_lemma_volume3} \mathbb{P}^D\big(A_1 \xleftrightarrow{\ge 0} A_2  \big).
\end{equation}
Here $V^{\ge 0}(A_1,n):=  \big\{x\in B(2n)\setminus B(\tfrac{1}{2}n): x\xleftrightarrow{\ge 0} A_1 \big\}$.

 We denote by $\mathsf{F}_1$ the event on the left-hand side of (\ref{finaluse518}) with $A=A_1$, and thus by Lemma \ref{lemma_prepare_volume} we have
 \begin{equation}\label{finaluse_545}
 	\mathbb{P}^{D}(\mathsf{F}_1) \gtrsim \mathbb{P}^D\big(A \xleftrightarrow{\ge 0} \partial B(n) \big). 
 \end{equation}
Meanwhile, similar to (\ref{finaluse_514}) and (\ref{finaluse_515}), we define 
 	\begin{equation*}
  	\mathcal{C}^{\mathrm{out}}(A_2,m):=\big\{v\in [\widetilde{B}(10m)]^c: v\xleftrightarrow{\widetilde{E}^{\ge 0}\cap [\widetilde{B}(10m)]^c} A_2\ \text{or}\ v\xleftrightarrow{\widetilde{E}^{\le 0}\cap [\widetilde{B}(10m)]^c} A_2 \big\}, 
  	  \end{equation*}  
 \begin{equation*}
 		\mathcal{H}^{\mathrm{out}}_D(A_2,m):=|\partial \mathcal{B}(5m)|^{-1}\sum\nolimits_{y\in \partial \mathcal{B}(5m)}\mathcal{H}_y(\mathcal{C}^{\mathrm{out}}(A_2,m) \cup D).  
\end{equation*}
Referring to \cite[Lemma 6.1]{inpreparation}, one has: for some constant $c_*>0$,  
\begin{equation}\label{finaluse_546}
\begin{split}
	\mathbb{P}^{D}\big(\mathsf{F}_2 \big):= & \mathbb{P}^{D}\big(\mathcal{H}^{\mathrm{out}}_D(A_2,\Cref{const_lemma_prepare_volume2} n)\ge c_* n^{-\frac{d}{2}+1}, \{A_2\xleftrightarrow{\le 0} \partial B(10\Cref{const_lemma_prepare_volume2} n)\}^c  \big) \\
 	 	\gtrsim   & \mathbb{P}^{D}\big( A_2\xleftrightarrow{\ge 0} \partial B( n)  \big). 
\end{split}
\end{equation}
 Note that $\mathsf{F}_1$ and $\mathsf{F}_2$ are both increasing events and are measurable with respect to $\mathcal{F}_{\mathcal{C}^{\mathrm{in}}(A_1 , \Cref{const_lemma_prepare_volume2}n)}$ and $\mathcal{F}_{\mathcal{C}^{\mathrm{out}}(A_2, \Cref{const_lemma_prepare_volume2}n)}$ respectively. By \cite[Corollary 2.7]{inpreparation}, we know that on $\mathsf{F}_1\cap \mathsf{F}_2$ (which implies $\{\mathcal{H}_D^{\mathrm{in}}(A, \Cref{const_lemma_prepare_volume2} n)\ge \cref{const_lemma_prepare_volume3}n^{-\frac{d}{2}+1}, \mathcal{H}^{\mathrm{out}}_D(A_2,\Cref{const_lemma_prepare_volume2}n)\ge c_* n^{-\frac{d}{2}+1} \}$), 
 \begin{equation}\label{finaluse_5.47}
 	\begin{split}
 		& \mathbb{P}^D\big( A_1\xleftrightarrow{\ge 0} A_2 \mid \mathcal{F}_{\mathcal{C}^{\mathrm{in}}(A_1 , \Cref{const_lemma_prepare_volume2}n)\cup \mathcal{C}^{\mathrm{out}}(A_2, \Cref{const_lemma_prepare_volume2} n)}   \big)\\
 		=& \mathbb{P}^D\big( \mathcal{C}^{\mathrm{in}}(A_1 , \Cref{const_lemma_prepare_volume2}n) \xleftrightarrow{\ge 0} \mathcal{C}^{\mathrm{out}}(A_2,  \Cref{const_lemma_prepare_volume2} n) \mid \mathcal{F}_{\mathcal{C}^{\mathrm{in}}(A_1 , \Cref{const_lemma_prepare_volume2}n)\cup \mathcal{C}^{\mathrm{out}}(A_2, \Cref{const_lemma_prepare_volume2}n)}   \big) \\
 		\gtrsim & \big[   n^{d-2}\cdot \mathcal{H}_D^{\mathrm{in}}(A_1, \Cref{const_lemma_prepare_volume2} n) \cdot \mathcal{H}^{\mathrm{out}}_D(A_2,\Cref{const_lemma_prepare_volume2}n) \big] \land 1 \gtrsim 1. 
 	\end{split}
 \end{equation}
Thus, since $\mathsf{F}_1\subset \big\{| \widehat{V}^{\mathrm{in}}_{\Cref{const_lemma_prepare_volume2}}(A_1,n)|  \ge \cref{const_lemma_prepare_volume4}n^{\frac{d}{2}+1} \big\} \subset \big\{ | V^{\ge 0}(A_1,n)|  \ge \cref{const_lemma_prepare_volume4}n^{\frac{d}{2}+1} \big\} $, we have 
\begin{equation}
	\begin{split}
		& \mathbb{P}^{D}\big(A_1 \xleftrightarrow{\ge 0} A_2, |V^{\ge 0}(A_1,n)|\ge \cref{const_lemma_prepare_volume4} n^{\frac{d}{2}+1} \big)  \\
		\ge & \mathbb{E}^{D} \big[\mathbbm{1}_{ \mathsf{F}_1\cap  \mathsf{F}_2}\cdot  \mathbb{P}^{D} \big(A_1 \xleftrightarrow{\ge 0} A_2\mid \mathcal{F}_{\mathcal{C}^{\mathrm{in}}(A_1 , \Cref{const_lemma_prepare_volume2}n)\cup \mathcal{C}^{\mathrm{out}}(A_2, \Cref{const_lemma_prepare_volume2}n)}  \big) \big]  \\
		\overset{ (\ref*{finaluse_5.47})}{\gtrsim } & \mathbb{P}^{D}\big(\mathsf{F}_1 \cap \mathsf{F}_2 \big)\overset{ \text{(FKG)}}{\gtrsim }   \mathbb{P}^{D}\big(\mathsf{F}_1   \big)    \mathbb{P}^{D}\big(  \mathsf{F}_2 \big)  \\
		\overset{ (\ref*{finaluse_545}),(\ref*{finaluse_546})}{\gtrsim }   &  \mathbb{P}^D\big(A_1 \xleftrightarrow{\ge 0} \partial B(n) \big)\mathbb{P}^D\big(A_2\xleftrightarrow{\ge 0} \partial B(n) \big) \overset{(\ref*{QM_ineq_1})}{\asymp}\mathbb{P}^D\big(A_1 \xleftrightarrow{\ge 0} A_2\big). 
			\end{split}
\end{equation}
 Now we obtain (\ref{final544}) with $\cref{const_finaluse_lemma_volume2}=\cref{const_lemma_prepare_volume4}$, thereby concluding the proof of this lemma. 
 \end{proof}

 \textbf{When $d\ge 7$.} Compared to the low-dimensional cases $3\le d\le 5$, the proof in this case is more straightforward because for the critical cluster in high dimensions, the typical order of its volume matches that of its capacity (which is not the case in low dimensions), and in addition, the capacity can be conveniently analyzed through the exploration martingale argument (see \cite{lupu2018random,ding2020percolation}). Specifically, the proof of (\ref{newineq4.2}) for $d\ge 7$ consists of the following two steps: 
 \begin{enumerate}
 	\item  Arbitrarily take $n\ll N$ and explore the cluster $\mathcal{C}^{\ge 0}(\bm{0})$, stopping upon hitting $\partial B(n)$. We then show that conditioned on $\{\bm{0}\xleftrightarrow{\ge 0} \partial B(N)\}$, with high probability the harmonic average on $\partial B(Cn)$ after sampling this sub-cluster is of order $n^{4-d}$.

 	\item  Let $n_1\ll n_2\ll N$. By the aforementioned arguments, we know that with high probability the harmonic averages considered in Step (1) for $n=n_1$ and $n=n_2$ differ significantly. The exploration martingale argument then shows that this difference typically leads to a substantial growth in capacity during the exploration from $\partial B(n_1)$ to $\partial B(n_2)$. Applying the subadditivity of capacity (see (\ref{2.8})), this growth in capacity provides a lower bound on the volume increment.


 	\textbf{P.S.:} These two steps also apply to the low-dimensional cases $3\le d\le 5$, showing that it is rare for a large critical cluster to have a small capacity.

 \end{enumerate}

 To implement Step (1), we present the following lemma. Recall $\mathcal{C}_n^{\ge 0}$ and $\mathcal{H}_n^{\ge 0}$ in (\ref{def_3.1}) and (\ref{finaluse_def_3.1}) respectively. Also recall that $\widehat{\mathbb{P}}_{d,N}(\cdot)=\widehat{\mathbb{P}}_{d,N}\big(\cdot \mid \bm{0}\xleftrightarrow{\ge 0} \partial B(N)\big)$.

\begin{lemma}\label{lemma_prepare_highd}
	For any $d\ge 7$ and $\lambda>2$, there exist $\Cl\label{const_lemma_prepare_highd1}(d,\lambda),\Cl\label{const_lemma_prepare_highd2}(d) >0$ such that for any $n\ge \Cref{const_lemma_prepare_highd1}$ and $N\ge 10dn$, 
		\begin{equation}
		\widehat{\mathbb{P}}_{d,N}\big(  \mathcal{H}_n^{\ge 0} \notin [\lambda^{-1} n^{4-d} ,\lambda n^{4-d} ]     \big) \le \Cref{const_lemma_prepare_highd2} \lambda^{-1}. 
	\end{equation}
\end{lemma}
\begin{proof}
We first estimate $\widehat{\mathbb{P}}_{d,N}\big(   \mathcal{H}_n^{\ge 0}  \le \lambda^{-1} n^{4-d}   \big)$. Note that $$\{\bm{0}\xleftrightarrow{\ge 0} \partial B(N)\}\subset \{\bm{0}\xleftrightarrow{\ge 0} \partial B(n)\} = \{\mathcal{H}_n^{\ge 0}>0 \}.$$ Therefore, by applying \cite[Lemma 2.8]{inpreparation}, we have 
 \begin{equation}\label{finaluse550}
 	\begin{split}
 		& \mathbb{P}\big(  \mathcal{H}_n^{\ge 0}  \le \lambda^{-1} n^{4-d}  , \bm{0}\xleftrightarrow{\ge 0} \partial B(N) \big) \\
 		=& \mathbb{E}\big[ \mathbbm{1}_{0< \mathcal{H}_n^{\ge 0}  \le \lambda^{-1} n^{4-d}}  \mathbb{P}\big(\mathcal{C}_n^{\ge 0} \xleftrightarrow{\ge 0}  \partial B(N) \mid \mathcal{F}_{\mathcal{C}_n^{\ge 0}}\big) \big]\\
 		\lesssim &\mathbb{P}\big( \mathcal{H}_n^{\ge 0} >0\big) \cdot \big[  n^{d-2}\cdot  \lambda^{-1} n^{4-d} \cdot \theta_d(N) \big]  
 		\overset{(\ref*{one_arm_high})}{\lesssim } \lambda^{-1}\theta_d(N) . 
 	\end{split}
 \end{equation}

Now we estimate $ \widehat{\mathbb{P}}_{d,N}\big(   \mathcal{H}_n^{\ge 0}  \ge \lambda n^{4-d}   \big)$. Arbitrarily take $y\in \partial B(d^2n)$. It directly follows from Theorem \ref{thm_1.4} that 
\begin{equation}\label{finaluse551}
	\mathbb{P}\big( \bm{0} \xleftrightarrow{\ge 0 } y, \partial B(N)   \big) \lesssim n^{4-d} \theta_d(N). 
\end{equation}
Meanwhile, on the event $\{ \mathcal{H}_n^{\ge 0}\ge \lambda n^{4-d} \}$, we have 
\begin{equation}\label{finaluse552}
	\begin{split}
		& \mathbb{P}\big( \bm{0} \xleftrightarrow{\ge 0 } y,\partial B(N)  \mid \mathcal{F}_{\mathcal{C}_n^{\ge 0}}   \big)  \\
		=& \mathbb{P}\big( \mathcal{C}_n^{\ge 0} \xleftrightarrow{\ge 0 } y,\partial B(N)  \mid \mathcal{F}_{\mathcal{C}_n^{\ge 0}}   \big)   \\
		\overset{\text{(FKG)}}{\ge }  & \mathbb{P}\big( \mathcal{C}_n^{\ge 0} \xleftrightarrow{\ge 0 } y  \mid \mathcal{F}_{\mathcal{C}_n^{\ge 0}}   \big)  \mathbb{P}\big(  \mathcal{C}_n^{\ge 0}\xleftrightarrow{\ge 0 } \partial B(N)  \mid \mathcal{F}_{\mathcal{C}_n^{\ge 0}}   \big) \\
		\overset{\text{Lemma}\ \ref*{newlemma3.1}}{\gtrsim }  & \lambda n^{4-d}  \mathbb{P}\big(  \mathcal{C}_n^{\ge 0}\xleftrightarrow{\ge 0 } \partial B(N)  \mid \mathcal{F}_{\mathcal{C}_n^{\ge 0}}   \big), 
	\end{split}
\end{equation}
where in the last inequality we need $\lambda n^{4-d}\lesssim 1$ (ensured by $n\ge \Cref{const_lemma_prepare_highd1}$ with a sufficiently large $\Cref{const_lemma_prepare_highd1}(d,\lambda )>0$). By taking the integrals on both sides of (\ref{finaluse552}) on the event $\{ \mathcal{H}_n^{\ge 0}\ge \lambda n^{4-d} \}$, we get 
\begin{equation}
\begin{split}
	\mathbb{P}\big(\bm{0} \xleftrightarrow{\ge 0 } y,\partial B(N)   \big) \gtrsim \lambda n^{4-d} \mathbb{P}\big(\mathcal{H}_n^{\ge 0}\ge \lambda n^{4-d} ,  \bm{0} \xleftrightarrow{\ge 0 } \partial B(N)   \big).
\end{split}
	\end{equation}
Combined with (\ref{finaluse551}), it yields that 
\begin{equation}\label{finaluse554}
	\mathbb{P}\big(\mathcal{H}_n^{\ge 0}\ge \lambda n^{4-d} ,  \bm{0} \xleftrightarrow{\ge 0 } \partial B(N)   \big) \lesssim \lambda^{-1}\theta_d(N). 
\end{equation}

By (\ref{finaluse550}) and (\ref{finaluse554}), we complete the proof of this lemma.
\end{proof}

For any $\epsilon>0$, we denote $C_*(d,\epsilon):=8\Cref{const_lemma_prepare_highd2} \epsilon^{-1}$. For any $M\ge 1$, let $\hat{M}:=\frac{M}{C_\dagger C_*}$, where $C_\dagger>0$ is a large constant to be determined later. For any $n\ge 1$, we define the event $\mathsf{H}^*_n:=\{C_*^{-1}n^{4-d} \le \mathcal{H}_n^{\ge 0}\le C_*  n^{4-d}\}$. By Lemma \ref{lemma_prepare_highd} we have 
\begin{equation}
	\widehat{\mathbb{P}}_{d,N}\big( (\mathsf{H}^*_M)^c\cup (\mathsf{H}^*_{\hat{M}})^c \big)\le \tfrac{\epsilon }{4}. 
\end{equation}
Thus, to prove (\ref{newineq4.2}), it remains to show (recalling that $\mathcal{V}_{M}^{\ge 0}=\mathrm{vol}(\mathcal{C}^{\ge 0}(\bm{0})\cap B(M))$)
\begin{equation}\label{finaluse5.56}
	\widehat{\mathbb{P}}_{d,N}\big( \mathcal{V}_{M}^{\ge 0}\le \cref{const_typical_volume_3}M^{4} ,  \mathsf{H}^*_M,  \mathsf{H}^*_{\hat{M}}  \big)\le \tfrac{\epsilon }{4}.  
\end{equation}

Now we employ the exploration martingale argument, as promised in Step (2). For any $t\ge 0$, we denote $\mathbf{B}_t:= \{ v\in \widetilde{B}(M): \inf_{w\in \widetilde{B}(\hat{M})} \|w-v \|\le t\}$ and 
\begin{equation}
	\widetilde{\mathcal{C}}_t:= \{v\in \mathbf{B}_t : \bm{0} \xleftrightarrow{\widetilde{E}^{\ge 0}\cap \mathbf{B}_t}  v \}. 
\end{equation}
Clearly, one has $\widetilde{\mathcal{C}}_0= \mathcal{C}_{\hat{M}}^{\ge 0}$ and $\widetilde{\mathcal{C}}_\infty = \mathcal{C}_{M}^{\ge 0}$. We arbitrarily fix a point $y\in \partial B(20dM)$. Since the $\sigma$-field $\mathcal{F}_{\widetilde{\mathcal{C}}_t}$ is increasing in $t$, we know that $\mathcal{M}_t:= \mathbb{E}\big[ \widetilde{\phi}_y\mid \mathcal{F}_{\widetilde{\mathcal{C}}_t}\big]$ for $t\ge 0$ is a martingale satisfying that $\mathcal{M}_0=\mathcal{H}_y(\mathcal{C}_{\hat{M}}^{\ge 0})$ and $\mathcal{M}_\infty=\mathcal{H}_y(\mathcal{C}_{M}^{\ge 0})$. We also denote the quadratic variation of this martingale by $\langle \mathcal{M} \rangle_t$ for $t\ge 0$. By \cite[Corollary 10]{ding2020percolation}, $\langle \mathcal{M} \rangle_t$ can be equivalently written as 
\begin{equation}
	\begin{split}
		   \widetilde{G}_{\widetilde{\mathcal{C}}_0}(y,y)- \widetilde{G}_{\widetilde{\mathcal{C}}_t}(y,y) 
		=& \sum\nolimits_{v\in \widetilde{\mathcal{C}}_t} \widetilde{\mathbb{P}}_y\big(\tau_{\widetilde{\mathcal{C}}_t} = \tau_v <\tau_{\widetilde{\mathcal{C}}_0}  \big)\widetilde{G}(v,y)\\
		\overset{(\ref*{bound_green} )}{\lesssim  } &   M^{2-d}\widetilde{\mathbb{P}}_y\big( \tau_{\widetilde{\mathcal{C}}_t} <\infty  \big)  \\
		\overset{(\ref*{new2.20})} {\lesssim  } & M^{4-2d} \mathrm{cap}\big(\widetilde{\mathcal{C}}_t \big)   \overset{ \widetilde{\mathcal{C}}_t\subset \mathcal{C}^{\ge 0}(\bm{0})\cap B(M),(\ref*{2.8}) } {\lesssim  }  M^{4-2d} \mathcal{V}_{M}^{\ge 0} .
	\end{split}
\end{equation}
 Therefore, there exists $C_\dagger >0$ such that 
 \begin{equation}\label{final5.59}
 	\big\{ \mathcal{V}_{M}^{\ge 0}\le  \cref{const_typical_volume_3}M^{ 4} \big\} \subset \big\{ \langle \mathcal{M} \rangle_\infty \le C_\dagger\cref{const_typical_volume_3}M^{ 8-2d}  \big\}. 
 	 \end{equation}
Moreover, similar to \cite[Theorem 11]{ding2020percolation} (which is derived from the martingale representation theorem), one has that (letting $T_t:=\inf\{s\ge 0:  \langle \mathcal{M} \rangle_s >t \}$)
\begin{equation}\label{timechange}
	W_t:= \left\{ \begin{aligned}
		&\mathcal{M}_{T_t}-\mathcal{M}_0, \ \ \ \ \   \forall 0\le t<  \langle \mathcal{M} \rangle_\infty  \\ 
		&\mathcal{M}_{\infty}-\mathcal{M}_0,\ \ \ \ \  \forall t\ge   \langle \mathcal{M} \rangle_\infty 
	\end{aligned}\right.
\end{equation}
is a standard Brownian motion stopped at $ \langle \mathcal{M} \rangle_\infty $.

 In what follows, we show that for a sufficiently small $\cref{const_typical_volume_3}(d,\epsilon)>0$, 
 \begin{equation}\label{finaluse5.60}
  \mathbb{P}\big( \langle \mathcal{M} \rangle_\infty \le C_\dagger\cref{const_typical_volume_3}M^{ 4}  ,  \mathsf{H}^*_M, \mathsf{H}^*_{\hat{M}}, \bm{0}\xleftrightarrow{\ge 0} \partial B(N)   \big)\le \tfrac{\epsilon }{4}\cdot \theta_d(N). 
 \end{equation}
 To achieve this, since $\{ \langle \mathcal{M} \rangle_\infty \le C_\dagger\cref{const_typical_volume_3}M^{ 4}\}$, $\mathsf{H}^*_M$ and $\mathsf{H}^*_{\hat{M}}$ are all measurable with respect to $\mathcal{F}_{\mathcal{C}_{M}^{\ge 0}}$, the left-hand side of (\ref{finaluse5.60}) equals to 
 \begin{equation}\label{finaluse562}
 	\begin{split}
 	& 	\mathbb{E}\big[\mathbbm{1}_{\langle \mathcal{M} \rangle_\infty \le C_\dagger\cref{const_typical_volume_3}M^{ 4}  ,  \mathsf{H}^*_M, \mathsf{H}^*_{\hat{M}}}  \mathbb{P}\big( \mathcal{C}_{M}^{\ge 0} \xleftrightarrow{\ge 0} \partial B(N) \mid \mathcal{F}_{\mathcal{C}_{M}^{\ge 0}} \big) \big] \\
 		\overset{}{\lesssim} & \mathbb{P}\big( \langle \mathcal{M} \rangle_\infty \le C_\dagger\cref{const_typical_volume_3}M^{ 4}  ,  \mathsf{H}^*_M, \mathsf{H}^*_{\hat{M}}\big)  \cdot C_*M^{2} \theta_d(N),
 	\end{split}
 \end{equation}
where in the last line we used \cite[Lemma 2.8]{inpreparation}. Next, we estimate the probability $\mathbb{P}\big( \langle \mathcal{M} \rangle_\infty \le C_\dagger\cref{const_typical_volume_3}M^{ 4}  ,  \mathsf{H}^*_M, \mathsf{H}^*_{\hat{M}}\big)$. When $\mathsf{H}^*_M\cap \mathsf{H}^*_{\hat{M}}$ occurs, one has 
\begin{equation}
	\mathcal{M}_0=\mathcal{H}_y(\mathcal{C}_{\hat{M}}^{\ge 0}) \overset{(\ref*{new2.20})}{\asymp} \big( \tfrac{\hat{M}}{M}\big)^{d-2} \mathcal{H}^{\ge 0}_{\hat{M}} \le \big( \tfrac{\hat{M}}{M}\big)^{d-2}\cdot C_*\hat{M}^{4-d}= C_\dagger^{-2} C_*^{-1}M^{4-d}, 
\end{equation}
\begin{equation}
	\mathcal{M}_\infty=\mathcal{H}_y(\mathcal{C}_{M}^{\ge 0})\asymp \mathcal{H}^{\ge 0}_{M} \ge C_*^{-1}M^{4-d}.  
\end{equation}
Therefore, by taking a sufficiently large $C_\dagger>0$, we have 
\begin{equation}\label{finaluse_565}
	\mathcal{M}_\infty - \mathcal{M}_0 \ge c_\ddagger C_*^{-1}M^{4-d}, 
\end{equation} 
for some constant $c_\ddagger>0$. Conditioned on $\mathcal{F}_{\mathcal{C}_{\hat{M}}^{\ge 0}}$ and given the occurrence of $ \mathsf{H}^*_{\hat{M}}$, when $\mathcal{M}_\infty - \mathcal{M}_0 \gtrsim C_*^{-1}M^{4-d}$ and $\langle \mathcal{M} \rangle_\infty \le C_\dagger\cref{const_typical_volume_3}M^{ 4}$ both happen, we know that the Brownian motion $W_t$ in (\ref{timechange}) must hit $c_\ddagger C_*^{-1}M^{4-d}$ before time $C_\dagger\cref{const_typical_volume_3}M^{ 8-2d}$, which, by the reflection principle, occurs with probability $2\mathbb{P}\big(\mathbf{X} \ge \tfrac{c_\ddagger C_*^{-1}}{\sqrt{C_\dagger\cref{const_typical_volume_3}}} \big)$ (here $\mathbf{X}\sim N(0,1)$). Thus, by taking a sufficiently small $\cref{const_typical_volume_3}(d,\epsilon)>0$ such that $2\mathbb{P}\big(\mathbf{X} \ge \tfrac{c_\ddagger C_*^{-1}}{\sqrt{C_\dagger\cref{const_typical_volume_3}}} \big)\le e^{-\epsilon^{-1}}$, we have that on the event $\mathsf{H}^*_{\hat{M}}$, 
\begin{equation}
\begin{split}
	&\mathbb{P}\big(\langle \mathcal{M} \rangle_\infty \le C_\dagger\cref{const_typical_volume_3}M^{ 4}  ,\mathsf{H}^*_M  \mid \mathcal{F}_{\mathcal{C}_{\hat{M}}^{\ge 0}} \big) \\
	\overset{(\ref*{finaluse_565})}{\le } &\mathbb{P}\big(\langle \mathcal{M} \rangle_\infty \le C_\dagger\cref{const_typical_volume_3}M^{ 4}  ,\mathcal{M}_\infty - \mathcal{M}_0 \ge c_\ddagger C_*^{-1}M^{4-d}  \mid \mathcal{F}_{\mathcal{C}_{\hat{M}}^{\ge 0}} \big)\le e^{-\epsilon^{-1}}, 
	\end{split}
\end{equation}
which further implies (note that $\mathsf{H}^*_{\hat{M}} \subset \{\bm{0}\xleftrightarrow{\ge 0}\partial B(\hat{M})\}$)
\begin{equation}\label{finaluse567}
	\begin{split}
		& \mathbb{P}\big(\langle \mathcal{M} \rangle_\infty \le C_\dagger\cref{const_typical_volume_3}M^{ 4}  ,\mathsf{H}^*_M, \mathsf{H}^*_{\hat{M}} \big)  
		\le  e^{-\epsilon^{-1}} \mathbb{P}\big( \mathsf{H}^*_{\hat{M}}  \big)  \overset{(\ref*{one_arm_high})}{\lesssim   }  e^{-\epsilon^{-1}}\hat{M}^{-2}. 
			\end{split}
\end{equation}
Plugging (\ref{finaluse567}) into (\ref{finaluse562}), we obtain that the left-hand side of (\ref{finaluse5.60}) is at most 
\begin{equation}
	Ce^{-\epsilon^{-1}} C_*^3C_\dagger^2\cdot   \theta_d(N) < \tfrac{\epsilon}{4}\cdot  \theta_d(N) 
\end{equation}  
for all sufficiently small $\epsilon>0$, thereby establishing (\ref{finaluse5.60}).

By combining (\ref{final5.59}) and (\ref{finaluse5.60}), we get (\ref{finaluse5.56}) and thus confirm (\ref{newineq4.2}) for $d\ge 7$. In conclusion, we have completed the proof of Theorem \ref{thm_1.2}.  \qed

\subsection{Proof of Theorem \ref{thm_1.3}}\label{subsection_proof_thm1.3}

Generally speaking, in order for the critical cluster to achieve a large volume, our strategy is to force the GFF value at $\bm{0}$ to be large (still of the constant order, occurring with uniform probability). Technically, our proof mainly relies on the analysis of explored clusters introduced in Section \ref{section3.1_lower}, which enables us to apply the FKG inequality to derive a lower bound on the probability that the critical cluster has both large volume and large diameter.

Without loss of generality, we assume that $\lambda>1$ is sufficiently large and $M\ge e^{-\lambda^{10}}$. Before presenting the proof, we first record some notations as follows. 
\begin{itemize}

    \item  We denote the volume $\widecheck{\mathcal{V}}= \widecheck{\mathcal{V}}(M):= \mathrm{vol}(\mathcal{C}^{\ge 0}(\bm{0})\cap [B(M)\setminus B(\tfrac{M}{2})])$.

    \item  We denote the event $\widecheck{\mathsf{V}}=\widecheck{\mathsf{V}}(M,\lambda):= \{\widecheck{\mathcal{V}} \ge \lambda M^{(\frac{d}{2}+1)\boxdot 4} \}$.

	\item Let $m:=\frac{M}{10d}$. We say $\mathcal{C}^{\ge 0}_{m}$ is nice if $\mathbb{P}\big(\widecheck{\mathsf{V}}  \mid \mathcal{F}_{\mathcal{C}^{\ge 0}_{m}}\big)\ge e^{-\lambda^5}$.

\end{itemize}

The following properties of a nice $\mathcal{C}^{\ge 0}_{m}$ are crucial for this proof. 

\begin{lemma}\label{Lemma_nice_4.1}
For any $d\ge 3$ with $d\neq 6$, the following hold.
\begin{enumerate}
	\item $\mathbb{P}(\mathcal{C}^{\ge 0}_{m}\ \text{is nice})\gtrsim  e^{-c\lambda^4}\theta_d(M)$;

	\item If $\mathcal{C}^{\ge 0}_{m}$ is nice, then $\mathcal{H}^{\ge 0}_m \gtrsim     e^{-\lambda^5}M^{-[(\frac{d}{2}-1)\boxdot (d-4)]}$.

\end{enumerate}
	
\end{lemma}
\begin{proof}
	(1) On the one hand, by Lemma \ref{newlemma3.1} one has 
	\begin{equation}\label{use4.30}
		\begin{split}
			\mathbb{E}\big[ \widecheck{\mathcal{V}} \mid \widetilde{\phi}_{\bm{0}}\ge \lambda^2 \big] = & \sum\nolimits_{x\in B(M)\setminus B(\frac{M}{2})}  \mathbb{P}\big(x \xleftrightarrow{\ge 0} \bm{0}  \mid \widetilde{\phi}_{\bm{0}}\ge \lambda^2 \big)\\
			\gtrsim &\mathrm{vol}(B(M)\setminus B(\tfrac{M}{2}))\cdot M^{2-d}\lambda^2			\asymp  M^2\lambda^2. 
		\end{split}
	\end{equation}
	On the other hand, as in (\ref{finaluse522})), by applying the AM-QM inequality, we have 
	\begin{equation}
		\begin{split}
			\mathbf{Q}(B(M)\setminus B(\tfrac{M}{2}))\lesssim M\sum\nolimits_{\frac{M}{2}\le n\le M} \mathbf{Q}(\partial B(n)) \overset{\text{Lemma}\ \ref*{lemma_Q_BN}}{\lesssim }  M^{(\frac{d}{2}+3)\boxdot 6}.
		\end{split}
	\end{equation}
	Combined with the fact that $e^{-C\lambda^4}\le \mathbb{P}( \widetilde{\phi}_{\bm{0}}\ge \lambda^2)\le e^{-c\lambda^4}$ (see e.g., \cite[Proposition 2.1.2]{vershynin2018high}), it implies that 
	\begin{equation}\label{use4.31}
		\begin{split}
			\mathbb{E}\big[ \widecheck{\mathcal{V}}^2 \mid \widetilde{\phi}_{\bm{0}}\ge \lambda^2 \big] \le&  [ \mathbb{P}( \widetilde{\phi}_{\bm{0}}\ge \lambda^2)]^{-1} \mathbf{Q}(B(M)\setminus B(\tfrac{M}{2}))
			\lesssim e^{C\lambda^4}M^{(\frac{d}{2}+3)\boxdot 6}.
		\end{split}
	\end{equation}
	 Applying the Paley-Zygmund inequality and using (\ref{use4.30}) and (\ref{use4.31}), we get 
	\begin{equation}\label{use4.32}
	\begin{split}
		\mathbb{P}\big(\widecheck{\mathsf{V}} \big)\ge  &  \mathbb{P}\big(\widecheck{\mathsf{V}} \mid \widetilde{\phi}_{\bm{0}}\ge \lambda^2 \big)\cdot \mathbb{P}( \widetilde{\phi}_{\bm{0}}\ge \lambda^2) \\
		\gtrsim &  \frac{( M^2\lambda^2)^2}{e^{C\lambda^4}M^{(\frac{d}{2}+3)\boxdot 6}}\cdot e^{-c\lambda^4} \overset{(\ref*{one_arm_low}), (\ref*{one_arm_high})}{\gtrsim} e^{-C'\lambda^4}\theta_d(M). 
	\end{split}
		 	\end{equation}
		Meanwhile, since $\widecheck{\mathsf{V}} \subset \{\mathcal{H}^{\ge 0}_m>0 \}=\{\bm{0}\xleftrightarrow{\ge 0} \partial B(m) \}$, one has 
	\begin{equation}\label{use4.33}
	\begin{split}
		&\mathbb{P}\big(\widecheck{\mathsf{V}},\mathcal{C}^{\ge 0}_{m}\ \text{is not nice} \big) \\
		=&\mathbb{E}\big[ \mathbbm{1}_{\{\mathcal{C}^{\ge 0}_{m}\ \text{is not nice}\}\cap \{\mathcal{H}^{\ge 0}_m>0\}} \mathbb{P}\big(\widecheck{\mathsf{V}} \mid \mathcal{F}_{\mathcal{C}^{\ge 0}_{m}}\big)\big] \lesssim e^{-\lambda^5} \theta_d(M). 
	\end{split}	
	\end{equation}
	By (\ref{use4.32}) and (\ref{use4.33}), we obtain Item (1): 
	\begin{equation}
	\begin{split}
				\mathbb{P}(\mathcal{C}^{\ge 0}_{m}\ \text{is nice})\ge&  \mathbb{P}(\widecheck{\mathsf{V}}, \mathcal{C}^{\ge 0}_{m}\ \text{is nice})\\
				=&\mathbb{P}\big(\widecheck{\mathsf{V}} \big) - \mathbb{P}\big(\widecheck{\mathsf{V}},\mathcal{C}^{\ge 0}_{m}\ \text{is not nice} \big)  \ge e^{-C\lambda^4}\theta_d(M).
	\end{split}
	\end{equation}
		

	(2) On the event $\{\mathcal{C}^{\ge 0}_{m}\ \text{is nice}\}$, one has 
	\begin{equation}\label{use_4.35}
		\begin{split}
			\mathbb{E}\big[ \widecheck{\mathcal{V}} \mid \mathcal{F}_{\mathcal{C}^{\ge 0}_{m}}\big]  \ge \lambda M^{(\frac{d}{2}+1)\boxdot 4}  \mathbb{P}\big( \widecheck{\mathsf{V}} \mid \mathcal{F}_{\mathcal{C}^{\ge 0}_{m}}\big) \overset{ }{\ge}   e^{-\lambda^5} M^{(\frac{d}{2}+1)\boxdot 4}. 
		\end{split}
	\end{equation}
	Moreover, by Lemma \ref{newlemma3.1} we have 
	\begin{equation}
		\begin{split}
			\mathbb{E}\big[ \widecheck{\mathcal{V}} \mid \mathcal{F}_{\mathcal{C}^{\ge 0}_{m}}\big]  =&  \sum\nolimits_{x\in B(M)\setminus B(\frac{M}{2})} \mathbb{P}\big( x\xleftrightarrow{\ge 0} \bm{0}  \mid \mathcal{F}_{\mathcal{C}^{\ge 0}_{m}}\big)\asymp M^d\cdot (\mathcal{H}^{\ge 0}_m \land 1).
					\end{split}
	\end{equation}
	Combined with (\ref{use_4.35}), it yields that on $\{\mathcal{C}^{\ge 0}_{m}\ \text{is nice}\}$, 
	\begin{equation}
		\mathcal{H}^{\ge 0}_m \land 1\gtrsim  e^{-\lambda^5} M^{-[(\frac{d}{2}-1) \boxdot (d-4)]},
	\end{equation}
	thereby concluding the proof of Item (2) of this lemma.
\end{proof}

Recall $\mathsf{Q}_b(\cdot )$ in (\ref{newQb}), and assume that $\lambda$ is sufficiently large below. We define 
\begin{equation}
	\mathsf{G}:= \{\mathcal{C}^{\ge 0}_{m}\ \text{is nice} \}\cap  \mathsf{Q}_{e^{\lambda^5}}(m, \partial B(N)). 
\end{equation}
By Item (1) of Lemma \ref{Lemma_nice_4.1} and (\ref{upper_Qb}), we have  \begin{equation}\label{use4.39}
\mathbb{P}(\mathsf{G})\ge \mathbb{P}(\mathcal{C}^{\ge 0}_{m}\ \text{is nice} ) - \mathbb{P}\big(\big[\mathsf{Q}_{e^{\lambda^5}}(m, \partial B(N))\big]^c\big) \ge e^{-\lambda^{5} }\theta_d(M). 
\end{equation}
In addition, on the event $\mathsf{G}$, it follows from Item (2) of Lemma \ref{Lemma_nice_4.1} that 
\begin{equation}\label{use4.40}
	\begin{split}
	&	\mathbb{E}\big[ \mathrm{vol}(\mathcal{C}^{\ge 0}(\bm{0})\cap \partial B(N)) \mid \mathcal{F}_{\mathcal{C}^{\ge 0}_{m}}\big]\\
		\gtrsim  &\mathrm{vol}( \partial B(N))\cdot N^{2-d} M^{d-2}\cdot   e^{-\lambda^5}M^{-[(\frac{d}{2}-1)\boxdot (d-4)]}\asymp e^{-\lambda^5}NM^{(\frac{d}{2}-1)\boxdot 2}.
	\end{split}
\end{equation}
Thus, on $\mathsf{G}$, by the Paley-Zygmund inequality, one has 
\begin{equation}\label{use4.41}
	\begin{split}
		\mathbb{P}\big( \bm{0}\xleftrightarrow{\ge 0}\partial B(N) \mid \mathcal{F}_{\mathcal{C}^{\ge 0}_{m}} \big)\gtrsim \frac{(e^{-\lambda^5}NM^{(\frac{d}{2}-1)\boxdot 2})^2}{e^{\lambda^5}[\theta_d(M)]^{-1} \mathbf{Q}(\partial B(N))}  \overset{\text{Lemma}\ \ref*{lemma_Q_BN}}{\gtrsim }  e^{-3\lambda^5} \frac{\theta_d(N)}{\theta_d(M)},  
	\end{split}
\end{equation}
where in the first transition we used (\ref{use4.40}) and the fact that $\mathsf{G}\subset \mathsf{Q}_{e^{\lambda^5}}(m, \partial B(N))= \{ \mathcal{Q}^{\ge 0}_{m}(\partial B(N))\le e^{\lambda^5}[\theta_d(m)]^{-1}  \mathbf{Q}(\partial B(N))\}$ (recall (\ref{newQb})). Similar to (\ref{add3.37}), by applying the FKG inequality and using (\ref{use4.39}) and (\ref{use4.41}), we obtain
\begin{equation}\label{use4.42}
	\begin{split}
		&\mathbb{P}\big( \widecheck{\mathsf{V}}, \bm{0}\xleftrightarrow{\ge 0} \partial B(N)   \big)\\
		\ge &  \mathbb{E}\big[ \mathbbm{1}_{\mathsf{G}}\cdot \mathbb{P}\big( \widecheck{\mathsf{V}} \mid \mathcal{F}_{\mathcal{C}^{\ge 0}_{m}}\big)\cdot \mathbb{P}\big( \bm{0}\xleftrightarrow{\ge 0} \partial B(N)   \mid \mathcal{F}_{\mathcal{C}^{\ge 0}_{m}}\big)  \big]\\
		\gtrsim & e^{-\lambda^{5} }\theta_d(M)\cdot e^{-\lambda^5}\cdot e^{-3\lambda^5} \frac{\theta_d(N)}{\theta_d(M)} = e^{-5\lambda^5}\theta_d(N). 
	\end{split}
\end{equation}
Observing that $\widecheck{\mathsf{V}}\subset \{\mathcal{V}_{M}^{\ge 0}\ge \lambda M^{(\frac{d}{2}+1)\boxdot 4} \}$ and dividing both sides of (\ref{use4.42}) by $\theta_d(N)$, we conclude the proof of Theorem \ref{thm_1.3}.   \qed

\section*{Acknowledgments}

We would like to thank Romain Panis and Alexis Pr{\'e}vost for their help in improving the reference of this paper.

	\bibliographystyle{plain}
	\bibliography{ref}
	
\end{document}